\documentclass{amsart}

\usepackage{amsfonts,mathrsfs}
\usepackage{amsmath,amssymb,amsthm, amscd, bm}
\usepackage{tikz}
\usetikzlibrary{arrows,calc,matrix,topaths,positioning,scopes,shapes,decorations,decorations.markings} 
\usepackage{dsfont}
\usepackage{epsfig}
\usepackage{tikz-cd}
\usepackage{hyperref}
\usepackage{fullpage}
\usepackage[foot]{amsaddr}
\usepackage{adjustbox}

\makeatletter
\renewcommand{\email}[2][]{%
  \ifx\emails\@empty\relax\else{\g@addto@macro\emails{,\space}}\fi%
  \@ifnotempty{#1}{\g@addto@macro\emails{\textrm{(#1)}\space}}%
  \g@addto@macro\emails{#2}%
}
\makeatother

\author{Julien Korinman}
\address{Department of Mathematics, Faculty of Science and Engineering, Waseda University,3-4-1 Ohkubo, Shinjuku-ku, Tokyo, 169-8555, Japan}
\email{julien.korinman@gmail.com}

\subjclass{$57$R$56$, $57$N$10$, $57$M$25$.}


\newcommand{\quotient}[2]{{\raisebox{.2em}{$#1$}\left/\raisebox{-.2em}{$#2$}\right.}}

\newcommand{\sslash}{\mathbin{/\mkern-6mu/}}
\newcommand{\Hom}{\operatorname{Hom}}
\newcommand{\tr}{\operatorname{tr}}

\newcommand{\SL}{\operatorname{SL}}
\newcommand{\id}{id}

\newcommand{\heightexch}[3]{
	\begin{tikzpicture}[baseline=-0.4ex,scale=0.5, >=stealth]
	\draw [fill=gray!60,gray!45] (-.7,-.75)  rectangle (.4,.75)   ;
	\draw[#1] (0.4,-0.75) to (.4,.75);
	\draw[line width=1.2] (0.4,-0.3) to (-.7,-.3);
	\draw[line width=1.2] (0.4,0.3) to (-.7,.3);
	\draw (0.65,0.3) node {\scriptsize{$#2$}}; 
	\draw (0.65,-0.3) node {\scriptsize{$#3$}}; 
	\end{tikzpicture}
}
\newcommand{\heightcurve}{
\begin{tikzpicture}[baseline=-0.4ex,scale=0.5]
\draw [fill=gray!20,gray!45] (-.7,-.75)  rectangle (.4,.75)   ;
\draw[-] (0.4,-0.75) to (.4,.75);
\draw[line width=1.2] (-.7,-0.3) to (-.4,-.3);
\draw[line width=1.2] (-.7,0.3) to (-.4,.3);
\draw[line width=1.15] (-.4,0) ++(-90:.3) arc (-90:90:.3);
\end{tikzpicture}
}

\newcommand{\heightcurveright}{
\begin{tikzpicture}[baseline=-0.4ex,scale=0.5]
\draw [fill=gray!20,gray!45] (-.7,-.75)  rectangle (.4,.75)   ;
\draw[-] (-0.7,-0.75) to (-.7,.75);
\draw[line width=1.2] (0.1,-0.3) to (.4,-.3);
\draw[line width=1.2] (0.1,0.3) to (.4,.3);
\draw[line width=1.15] (.1,0) ++(90:.3) arc (90:270:.3);
\end{tikzpicture}
}

\newcommand{\heightexchright}[3]{
	\begin{tikzpicture}[baseline=-0.4ex,scale=0.5, >=stealth]
	\draw [fill=gray!60,gray!45] (-.7,-.75)  rectangle (.4,.75)   ;
	\draw[#1] (-0.7,-0.75) to (-0.7,.75);
	\draw[line width=1.2] (0.4,-0.3) to (-.7,-.3);
	\draw[line width=1.2] (0.4,0.3) to (-.7,.3);
	\draw (-1,0.3) node {\scriptsize{$#2$}}; 
	\draw (-1,-0.3) node {\scriptsize{$#3$}}; 
	\end{tikzpicture}
}

\newcommand{\bigonheightcurve}[2]{
\begin{tikzpicture}[baseline=-0.4ex,scale=0.5,>=stealth]
\draw [fill=gray!45,gray!45] (-.7,-.75)  rectangle (.4,.75)   ;
\draw[->] (0.4,-0.75) to (.4,.75);
\draw[->] (-0.7,-0.75) to (-.7,.75);
\draw[line width=1.2] (0.4,-0.3) to (0,-.3);
\draw[line width=1.2] (0.4,0.3) to (0,.3);
\draw[line width=1.1] (0,0) ++(90:.3) arc (90:270:.3);
\draw (0.65,0.3) node {\scriptsize{$#1$}}; 
\draw (0.65,-0.3) node {\scriptsize{$#2$}}; 
\end{tikzpicture}
}

\newcommand{\bigonheightcurveright}[2]{
\begin{tikzpicture}[baseline=-0.4ex,scale=0.5,>=stealth]
\draw [fill=gray!45,gray!45] (-.7,-.75)  rectangle (.4,.75)   ;
\draw[->] (0.4,-0.75) to (.4,.75);
\draw[->] (-0.7,-0.75) to (-.7,.75);
\draw[line width=1.2] (-.7,-0.3) to (-.4,-.3);
\draw[line width=1.2] (-.7,0.3) to (-.4,.3);
\draw[line width=1.15] (-.4,0) ++(-90:.3) arc (-90:90:.3);
\draw (-1,0.3) node {\scriptsize{$#1$}}; 
\draw (-1,-0.3) node {\scriptsize{$#2$}}; 
\end{tikzpicture}
}

\begin{document}

\theoremstyle{plain}
\newtheorem{theorem}{Theorem}[section]
\newtheorem{main_theorem}[theorem]{Main Theorem}
\newtheorem{proposition}[theorem]{Proposition}
\newtheorem{corollary}[theorem]{Corollary}
\newtheorem{corollaire}[theorem]{Corollaire}
\newtheorem{lemma}[theorem]{Lemma}
\theoremstyle{definition}
\newtheorem{notations}[theorem]{Notations}
\newtheorem{convention}[theorem]{Convention}
\newtheorem{problem}[theorem]{Problem}
\newtheorem*{notations*}{Notations}
\newtheorem{definition}[theorem]{Definition}
\newtheorem{Theorem-Definition}[theorem]{Theorem-Definition}
\theoremstyle{remark}
\newtheorem{remark}[theorem]{Remark}
\newtheorem{conjecture}[theorem]{Conjecture}
\newtheorem{example}[theorem]{Example}

\title[Finite presentations for stated skein algebras and LGFT]{Finite presentations for stated skein algebras and lattice gauge field theory}
%
%
%

\date{}
\maketitle


\begin{abstract} 
We provide finite presentations for stated skein algebras and deduce that those algebras are Koszul and that they are isomorphic to the quantum moduli algebras appearing in lattice gauge field theory, generalizing  previous results of Bullock, Frohman, Kania-Bartoszynska and Faitg.
\end{abstract}


\section{Introduction}

\paragraph{\textbf{Stated skein algebras and lattice gauge field theory}}

\par  A \textit{punctured surface} is a pair $\mathbf{\Sigma}=(\Sigma,\mathcal{P})$, where $\Sigma$ is a compact oriented surface and $\mathcal{P}$ is a (possibly empty) finite subset of $\Sigma$ which intersects non-trivially each boundary component.  We write $\Sigma_{\mathcal{P}}:= \Sigma \setminus \mathcal{P}$.
  The set $\partial \Sigma\setminus \mathcal{P}$ consists of a disjoint union of open arcs which we call \textit{boundary arcs}.
   \\ \textbf{Warning:} In this paper, the punctured surface $\mathbf{\Sigma}$ will be called open if the surface $\Sigma$ has non empty boundary and closed otherwise. This convention differs from the traditional one, where some authors refer to open surface a punctured surface $\mathbf{\Sigma}=(\Sigma, \mathcal{P})$ with $\Sigma$ closed and $\mathcal{P}\neq \emptyset$ (in which case $\Sigma_{\mathcal{P}}$ is not closed).

\vspace{2mm}
\par The \textit{Kauffman-bracket skein algebras} were introduced by Bullock and Turaev as a tool to study the $\mathrm{SU}(2)$ Witten-Reshetikhin-Turaev topological quantum field theories (\cite{Wi2, RT}). They are associative unitary algebras $\mathcal{S}_{\omega}(\mathbf{\Sigma})$ indexed by a closed punctured surface $\mathbf{\Sigma}$ and an invertible element $\omega \in \mathds{k}^{\times}$ in some commutative unital ring $\mathds{k}$. Bonahon-Wong  \cite{BonahonWongqTrace} and  L\^e \cite{LeStatedSkein} generalized  the notion of Kauffman-bracket skein algebras to open punctured surfaces, where in addition to closed curves the algebras are generated by arcs whose endpoints are endowed with a sign $\pm$ (a state). The motivation for the introduction of these so-called \textit{stated skein algebras} is their good behaviour for the operation of gluing two boundary arcs together. This property permitted the authors of \cite{BonahonWongqTrace} to define an embedding of the skein algebra into a quantum torus, named the quantum trace,  and offers new tools to study the representation theory of skein algebras. 
\vspace{2mm}
\par Except for genus $0$ and $1$ surfaces (\cite{BullockPrzytycki_00}), no finite presentation for the Kauffman-bracket skein algebras is known, though a conjecture in that direction was formulated in \cite[Conjecture $1.2$]{SantharoubaneSkeinGenerators}. However, it is well-known that they are finitely generated (\cite{BullockGeneratorsSkein, AbdielFrohman_SkeinFrobenius, FrohmanKania_SkeinRootUnity, SantharoubaneSkeinGenerators}). The corresponding problem for stated skein algebras of open punctured surfaces is easier. Finite  presentations of stated skein algebras were given for a disc with two punctures on its boundary (i.e. for the bigon) and for the disc with three punctures on its boundary (i.e. for the triangle) in \cite{LeStatedSkein}, for the disc with two punctures on its boundary and one inner puncture in \cite{KojuQGroupsBraidings} and for any connected punctured surface having exactly one boundary component, one puncture on the boundary and eventual inner punctures, in \cite{Faitg_LGFT_SSkein}. 
\vspace{2mm}
\par
The first purpose of this paper is to provide explicit finite presentations for stated skein algebras of an arbitrary connected open punctured surface $\mathbf{\Sigma}$. Let us briefly sketch their construction, we refer to Section  \ref{sec_presentation} for details.

The finite presentations we will define depend on the choice of a finite presentation $\mathbb{P}$ of some groupoid $\Pi_1(\Sigma_{\mathcal{P}}, \mathbb{V})$. In brief, for each boundary arc $a$ of $\mathbf{\Sigma}$, choose a point $v_a \in a$ and let $\mathbb{V}$ be the set of such points. The groupoid $\Pi_1(\Sigma_{\mathcal{P}}, \mathbb{V})$ is the full subcategory of the fundamental groupoid of $\Sigma_{\mathcal{P}}$ whose set of objects is $\mathbb{V}$. A finite presentation $\mathbb{P}=(\mathbb{G}, \mathbb{RL})$ for $\Pi_1(\Sigma_{\mathcal{P}}, \mathbb{V})$ will consist in a finite set $\mathbb{G}$ of generating paths relating points of $\mathbb{V}$ and a finite set $\mathbb{RL}$ of relations among those paths which satisfy some axioms (see Section \ref{sec_presentation} for details). For instance for the triangle $\mathbb{T}$ (the disc with three punctures on its boundary), the groupoid $\Pi_1(\mathbb{T}, \mathbb{V})$ admits the presentation  with generators $\mathbb{G}= \{ \alpha, \beta, \gamma \}$ drawn in Figure \ref{fig_triangle} and the unique relation $\alpha \beta \gamma = 1$.

\begin{figure}[!h] 
\centerline{\includegraphics[width=2cm]{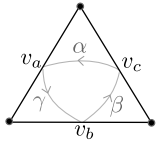} }
\caption{The triangle and some paths.} 
\label{fig_triangle} 
\end{figure}

A path $\alpha \in \mathbb{G}$ can be seen as an arc in $\Sigma_{\mathcal{P}}$ and, after choosing some states $\varepsilon , \varepsilon' \in \{ -, +\}$ for its endpoints, we get an element $\alpha_{\varepsilon \varepsilon'} \in \mathcal{S}_{\omega}(\mathbf{\Sigma})$ in the stated skein algebra. We denote by $\mathcal{A}^{\mathbb{G}} \subset  \mathcal{S}_{\omega}(\mathbf{\Sigma})$ the (finite) set of such elements. It was proved in \cite{KojuAzumayaSkein} that $\mathcal{A}^{\mathbb{G}}$ generates $\mathcal{S}_{\omega}(\mathbf{\Sigma})$ and its elements will be the generators of our presentations.

Concerning the relations, first for each $\alpha \in \mathbb{G}$, one has a  \textit{q-determinant} relation between the elements $\alpha_{\varepsilon \varepsilon'}$. 
For each pair $(\alpha, \beta) \in \mathbb{G}^2$ we will associate a finite set of  \textit{arcs exchange relations} permitting to express an element of the form $\alpha_{\varepsilon \varepsilon'} \beta_{\mu \mu'} \in \mathcal{S}_{\omega}(\mathbf{\Sigma})$ as a linear combination of elements of the form $\beta_{a b} \alpha_{c d}$. Eventually, to each relation $R\in \mathbb{RL}$ in the finite presentation $\mathbb{P}$, we will associate a finite set of so-called \textit{trivial loops relations}. 

\begin{theorem}\label{theorem1}
Let $\mathbf{\Sigma}$ a connected open punctured surface and  $\mathbb{P}$  a finite presentation of $\Pi_1(\Sigma_{\mathcal{P}}, \mathbb{V})$. Then the stated skein algebra $\mathcal{S}_{\omega}(\mathbf{\Sigma})$ is presented by the set of generators $\mathcal{A}^{\mathbb{G}}$ and by the q-determinant, arcs exchange and trivial loops relations.
\end{theorem}

For every open punctured surface, we can choose finite presentation $\mathbb{P}$ of $\Pi_1(\Sigma_{\mathcal{P}}, \mathbb{V})$ such that the set of relations is empty (for instance for the triangle of Figure \ref{fig_triangle}, one might choose the presentation with generators $\mathbb{G}= \{ \alpha, \beta \}$ and no relation). In this case, the presentation of $\mathcal{S}_{\omega}(\mathbf{\Sigma})$  is quadratic inhomogeneous and, by using the Diamond Lemma, we prove the

\begin{theorem}\label{theorem2}
For  $\mathbf{\Sigma}$ a connected open punctured surface, the quadratic inhomogeneous algebra $\mathcal{S}_{\omega}(\mathbf{\Sigma})$ is Koszul and admits a Poincar\'e-Birkhoff-Witt basis.
\end{theorem}

 Theorem \ref{theorem2} implies that $\mathcal{S}_{\omega}(\mathbf{\Sigma})$ has an explicit minimal projective resolution (the so-called Koszul resolution) which permits to compute effectively its cohomology (see \cite{LodayValletteOperads} for details).

\vspace{2mm}
 \par Let $(\Gamma,c)$ be a ciliated graph, that is a finite graph with the data for each vertex of a linear ordering of its adjacent half-edges. Inspired by Fock and Rosly's original work in \cite{FockRosly} on the Poisson structure of character varieties, Alekseev-Grosse-Schomerus (\cite{AlekseevGrosseSchomerus_LatticeCS1,AlekseevGrosseSchomerus_LatticeCS2, AlekseevSchomerus_RepCS}) and Buffenoir-Roche (\cite{BuffenoirRoche, BuffenoirRoche2}) independently defined the so-called \textit{quantum moduli algebras} $\mathcal{L}_{\omega}(\Gamma, c)$, which are combinatorial quantizations of relative character varieties (see Section \ref{sec_charvar} for details). Those algebras arise with some right comodule map $\Delta^{\mathcal{G}} : \mathcal{L}_{\omega}(\Gamma, c) \rightarrow \mathcal{L}_{\omega}(\Gamma, c) \otimes \mathcal{O}_q[\mathcal{G}]$, where $\mathcal{O}_q[\mathcal{G}]=\mathcal{O}_q[\SL_2]^{\otimes \mathring{V}(\Gamma)}$ is the so-called quantum gauge group Hopf algebra. The subalgebra $\mathcal{L}^{inv}_{\omega}(\Gamma) \subset \mathcal{L}_{\omega}(\Gamma,c)$ of coinvariant vectors plays an important role in combinatorial quantization. 
  More precisely, as reviewed in Section \ref{sec_graphs}, we associate to each ciliated graph $(\Gamma,c)$ two punctured surfaces: an open one $\mathbf{\Sigma}^0(\Gamma,c)$ and a closed one $\mathbf{\Sigma}(\Gamma)$ such that the algebras $\mathcal{L}_{\omega}(\Gamma, c) $ and $\mathcal{L}^{inv}_{\omega}(\Gamma)$ are quantization of the $\SL_2(\mathbb{C})$ (relative) character varieties of $\mathbf{\Sigma}^0(\Gamma,c)$ and  $\mathbf{\Sigma}(\Gamma)$ respectively with their Fock-Rosly Poisson structures. 
  We deduce from Theorem \ref{theorem1} the following: 
  
  \begin{theorem}\label{theorem3}
  There exist isomorphisms of algebras $\mathcal{S}_{\omega}(\mathbf{\Sigma}^0(\Gamma,c)) \cong \mathcal{L}_{\omega}(\Gamma,c)$ and $\mathcal{S}_{\omega}(\mathbf{\Sigma}(\Gamma)) \cong \mathcal{L}^{inv}_{\omega}(\Gamma)$.
  \end{theorem}
  
  Theorem \ref{theorem3} is not surprising and was already proved in some cases. First it is well-known that (stated) skein algebras also induces deformation quantizations of (relative) character varieties : it follows from the work in \cite{Bullock, PS00, Turaev91} for closed punctured surfaces and is proved in 
  \cite[Theorem $1.3$]{KojuQuesneyClassicalShadows} and \cite[Theorem $8.12$]{CostantinoLe19} for open punctured surfaces. So Theorem \ref{theorem3} was expected; for instance its statement was  conjectured by Costantino and L\^e in \cite{CostantinoLe19}. 
  Next the skein origin of the defining relations of quantum moduli algebra was discovered by Bullock, Frohman and Kania-Bartoszynska in  \cite{BullockFrohmanKania_LGFT} where the authors already proved that $\mathcal{S}_{\omega}(\mathbf{\Sigma}(\Gamma))$ and $\mathcal{L}^{inv}_{\omega}(\Gamma)$ are isomorphic in the particular case where $\mathds{k}=\mathbb{C}[[\hbar]]$ and $q:=\omega^{-4}=\exp{\hbar}$. However, their proof does not extend straightforwardly to arbitrary ring (see item $(6)$ of Section \ref{sec_final}).  Eventually, in the special case where $(\Gamma,c)$ is the so-called daisy graph (it has only one vertex, so $\mathbf{\Sigma}^0(\Gamma,c)$ has exactly one boundary component with one puncture on it), Theorem \ref{theorem3} was proved by Faitg in \cite{Faitg_LGFT_SSkein} in the case where $\omega$ is not a root of unity. A detailed comparison between Faitg's isomorphism and ours is made in Section \ref{sec_comparaison}.  Faitg's result can also be derived indirectly from the works in \cite{BenzviBrochierJordan_FactAlg1, GunninghamJordanSafranov_FinitenessConjecture}, as detailed in Section \ref{sec_comparaison}, though the so-obtained isomorphism is less explicit due to a change of duality.

\vspace{2mm}
\paragraph{\textbf{Acknowledgments.}} The author thanks  S.Baseilhac, F.Costantino, M.Faitg, L.Funar , A.Quesney and P.Roche for useful discussions. He acknowledges  support from the Japanese Society for Promotion of Science (JSPS) and the Centre National pour la Recherche et l'Enseignement (CNRS). 


\section{Finite presentations for stated skein algebras}
\subsection{Definitions and first properties of stated skein algebras}\label{sec_def_basic}

\begin{definition}
A \textit{punctured surface} is a pair $\mathbf{\Sigma}=(\Sigma,\mathcal{P})$ where $\Sigma$ is a compact oriented surface and $\mathcal{P}$ is a finite subset of $\Sigma$ which intersects non-trivially each boundary component. A \textit{boundary arc} is a connected component of $\partial \Sigma \setminus \mathcal{P}$. We write $\Sigma_{\mathcal{P}}:= \Sigma \setminus \mathcal{P}$.
\end{definition}

\textbf{Definition of stated skein algebras}
\\ Before stating precisely the definition of stated skein algebras, let us sketch it informally. Given a punctured surface $\mathbf{\Sigma}$ 	and an invertible element $\omega \in \mathds{k}^{\times}$ in some commutative unital ring $\mathds{k}$, the stated skein algebra $\mathcal{S}_{\omega}(\mathbf{\Sigma})$ is the quotient of the $\mathds{k}$-module freely spanned by isotopy classes of stated tangles in $\Sigma_{\mathcal{P}}\times (0,1)$  by some local skein relations. The left part of Figure \ref{fig_statedtangle} illustrates such a stated tangle: each point of $\partial T \subset \partial \Sigma_{\mathcal{P}}$ is equipped with a sign $+$ or $-$ (the state). Here the stated tangle is the union of three stated arcs and one closed curve. In order to work with two-dimensional pictures, we will consider the projection of tangles in $\Sigma_{\mathcal{P}}$ as in the right part of Figure \ref{fig_statedtangle}; such a projection will be referred to as a diagram.

\begin{figure}[!h] 
\centerline{\includegraphics[width=6cm]{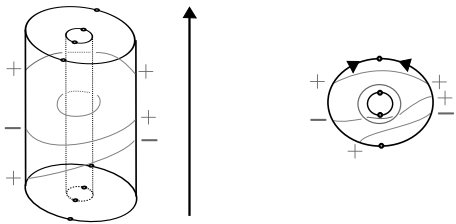} }
\caption{On the left: a stated tangle. On the right: its associated diagram. The arrows represent the height orders. } 
\label{fig_statedtangle} 
\end{figure} 

\par   A \textit{tangle} in $ \Sigma_{\mathcal{P}}\times (0,1)$   is a  compact framed, properly embedded $1$-dimensional manifold $T\subset \Sigma_{\mathcal{P}}\times (0,1)$ such that for every point of $\partial T \subset \partial \Sigma_{\mathcal{P}}\times (0,1)$ the framing is parallel to the $(0,1)$ factor and points to the direction of $1$.  Here, by framing, we refer to a thickening of $T$ to an oriented surface. The \textit{height} of $(v,h)\in \Sigma_{\mathcal{P}}\times (0,1)$ is $h$.  If $b$ is a boundary arc and $T$ a tangle, we impose that no two points in $\partial_bT:= \partial T \cap b\times(0,1)$  have the same heights, hence the set $\partial_bT$ is totally ordered by the heights. Two tangles are isotopic if they are isotopic through the class of tangles that preserve the boundary height orders. By convention, the empty set is a tangle only isotopic to itself.
 
\vspace{2mm}
\par Let $\pi : \Sigma_{\mathcal{P}}\times (0,1)\rightarrow \Sigma_{\mathcal{P}}$ be the projection with $\pi(v,h)=v$. A tangle $T$ is in \textit{generic position} if for each of its points, the framing is parallel to the $(0,1)$ factor and points in the direction of $1$ and is such that $\pi_{\big| T} : T\rightarrow \Sigma_{\mathcal{P}}$ is an immersion with at most transversal double points in the interior of $\Sigma_{\mathcal{P}}$. Every tangle is isotopic to a tangle in generic position. We call \textit{diagram}  the image $D=\pi(T)$ of a tangle in generic position, together with the over/undercrossing information at each double point. An isotopy class of diagram $D$ together with a total order of $\partial_b D:=\partial D\cap b$ for each boundary arc $b$, define uniquely an isotopy class of tangle. When choosing an orientation $\mathfrak{o}(b)$ of a boundary arc $b$ and a diagram $D$, the set $\partial_bD$ receives a natural order by setting that the points are increasing when going in the direction of $\mathfrak{o}(b)$. We will represent tangles by drawing a diagram and an orientation (an arrow) for each boundary arc, as in Figure \ref{fig_statedtangle}. When a boundary arc $b$ is oriented we assume that $\partial_b D$ is ordered according to the orientation. A \textit{state} of a tangle is a map $s:\partial T \rightarrow \{-, +\}$. A pair $(T,s)$ is called a \textit{stated tangle}. We define a \textit{stated diagram} $(D,s)$ in a similar manner.

 \vspace{2mm}
\par  Let  $\omega\in \mathds{k}^{\times}$ an invertible element and write $A:=\omega^{-2}$.

\begin{definition}\label{def_stated_skein}\cite{LeStatedSkein} 
  The \textit{stated skein algebra}  $\mathcal{S}_{\omega}(\mathbf{\Sigma})$ is the  free $\mathds{k}$-module generated by isotopy classes of stated tangles in $\Sigma_{\mathcal{P}}\times (0, 1)$ modulo the following relations \eqref{eq: skein 1} and \eqref{eq: skein 2}, 
  	\begin{equation}\label{eq: skein 1} 
\begin{tikzpicture}[baseline=-0.4ex,scale=0.5,>=stealth]	
\draw [fill=gray!45,gray!45] (-.6,-.6)  rectangle (.6,.6)   ;
\draw[line width=1.2,-] (-0.4,-0.52) -- (.4,.53);
\draw[line width=1.2,-] (0.4,-0.52) -- (0.1,-0.12);
\draw[line width=1.2,-] (-0.1,0.12) -- (-.4,.53);
\end{tikzpicture}
=A
\begin{tikzpicture}[baseline=-0.4ex,scale=0.5,>=stealth] 
\draw [fill=gray!45,gray!45] (-.6,-.6)  rectangle (.6,.6)   ;
\draw[line width=1.2] (-0.4,-0.52) ..controls +(.3,.5).. (-.4,.53);
\draw[line width=1.2] (0.4,-0.52) ..controls +(-.3,.5).. (.4,.53);
\end{tikzpicture}
+A^{-1}
\begin{tikzpicture}[baseline=-0.4ex,scale=0.5,rotate=90]	
\draw [fill=gray!45,gray!45] (-.6,-.6)  rectangle (.6,.6)   ;
\draw[line width=1.2] (-0.4,-0.52) ..controls +(.3,.5).. (-.4,.53);
\draw[line width=1.2] (0.4,-0.52) ..controls +(-.3,.5).. (.4,.53);
\end{tikzpicture}
\hspace{.5cm}
\text{ and }\hspace{.5cm}
\begin{tikzpicture}[baseline=-0.4ex,scale=0.5,rotate=90] 
\draw [fill=gray!45,gray!45] (-.6,-.6)  rectangle (.6,.6)   ;
\draw[line width=1.2,black] (0,0)  circle (.4)   ;
\end{tikzpicture}
= -(A^2+A^{-2}) 
\begin{tikzpicture}[baseline=-0.4ex,scale=0.5,rotate=90] 
\draw [fill=gray!45,gray!45] (-.6,-.6)  rectangle (.6,.6)   ;
\end{tikzpicture}
;
\end{equation}

\begin{equation}\label{eq: skein 2} 
\begin{tikzpicture}[baseline=-0.4ex,scale=0.5,>=stealth]
\draw [fill=gray!45,gray!45] (-.7,-.75)  rectangle (.4,.75)   ;
\draw[->] (0.4,-0.75) to (.4,.75);
\draw[line width=1.2] (0.4,-0.3) to (0,-.3);
\draw[line width=1.2] (0.4,0.3) to (0,.3);
\draw[line width=1.1] (0,0) ++(90:.3) arc (90:270:.3);
\draw (0.65,0.3) node {\scriptsize{$+$}}; 
\draw (0.65,-0.3) node {\scriptsize{$+$}}; 
\end{tikzpicture}
=
\begin{tikzpicture}[baseline=-0.4ex,scale=0.5,>=stealth]
\draw [fill=gray!45,gray!45] (-.7,-.75)  rectangle (.4,.75)   ;
\draw[->] (0.4,-0.75) to (.4,.75);
\draw[line width=1.2] (0.4,-0.3) to (0,-.3);
\draw[line width=1.2] (0.4,0.3) to (0,.3);
\draw[line width=1.1] (0,0) ++(90:.3) arc (90:270:.3);
\draw (0.65,0.3) node {\scriptsize{$-$}}; 
\draw (0.65,-0.3) node {\scriptsize{$-$}}; 
\end{tikzpicture}
=0,
\hspace{.2cm}
\begin{tikzpicture}[baseline=-0.4ex,scale=0.5,>=stealth]
\draw [fill=gray!45,gray!45] (-.7,-.75)  rectangle (.4,.75)   ;
\draw[->] (0.4,-0.75) to (.4,.75);
\draw[line width=1.2] (0.4,-0.3) to (0,-.3);
\draw[line width=1.2] (0.4,0.3) to (0,.3);
\draw[line width=1.1] (0,0) ++(90:.3) arc (90:270:.3);
\draw (0.65,0.3) node {\scriptsize{$+$}}; 
\draw (0.65,-0.3) node {\scriptsize{$-$}}; 
\end{tikzpicture}
=\omega
\begin{tikzpicture}[baseline=-0.4ex,scale=0.5,>=stealth]
\draw [fill=gray!45,gray!45] (-.7,-.75)  rectangle (.4,.75)   ;
\draw[-] (0.4,-0.75) to (.4,.75);
\end{tikzpicture}
\hspace{.1cm} \text{ and }
\hspace{.1cm}
\omega^{-1}
\heightexch{->}{-}{+}
- \omega^{-5}
\heightexch{->}{+}{-}
=
\heightcurve.
\end{equation}
The product of two classes of stated tangles $[T_1,s_1]$ and $[T_2,s_2]$ is defined by  isotoping $T_1$ and $T_2$  in $\Sigma_{\mathcal{P}}\times (1/2, 1) $ and $\Sigma_{\mathcal{P}}\times (0, 1/2)$ respectively and then setting $[T_1,s_1]\cdot [T_2,s_2]=[T_1\cup T_2, s_1\cup s_2]$. Figure \ref{fig_product} illustrates this product.
\end{definition}
\par For a closed punctured surface, $\mathcal{S}_{\omega}(\mathbf{\Sigma})$ coincides with the classical (Turaev's) Kauffman-bracket skein algebra.

\begin{figure}[!h] 
\centerline{\includegraphics[width=8cm]{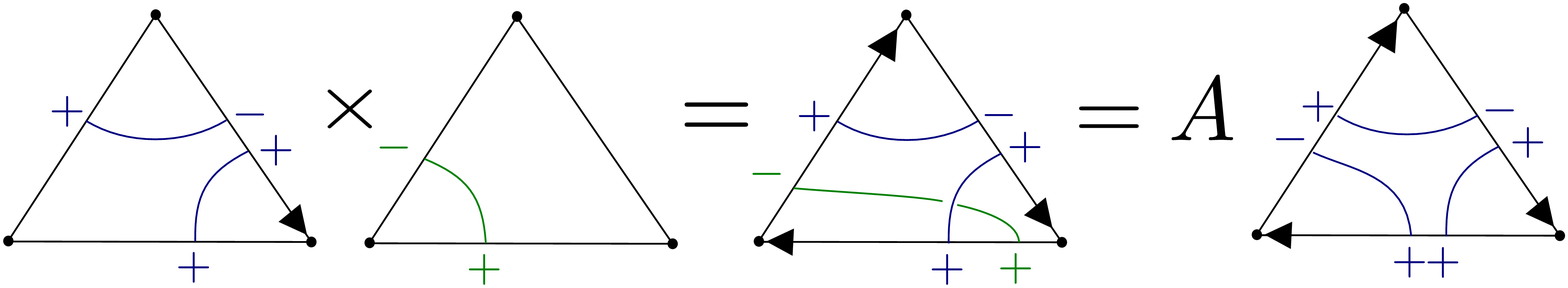} }
\caption{An illustration of the product in stated skein algebras.} 
\label{fig_product} 
\end{figure}

\textbf{Reflexion anti-involution}
\\ Suppose that $\mathds{k} = \mathbb{Z}[\omega^{\pm 1}]$ and consider the $\mathbb{Z}$-linear involution $x\mapsto x^*$ on $\mathds{k}$ sending $\omega$ to $\omega^{-1}$.
Let $r: \Sigma_{\mathcal{P}} \times (0,1) \xrightarrow{\cong} \Sigma_{\mathcal{P}}$ be the homeomorphism defined by $r(x,t) = (x, 1-t)$. Define an anti-linear map  $\theta : \mathcal{S}_{\omega}(\mathbf{\Sigma}) \xrightarrow{\cong} \mathcal{S}_{\omega}(\mathbf{\Sigma}) $ by setting 
$$\theta \left( \sum_i x_i [T_i, s_i] \right) := \sum_i x_i^* [r(T_i), s_i\circ r].$$

\begin{proposition} (\cite[Proposition $2.7$]{LeStatedSkein})
$\theta$ is an anti-morphism of algebras, i.e. $\theta(xy)= \theta(y)\theta(x)$.
\end{proposition}

\vspace{2mm}
\par 
\textbf{Bases for stated skein algebras}
\\ A closed component of a diagram $D$ is trivial if it bounds an embedded disc in $\Sigma_{\mathcal{P}}$. An open component of $D$ is trivial if it can be isotoped, relatively to its boundary, inside some boundary arc. A diagram is \textit{simple} if it has neither double point nor trivial component. By convention, the empty set is a simple diagram. Let $\mathfrak{o}$ denote an arbitrary orientation of the boundary arcs of $\mathbf{\Sigma}$. For each boundary arc $b$ we write $<_{\mathfrak{o}}$ the induced total order on $\partial_b D$. A state $s: \partial D \rightarrow \{ - , + \}$ is $\mathfrak{o}-$\textit{increasing} if for any boundary arc $b$ and any two points $x,y \in \partial_b D$, then $x<_{\mathfrak{o}} y$ implies $s(x)< s(y)$, with the convention $- < +$. 

\begin{definition}\label{def_basis}
 We denote by $\mathcal{B}^{\mathfrak{o}}\subset \mathcal{S}_{\omega}(\mathbf{\Sigma})$ the set of classes of stated diagrams $(D,s)$ such that $D$ is simple and $s$ is $\mathfrak{o}$-increasing. 
\end{definition}

\begin{theorem}\label{theorem_basis} (\cite[Theorem $2.11$]{LeStatedSkein})  the set $\mathcal{B}^{\mathfrak{o}}$ is a  basis of $\mathcal{S}_{\omega}(\mathbf{\Sigma})$. \end{theorem}

\begin{remark}\label{remark_change_basis}
The basis $\mathcal{B}^{\mathfrak{o}}$ is independent on the choice of the ground ring $\mathds{k}$ and of $q\in \mathds{k}^{\times}$. This fact has the following useful consequence: let $\mathds{k} := \mathbb{Z}[\omega^{\pm 1}]$ and $\mathds{k}'$ be any other commutative unital ring with an invertible element $\omega' \in \mathds{k}^{'\times}$. There is a unique morphism of rings $\mu : \mathds{k} \rightarrow \mathds{k}'$ sending $\omega$ to $\omega'$ and the two $\mathds{k}'$ algebras $\mathcal{S}_{\omega}(\mathbf{\Sigma})\otimes_{\mathds{k}} \mathds{k}'$ and $\mathcal{S}_{\omega'}(\mathbf{\Sigma})$ are canonically isomorphic through the isomorphism preserving the basis $\mathcal{B}^{\mathfrak{o}}$. This fact permits to prove formulas in $\mathds{k}$ using the reflexion anti-involution $\theta$ and then apply them to any ring $\mathds{k}'$ by changing the coefficients.

\end{remark}

\vspace{2mm}
\par 
\textbf{Gluing maps}
\\ Let $a$, $b$ be two distinct boundary arcs of $\mathbf{\Sigma}$ and let $\mathbf{\Sigma}_{|a\#b}$ be the punctured surface obtained from $\mathbf{\Sigma}$ by gluing $a$ and $b$. Denote by $\pi : \Sigma_{\mathcal{P}} \rightarrow (\Sigma_{|a\#b})_{\mathcal{P}_{|a\#b}}$ the projection and $c:=\pi(a)=\pi(b)$. Let $(T_0, s_0)$ be a stated framed tangle of ${\Sigma_{|a\#b}}_{\mathcal{P}_{|a\#b}} \times (0,1)$ transversed to $c\times (0,1)$ and such that the heights of the points of $T_0 \cap c\times (0,1)$ are pairwise distinct and the framing of the points of $T_0 \cap c\times (0,1)$ is vertical. Let $T\subset \Sigma_{\mathcal{P}}\times (0,1)$ be the framed tangle obtained by cutting $T_0$ along $c$. 
Any two states $s_a : \partial_a T \rightarrow \{-,+\}$ and $s_b : \partial_b T \rightarrow \{-,+\}$ give rise to a state $(s_a, s, s_b)$ on $T$. 
Both the sets $\partial_a T$ and $\partial_b T$ are in canonical bijection with the set $T_0\cap c$ by the map $\pi$. Hence the two sets of states $s_a$ and $s_b$ are both in canonical bijection with the set $\mathrm{St}(c):=\{ s: c\cap T_0 \rightarrow \{-,+\} \}$. 

\begin{definition}\label{def_gluing_map}
Let $i_{|a\#b}: \mathcal{S}_{\omega}(\mathbf{\Sigma}_{|a\#b}) \rightarrow \mathcal{S}_{\omega}(\mathbf{\Sigma})$ be the linear map given, for any $(T_0, s_0)$ as above, by: 
$$ i_{|a\#b} \left( [T_0,s_0] \right) := \sum_{s \in \mathrm{St}(c)} [T, (s, s_0 , s) ].$$
\end{definition}

\begin{figure}[!h] 
\centerline{\includegraphics[width=8cm]{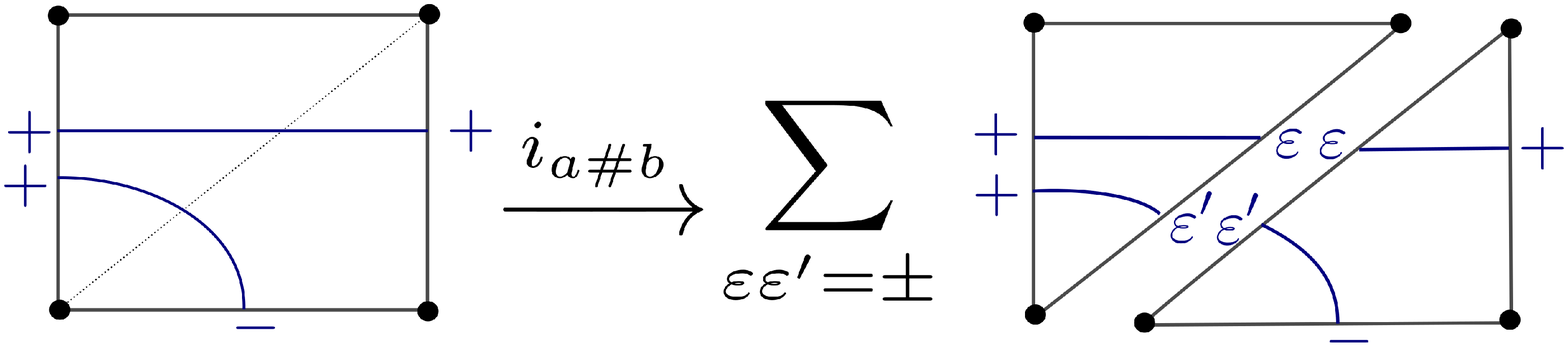} }
\caption{An illustration of the gluing map $i_{a\#b}$.} 
\label{fig_gluingmap} 
\end{figure} 

\begin{theorem}\label{theorem_gluing}\cite[Theorem $3.1$]{LeStatedSkein}
 The linear map $i_{|a\#b}: \mathcal{S}_{\omega}(\mathbf{\Sigma}_{|a\#b}) \rightarrow \mathcal{S}_{\omega}(\mathbf{\Sigma})$ is an injective morphism of algebras. Moreover the gluing operation is coassociative in the sense that if $a,b,c,d$ are four distinct boundary arcs, then we have $i_{|a\#b} \circ i_{|c\#d} = i_{|c\#d} \circ i_{|a\#b}$.
\end{theorem}

\textbf{Relation with $U_q\mathfrak{sl}_2$ and its restricted dual $\mathcal{O}_q[\mathrm{SL}_2]$}
\\ Recall that $A=\omega^{-2}$ and write $q:=A^2$. The stated skein algebra has deep relations with the quantum group $U_q\mathfrak{sl}_2$ and its restricted dual $\mathcal{O}_q(\mathrm{SL}_2)$, explored in \cite{LeStatedSkein, KojuQuesneyClassicalShadows, CostantinoLe19, Faitg_LGFT_SSkein} that we briefly reproduce here for later use. Let $\rho : U_q \mathfrak{sl}_2 \rightarrow \mathrm{End}(V)$ be the standard representation of $U_q\mathfrak{sl}_2$, where $V$ is two dimensional with basis $(v_+, v_-)$ and 
$$ \rho(E) = \begin{pmatrix} 0 & 0 \\ 1 & 0 \end{pmatrix}, \quad \rho(F) = \begin{pmatrix} 0 & 1 \\ 0 & 0 \end{pmatrix}, \quad \rho(K) = \begin{pmatrix} q & 0 \\ 0 & q^{-1} \end{pmatrix}.
$$
\par Let $\rho^* : U_q \mathfrak{sl}_2 \rightarrow \mathrm{End}(V^*)$ be the dual representation of $(\rho, V)$, where $\rho^*(x)$ is the transposed of $\rho(S(x))$. One has a $U_q\mathfrak{sl}_2$-equivariant isomorphism $V^* \xrightarrow{\cong} V$ whose matrix in the bases $(v_+^*, v_-^*)$ and $(v_+, v_-)$ writes 
$$ C= \begin{pmatrix} C_+^+ & C_-^+ \\ C_+^- & C_-^- \end{pmatrix} := \begin{pmatrix} 0 & \omega \\ -\omega^{5} & 0 \end{pmatrix}. \mbox{ Therefore }C^{-1}= -A^3 C =  \begin{pmatrix} 0 & -\omega^{-5} \\ \omega^{-1} & 0 \end{pmatrix}.$$

\par Define the operators $\tau, q^{\frac{H\otimes H}{2}} \in \mathrm{End}(V^{\otimes 2})$ by $\tau (v_i\otimes v_j ):= v_j \otimes v_i$ and $q^{\frac{H\otimes H}{2}} (v_i\otimes v_j) = A^{ij} v_i\otimes v_j$ for $i,j \in \{+, -\}$ (we identified $-$ with $-1$ and $+$ with $+1$). Let $\mathscr{R} \in \mathrm{End}(V^{\otimes 2})$ be the braiding operator
$$ \mathscr{R} = \tau \circ q^{\frac{H\otimes H}{2}} \circ \exp_q\left( (q-q^{-1})\rho(E) \otimes \rho(F) \right) =  \tau \circ q^{\frac{H\otimes H}{2}} \circ \left( \mathds{1}_2 + (q-q^{-1}) \rho(E) \otimes \rho(F) \right) .$$
In the basis $\left( v_+ \otimes v_+, v_+ \otimes v_-, v_-\otimes v_+, v_-\otimes v_- \right)$, it writes 
$$ \mathscr{R} =  
\begin{pmatrix}
 \mathscr{R}_{++}^{++} & \mathscr{R}_{+-}^{++} &\mathscr{R}_{-+}^{++} &\mathscr{R}_{--}^{++} \\
\mathscr{R}_{++}^{+-} &\mathscr{R}_{+-}^{+-} &\mathscr{R}_{-+}^{+-} &\mathscr{R}_{--}^{+-}  \\
\mathscr{R}_{++}^{-+} &\mathscr{R}_{+-}^{-+} &\mathscr{R}_{-+}^{-+} &\mathscr{R}_{--}^{-+} \\
\mathscr{R}_{++}^{--} &\mathscr{R}_{+-}^{--} &\mathscr{R}_{-+}^{--} &\mathscr{R}_{--}^{--} 
  \end{pmatrix}
:= \begin{pmatrix} A & 0 & 0 & 0 \\ 0 & 0 &A^{-1} & 0 \\ 0 & A^{-1} & A-A^{-3} & 0 \\ 0 & 0 & 0 & A \end{pmatrix}, \mbox{so } 
\mathscr{R}^{-1}  =
\begin{pmatrix} A^{-1} & 0 & 0 & 0 \\ 0 & A^{-1} - A^3 & A & 0 \\ 0 & A & 0 & 0 \\ 0 & 0 & 0 & A^{-1}\end{pmatrix}.
$$

We now list three families of skein relations, which are straightforward consequences of the definition, and will be used in the paper. Let $i,j \in \{ -, + \}$.

\par $\bullet$  \textit{The trivial arc relations:} 

\begin{equation}\label{trivial_arc_rel}
\begin{tikzpicture}[baseline=-0.4ex,scale=0.5,>=stealth]
\draw [fill=gray!45,gray!45] (-.7,-.75)  rectangle (.4,.75)   ;
\draw[->] (0.4,-0.75) to (.4,.75);
\draw[line width=1.2] (0.4,-0.3) to (0,-.3);
\draw[line width=1.2] (0.4,0.3) to (0,.3);
\draw[line width=1.1] (0,0) ++(90:.3) arc (90:270:.3);
\draw (0.65,0.3) node {\scriptsize{$i$}}; 
\draw (0.65,-0.3) node {\scriptsize{$j$}}; 
\end{tikzpicture}
= C^i_j 
\hspace{.2cm}
\begin{tikzpicture}[baseline=-0.4ex,scale=0.5,>=stealth]
\draw [fill=gray!45,gray!45] (-.7,-.75)  rectangle (.4,.75)   ;
\draw[-] (0.4,-0.75) to (.4,.75);
\end{tikzpicture}
, \hspace{.4cm}
\begin{tikzpicture}[baseline=-0.4ex,scale=0.5,>=stealth]
\draw [fill=gray!45,gray!45] (-.7,-.75)  rectangle (.4,.75)   ;
\draw[->] (-0.7,-0.75) to (-.7,.75);
\draw[line width=1.2] (-0.7,-0.3) to (-0.3,-.3);
\draw[line width=1.2] (-0.7,0.3) to (-0.3,.3);
\draw[line width=1.15] (-.4,0) ++(-90:.3) arc (-90:90:.3);
\draw (-0.9,0.3) node {\scriptsize{$i$}}; 
\draw (-0.9,-0.3) node {\scriptsize{$j$}}; 
\end{tikzpicture}
=(C^{-1})^i_j 
\hspace{.2cm}
\begin{tikzpicture}[baseline=-0.4ex,scale=0.5,>=stealth]
\draw [fill=gray!45,gray!45] (-.7,-.75)  rectangle (.4,.75)   ;
\draw[-] (-0.7,-0.75) to (-0.7,.75);
\end{tikzpicture}.
\end{equation}

\par $\bullet$  \textit{The cutting arc relations:}

\begin{equation}\label{cutting_arc_rel}
\heightcurveright
= \sum_{i,j = \pm} C^i_j
 \hspace{.2cm} 
\heightexchright{->}{i}{j}
, \hspace{.4cm} 
\heightcurve =
\sum_{i,j = \pm} (C^{-1})_j^i
\hspace{.2cm}
\heightexch{->}{i}{j}
\end{equation}.

\par $\bullet$ \textit{The height exchange relations:}

\begin{equation}\label{height_exchange_rel}
\heightexch{->}{i}{j}= \adjustbox{valign=c}{\includegraphics[width=0.9cm]{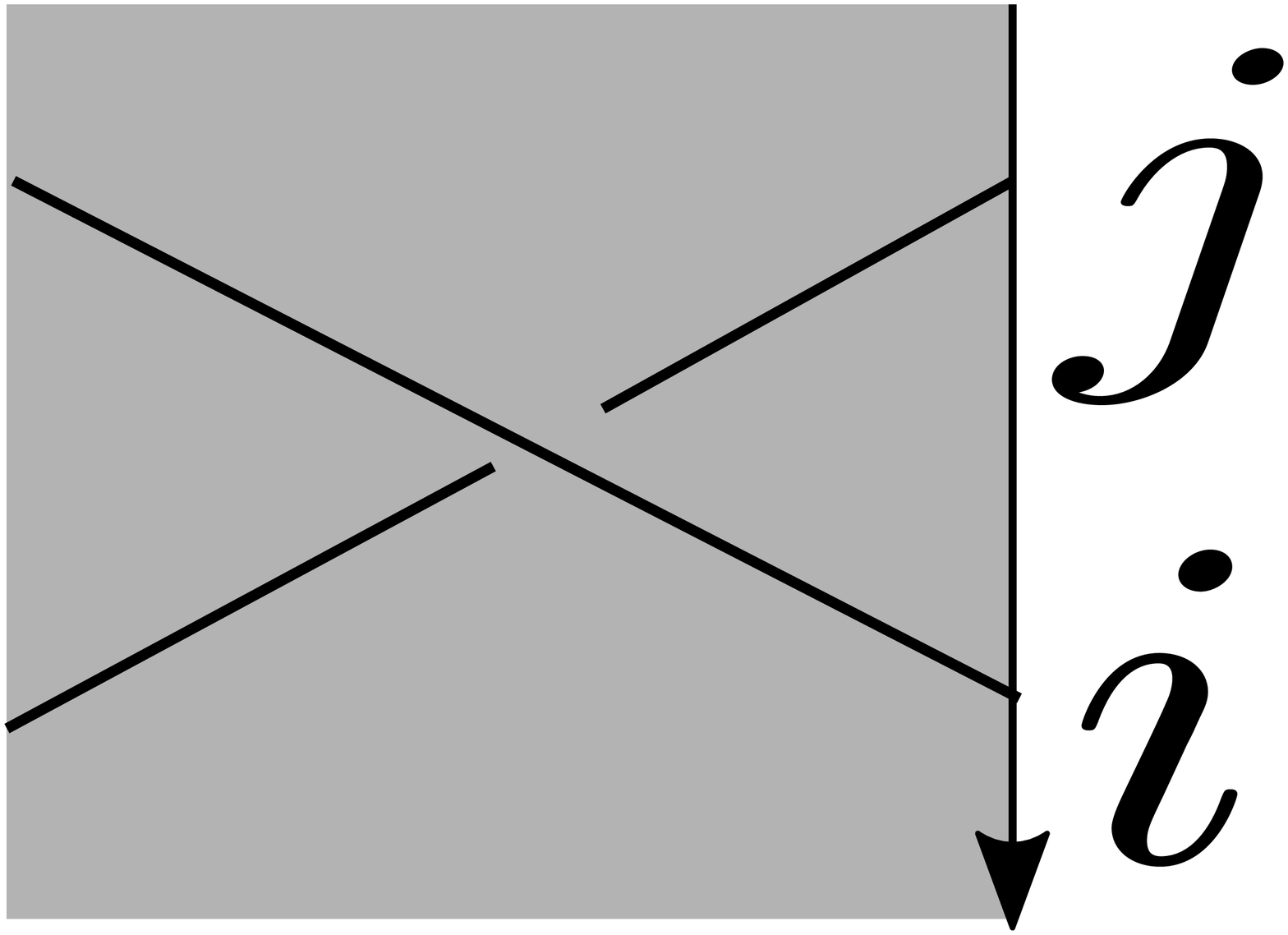}} =  \sum_{k,l = \pm} \mathscr{R}_{i j}^{k l} 
 \hspace{.2cm} 
 \heightexch{<-}{l}{k}
 , \hspace{.4cm}
 \heightexch{<-}{j}{i} = \adjustbox{valign=c}{\includegraphics[width=0.9cm]{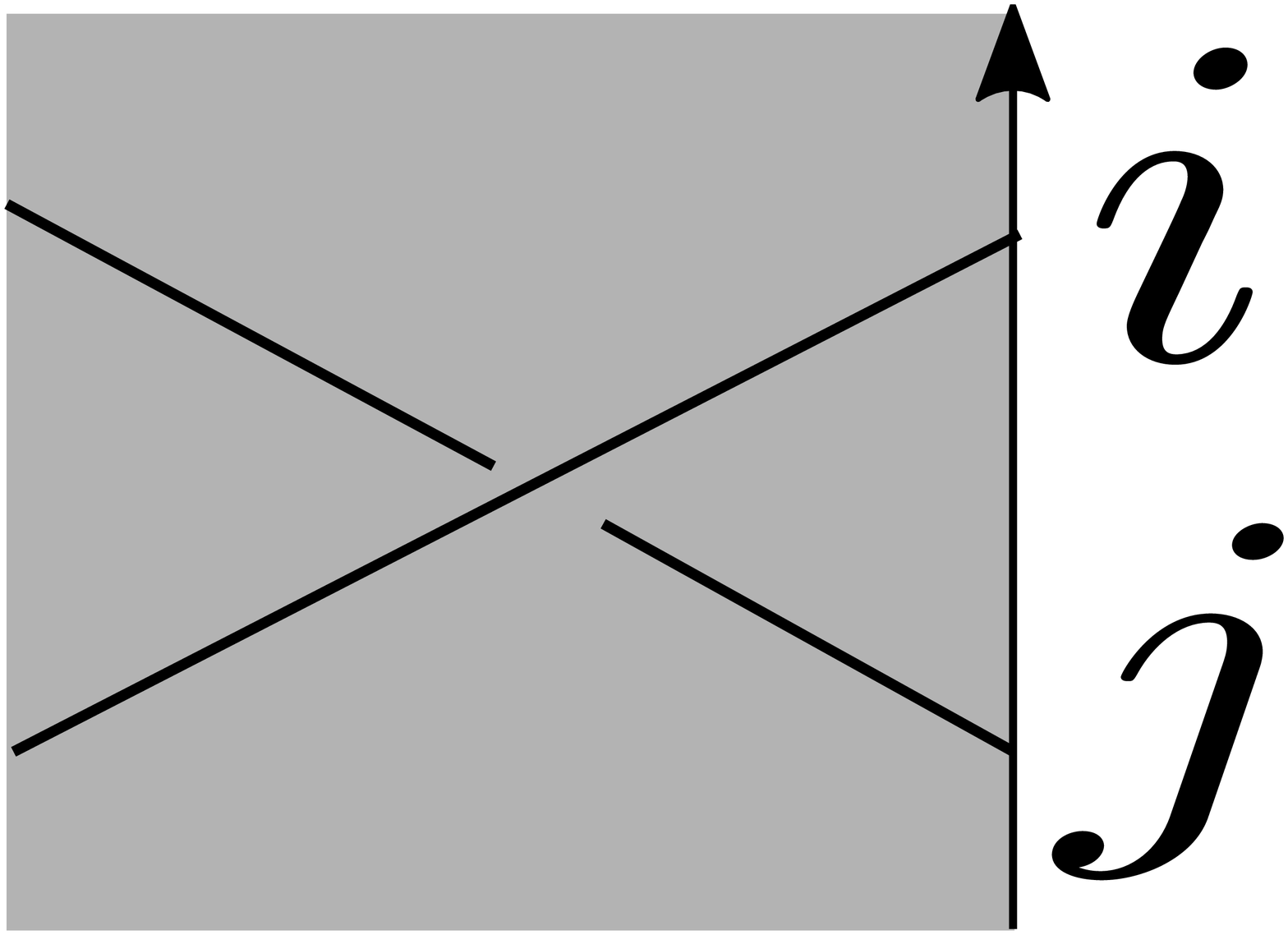}} = \sum_{k,l = \pm} (\mathscr{R}^{-1})_{i j}^{k l} 
 \hspace{.2cm}
 \heightexch{->}{k}{l}.
\end{equation}

We refer to \cite{LeStatedSkein} for proofs. 

\vspace{2mm}
\par

The algebra $\mathcal{O}_q[\mathrm{SL}_2]$ is the algebra presented by generators $x_{\varepsilon \varepsilon'}, \varepsilon, \varepsilon' \in \{-, +\}$ and relations

\begin{align*}\label{relbigone}
x_{++}x_{+-} &= q^{-1}x_{+-}x_{++} & x_{++}x_{-+}&=q^{-1}x_{-+}x_{++}
\\ x_{--}x_{+-} &= q x_{+-}x_{--} & x_{--}x_{-+}&=q x_{-+}x_{--}
\\ x_{++}x_{--}&=1+q^{-1}x_{+-}x_{-+} &  x_{--}x_{++}&=1 + q x_{+-}x_{-+} 
\\ x_{-+}x_{+-}&=x_{+-}x_{-+} & &
\end{align*}

It has a Hopf algebra structured characterized by the formulas 

 \begin{equation*}
 \begin{pmatrix} \Delta (x_{++}) & \Delta (x_{+-}) \\ \Delta(x_{-+}) & \Delta(x_{--}) \end{pmatrix} 
 = 
 \begin{pmatrix} x_{++} & x_{+-} \\ x_{-+} & x_{--} \end{pmatrix} 
 \otimes 
 \begin{pmatrix} x_{++} & x_{+-} \\ x_{-+} & x_{--} \end{pmatrix} 
 \end{equation*}
 \begin{equation*}
 \begin{pmatrix} \epsilon(x_{++}) & \epsilon(x_{+-}) \\ \epsilon(x_{-+}) & \epsilon(x_{--}) \end{pmatrix} =
\begin{pmatrix} 1 &0 \\ 0& 1 \end{pmatrix}  
\text{ and }
\begin{pmatrix} S(x_{++}) & S(x_{+-}) \\ 	S(x_{-+}) & S(x_{--}) \end{pmatrix} 
	= 
	\begin{pmatrix} x_{--} & -q x_{+-} \\ -q^{-1}x_{-+} & x_{++} \end{pmatrix} .
 \end{equation*}
 
 When $q\in \mathbb{C}^*$ is generic (not a root of unity), $\mathcal{O}_q[\mathrm{SL}_2]$ is the restricted dual of $U_q\mathfrak{sl}_2$ (see \cite{BrownGoodearl}).  The \textit{bigon} $\mathbb{B}$ is the punctured surface made of a disc with two punctures on its boundary. It has two boundary arcs $a$ and $b$ and is generated by the stated arcs $\alpha_{\varepsilon \varepsilon}, \varepsilon, \varepsilon'=\pm$ made of an arc $\alpha$ linking $a$ to $b$ with state $\varepsilon$ on $\alpha \cap a$ and $\varepsilon'$ on $\alpha \cap b$. Consider a disjoint union $\mathbb{B} \bigsqcup \mathbb{B}$ of two bigons; by gluing together the boundary arc $b_1$ of the first bigon with the boundary arc $a_2$ of the second, one obtains a morphism $\Delta:=i_{b_1\# a_2} : \mathcal{S}_{\omega}(\mathbb{B}) \rightarrow \mathcal{S}_{\omega}(\mathbb{B})^{\otimes 2}$ which endows $\mathcal{S}_{\omega}(\mathbb{B})$ with a structure of Hopf algebra where $\Delta$ is the coproduct. 
 
 \begin{theorem}[\cite{LeStatedSkein, CostantinoLe19, KojuQuesneyClassicalShadows}] There is an isomorphism of Hopf algebras $\varphi : \mathcal{O}_q[\SL_2] \cong \mathcal{S}_{\omega}(\mathbb{B})$ sending the generator $x_{\varepsilon \varepsilon'} \in \mathcal{O}_q[\SL_2]$ to the element $\alpha_{\varepsilon \varepsilon'} \in \mathcal{S}_{\omega}(\mathbb{B})$.
 \end{theorem}
 
 More precisely, the fact that $\varphi$ is an isomorphism of algebras is proved in \cite{LeStatedSkein} and the fact that it preserves the coproduct was noticed independently in 
 \cite{CostantinoLe19, KojuQuesneyClassicalShadows}.  In all the paper, we will (abusively) identify the Hopf algebras $ \mathcal{O}_q[\SL_2]$ and $ \mathcal{S}_{\omega}(\mathbb{B})$ using $\varphi$. Note that the definition of $\varphi$ depends on an indexing by $a$ and $b$ of the boundary arcs of $\mathbb{B}$. 
\vspace{2mm}
\par 
 Now consider a punctured surface $\mathbf{\Sigma}$ and a boundary arc $c$. By gluing a bigon $\mathbb{B}$ along $\mathbf{\Sigma}$ while gluing $b$ with $c$, one obtains a punctured surface isomorphic to $\mathbf{\Sigma}$, hence a map $\Delta_c^L:= i_{b\#c} :   \mathcal{S}_{\omega}(\mathbf{\Sigma}) \rightarrow  \mathcal{O}_q[\SL_2] \otimes \mathcal{S}_{\omega}(\mathbf{\Sigma})$ which endows $ \mathcal{S}_{\omega}(\mathbf{\Sigma})$ with a structure of left $\mathcal{O}_q[\SL_2]$ comodule. Similarly, gluing $c$ with $a$ induces a right comodule morphism $\Delta_c^R := i_{c\#a} :  \mathcal{S}_{\omega}(\mathbf{\Sigma}) \rightarrow  \mathcal{S}_{\omega}(\mathbf{\Sigma})\otimes  \mathcal{O}_q[\SL_2] $.
The following theorem characterizes the image of the gluing map and  was proved independently in \cite{CostantinoLe19} and \cite{KojuQuesneyClassicalShadows}.

\begin{theorem}\label{theorem_exactsequence}(\cite[Theorem $4.7$]{CostantinoLe19}, \cite[Theorem $1.1$]{KojuQuesneyClassicalShadows})
Let $\mathbf{\Sigma}$ be a punctured surface and $a,b$ two boundary arcs. The following sequence is exact: 
$$ 0 \to \mathcal{S}_{\omega}(\mathbf{\Sigma}_{|a\#b}) \xrightarrow{i_{a\#b}} \mathcal{S}_{\omega}(\mathbf{\Sigma}) \xrightarrow{\Delta^L_a - \sigma \circ \Delta^R_b}\mathcal{O}_q[\SL_2] \otimes \mathcal{S}_{\omega}(\mathbf{\Sigma}), $$
where $\sigma(x\otimes y) := y\otimes x$. 
\end{theorem}

An easy but very important consequence of the fact that $\Delta_a^L$ and $\Delta_a^R$ are comodule maps are the following \textit{boundary skein relations}:

\begin{equation}\label{boundary_skein_rel}
\left(  \epsilon \otimes \id \right) \circ \Delta_a^L = \id  \quad \mbox{and} \quad  \left( \id \otimes  \epsilon \right) \circ \Delta_a^R = \id. 
\end{equation}

 The image through the counit $\epsilon$ of a stated diagram in $\mathbb{B}$ can be computed using the formulas: 
 
 \begin{equation}\label{epsilon_formula}
\epsilon \left( \bigonheightcurve{i}{j}\right) = C_j^i , \quad \epsilon \left( \bigonheightcurveright{i}{j} \right) = (C^{-1})_j^i , \quad  \epsilon \left( \adjustbox{valign=c}{\includegraphics[width=1.2cm]{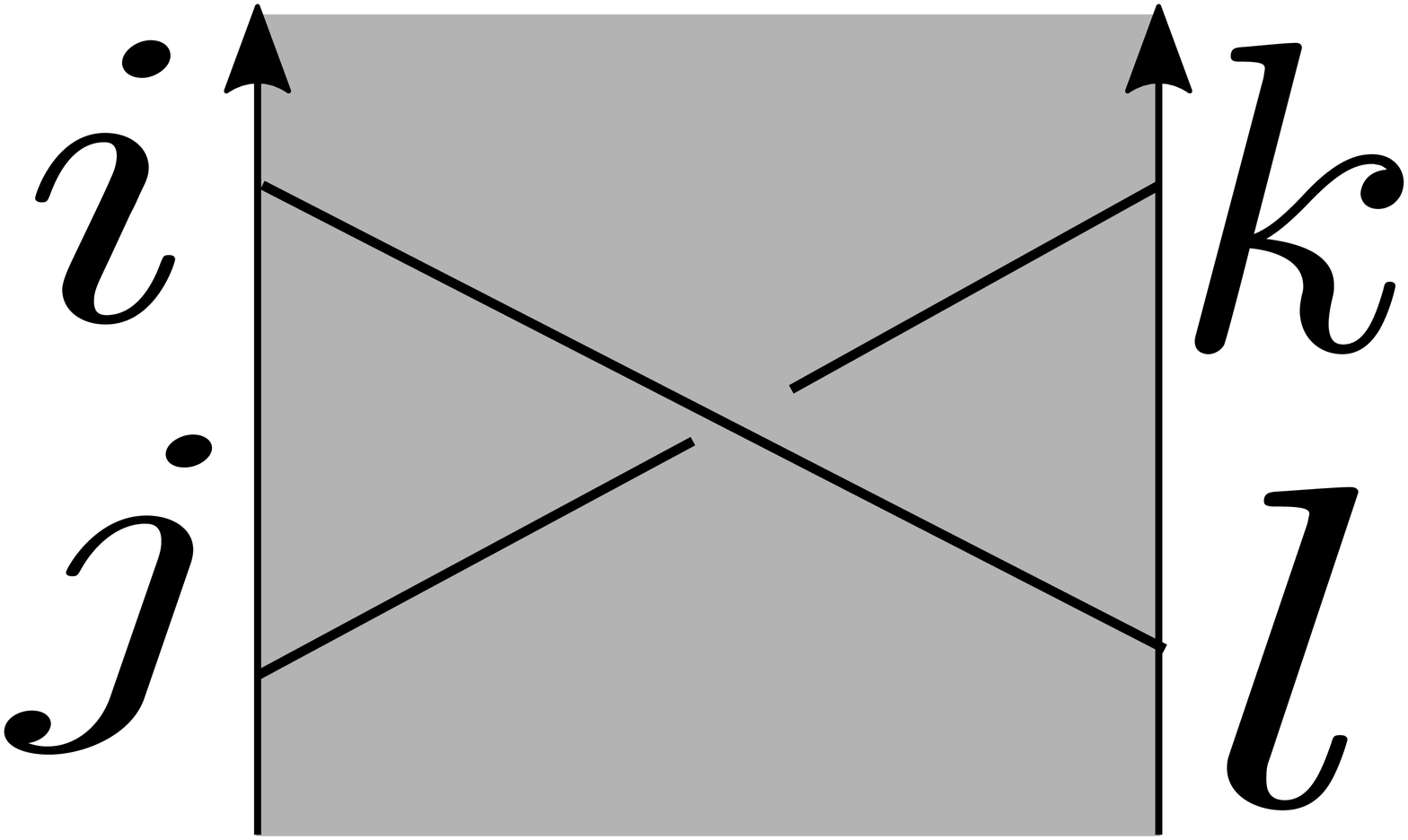}} \right) = \mathscr{R}^{i j}_{k l} , \quad \epsilon \left( \adjustbox{valign=c}{\includegraphics[width=1.2cm]{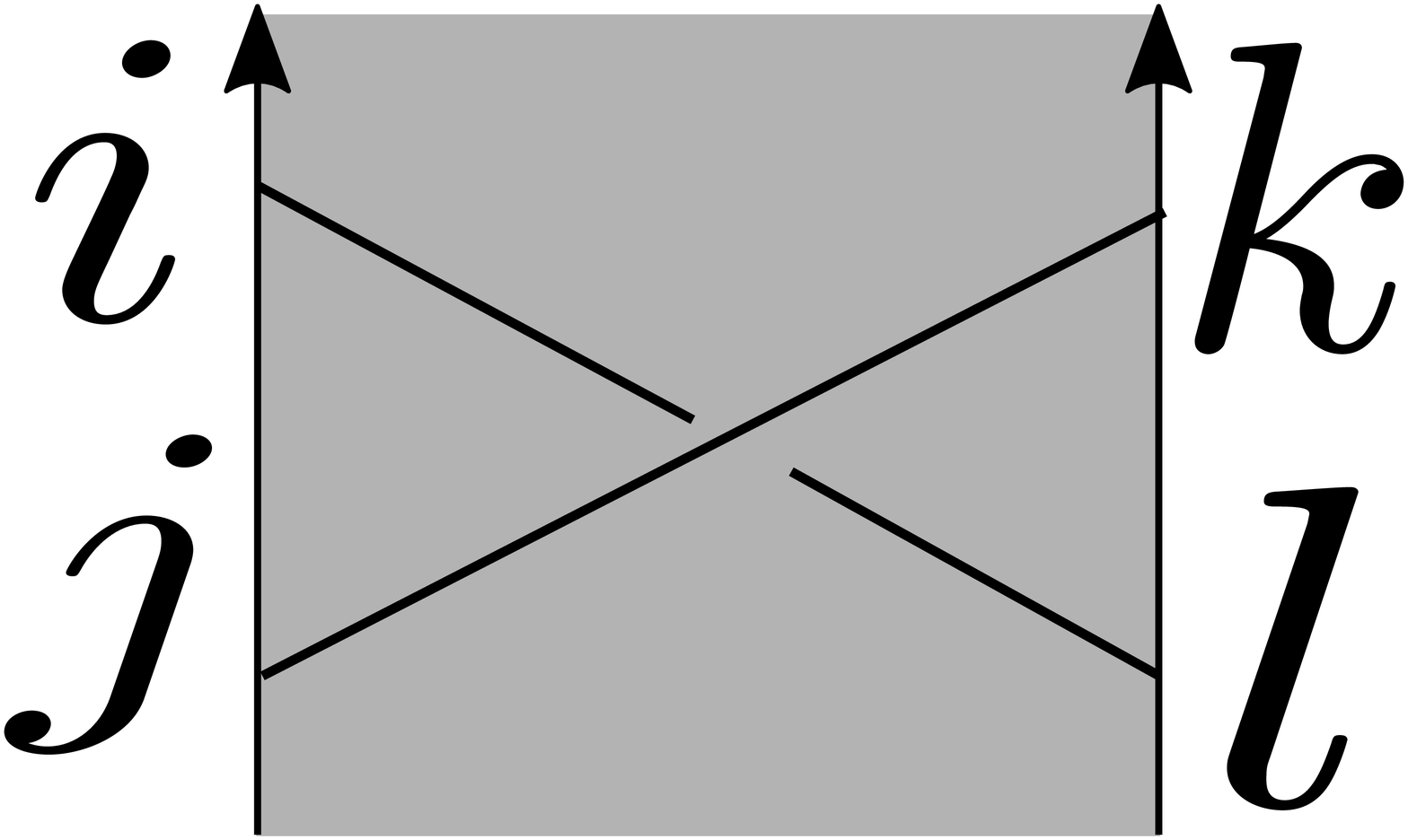}}\right) = (\mathscr{R}^{-1})^{i j}_{k l}.
 \end{equation}
 
 Figure \ref{fig_boundary_skein} illustrates an instance of  boundary skein relation \eqref{boundary_skein_rel}. Here we draw a dotted arrow to illustrate where we cut the bigon. 
 Note that all the trivial arcs \eqref{trivial_arc_rel}, cutting arc \eqref{cutting_arc_rel} and height exchange \eqref{height_exchange_rel} relations are particular cases of \eqref{boundary_skein_rel}.

 \begin{figure}[!h] 
\centerline{\includegraphics[width=10cm]{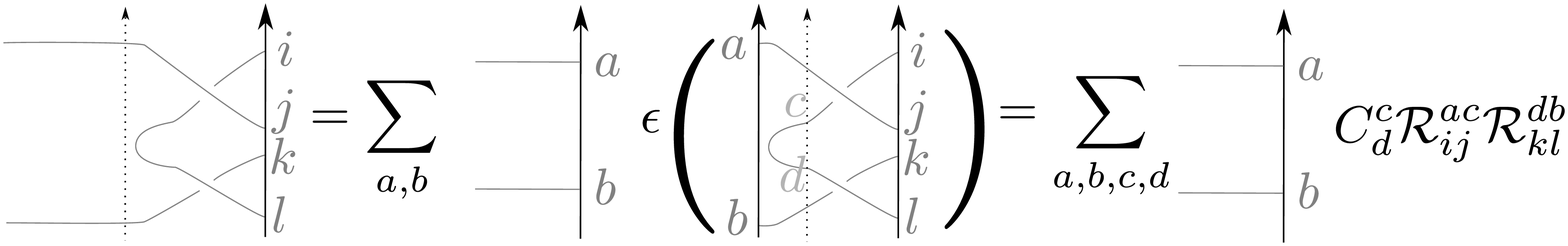} }
\caption{An example of boundary skein relation.} 
\label{fig_boundary_skein} 
\end{figure}

\subsection{The small fundamental groupoid  and its finite presentations}\label{sec_presentation}

During all this section we fix a punctured surface $\mathbf{\Sigma}=(\Sigma, \mathcal{P})$ such that $\Sigma$ is connected and has non empty boundary. 
  For each boundary arc $a$ of $\mathbf{\Sigma}$, fix a point $v_a \in a$ and denote by $\mathbb{V}$ the set $\{v_a\}_a$. 
  
  \begin{definition}
  The \textit{small fundamental groupoid}  $\Pi_1(\Sigma_{\mathcal{P}}, \mathbb{V})$ is the full subcategory of the fundamental groupoid $\Pi_1(\Sigma_{\mathcal{P}})$ generated by $\mathbb{V}$. 
  \end{definition}
  Said differently, $\Pi_1(\Sigma_{\mathcal{P}}, \mathbb{V})$ is the small groupoid whose set of objects is $\mathbb{V}$ and such that a morphism (called path) $\alpha : v_1 \rightarrow v_2$ is a homotopy class of continuous map $\varphi_{\alpha} : [0,1] \rightarrow \Sigma_{\mathcal{P}}$  with $\varphi_{\alpha}(0)=v_1$ and $\varphi_{\alpha}(1)=v_2$. 
  The map $\varphi_{\alpha}$ will be referred to as a \textit{geometric representative} of $\alpha$.
  The composition is the concatenation of paths. For a path $\alpha : v_1 \rightarrow v_2$, we write $s(\alpha)=v_1$ (the source point) and $t(\alpha)=v_2$ (the target point) and $\alpha^{-1} : v_2 \rightarrow v_1$ the path with opposite orientation (\textit{i.e.} $\varphi_{\alpha^{-1}}(t) = \varphi_{\alpha} (1-t)$).  
  \vspace{2mm}
\par We will define the notion of \textit{finite presentation} $\mathbb{P}$ of the groupoid  $\Pi_1(\Sigma_{\mathcal{P}}, \mathbb{V})$ and attach to each such $\mathbb{P}$ a finite presentation of $\mathcal{S}_{\omega}(\mathbf{\Sigma})$. In order to get some intuition, consider the punctured surface in Figure \ref{fig_presentation}: it is an annulus with two punctures per boundary component, so it has four boundary arcs. The figure shows some paths $\beta_1, \ldots, \beta_5$ and we will say that $\Pi_1(\Sigma_{\mathcal{P}}, \mathbb{V})$ is finitely presented by the set of generators $\{\beta_1, \ldots, \beta_5\}$ together with the relation $\beta_2^{-1} \beta_4 \beta_5 \beta_3 =1$. 
We will deduce that $\mathcal{S}_{\omega}(\mathbf{\Sigma})$ is generated by the stated arcs $(\beta_i)_{\varepsilon \varepsilon'}$ and that the relation $\beta_2^{-1} \beta_4 \beta_5 \beta_3 =1$ induces a relation among them.
Alternatively, the same punctured surface has a presentation with the smaller set of generators $\{\beta_1, \ldots, \beta_4\}$ and no relation. The induced finite presentation of $\mathcal{S}_{\omega}(\mathbf{\Sigma})$ will be simpler.

\begin{figure}[!h] 
\centerline{\includegraphics[width=3cm]{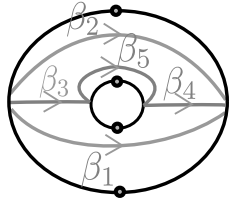} }
\caption{A punctured surface and a set of generators for its small fundamental groupoid.} 
\label{fig_presentation} 
\end{figure}

\begin{definition}\label{def_generators} \begin{enumerate}
\item  A \textit{set of generators} for $\Pi_1(\Sigma_{\mathcal{P}}, \mathbb{V})$ is a set $\mathbb{G}$ of paths in $\Pi_1(\Sigma_{\mathcal{P}}, \mathbb{V})$ such that any path $\alpha \in \Pi_1(\Sigma_{\mathcal{P}}, \mathbb{V})$ decomposes as $\alpha= \alpha_1^{\varepsilon_1} \ldots \alpha_n^{\varepsilon_n}$ with $\varepsilon_i = \pm 1$ and $\alpha_i \in \mathbb{G}$. We also require that each path $\alpha \in \mathbb{G}$ is the homotopy class of some embedding  $\varphi_{\alpha} : [0,1] \rightarrow \Sigma_{\mathcal{P}}$ such that the images of the $\varphi_{\alpha}$  do not intersect outside $\mathbb{V}$ and eventually intersect transversally at $\mathbb{V}$. The \textit{generating graph} is the oriented ribbon graph $\Gamma \subset \Sigma_{\mathcal{P}}$ whose set of vertices is $\mathbb{V}$ and edges are the image of the $\varphi$. We will always assume implicitly that the geometric representatives $\varphi_{\alpha}$ is part of the data defining a set of generators. Moreover, when $\alpha \in \mathbb{G}$ is a path such that $s(\alpha)=t(\alpha)$ (i.e. $\alpha$ is a loop) we add the additional datum of a "height order" for its endpoints, that is we specify whether $h(s(\alpha))< h(t(\alpha))$ or $h(t(\alpha))< h(s(\alpha))$.
\item For a path $\alpha : v_1 \rightarrow v_2$ and $\varepsilon, \varepsilon' \in \{-,+\}$, we denote by $\alpha_{\varepsilon \varepsilon'} \in \mathcal{S}_{\omega}(\mathbf{\Sigma})$ the class of the stated arc $(\alpha, \sigma)$, where the state $\sigma$ is given by $\sigma(v_1)=\varepsilon$ and $\sigma(v_2)=\varepsilon'$. When both endpoints lye in the same boundary arc (i.e. when $s(\alpha)=t(\alpha)$) we use the chosen height order to specify which endpoint lies on the top. Set 
$$ \mathcal{A}^{\mathbb{G}} := \{ \alpha_{\varepsilon \varepsilon'} | \alpha \in \mathbb{G}, \varepsilon, \varepsilon' \in \{-,+\}  \} \subset \mathcal{S}_{\omega}(\mathbf{\Sigma}). $$
\end{enumerate}
 \end{definition}

\begin{example}\label{exemple_pres}
For any connected open punctured surface $\mathbf{\Sigma}$, the groupoid $\Pi_1(\Sigma_{\mathcal{P}}, \mathbb{V})$ admits a finite set of generators depicted in Figure \ref{fig_generators_final} and defined as follows. Denote by $a_0, \ldots, a_n$ the boundary arcs, by $\partial_0, \ldots, \partial_r$ the boundary components of $\Sigma$ with $a_0\subset \partial_0$ and write  $v_i:= a_i \cap \mathbb{V}$. 
Let $\overline{\Sigma}$ be the surface obtained from $\Sigma$ by gluing a disc along each boundary component $\partial_i$ for $1\leq i \leq r$, and choose $\alpha_1, \beta_1, \ldots, \alpha_g, \beta_g$ some paths in $\pi_1(\Sigma_{\mathcal{P}}, v_0)(=\mathrm{End}_{\Pi_1(\Sigma_{\mathcal{P}}, \mathbb{V})}(v_0)$), such that their images in $\overline{\Sigma}$ generate the free group $\pi_1(\overline{\Sigma}, v_0)$ (said differently, the $\alpha_i$ and $\beta_i$ are longitudes and meridians of $\Sigma$). For each inner puncture $p$ choose a peripheral curve $\gamma_p \in \pi_1(\Sigma_{\mathcal{P}}, v_0)$ encircling $p$ once and for each boundary puncture $p_{\partial}$ between two boundary arcs $a_i$ and $a_j$, consider the path $\alpha_{p_{\partial}} : v_i \rightarrow v_j $ represented by the corner arc in $p_{\partial}$. Eventually, for each boundary component $\partial_j$, with $1\leq j \leq r$, containing a boundary arc $a_{k_j} \subset \partial_j$,  choose a path $\delta_{\partial_j} : v_0 \rightarrow v_{k_j}$. The set 
$$\mathbb{G}':= \{ \alpha_i, \beta_i, \alpha_p, \delta_{\partial_j} | 1\leq i \leq g, p\in \mathcal{P}, 1\leq j \leq r\}$$
is a generating set for $\Pi_1(\Sigma_{\mathcal{P}}, \mathbb{V})$ and Figure \ref{fig_generators_final} represents a set of geometric representatives for $\mathbb{G}'$. Moreover each of its generators which is not one of the $\delta_{\partial_j}$ can be expressed as a composition of the other ones (we will soon say that there is a relation among those generators), therefore a set $\mathbb{G}$ obtained from $\mathbb{G}'$ by removing one of the element of the form $\alpha_i, \beta_i$ or $\gamma_p$, is still a generating set for  $\Pi_1(\Sigma_{\mathcal{P}}, \mathbb{V})$. The height orders can be chosen arbitrarily.
Note that $\mathbb{G}$ has cardinality $2g-2+s+n_{\partial}$, where $g$ is the genus of $\Sigma$, $s:=|\mathcal{P}|$ is the number of punctures and $n_{\partial}:= |\pi_0(\partial \Sigma)|$ is the number of boundary components.

 In the particular case where $\Sigma$ has exactly one boundary component with one puncture on it (and eventual inner punctures), the generating graph of $\mathbb{G}$ is called the \textit{daisy graph}. The daisy graph was first considered in \cite{AlekseevMalkin_PoissonCharVar} in the context of classical lattice gauge field theory and in \cite{AlekseevSchomerus_RepCS, Faitg_LGFT_MCG, Faitg_LGFT_SSkein, BaseilhacRoche_LGFT1} in the quantum case.
\begin{figure}[!h] 
\centerline{\includegraphics[width=9cm]{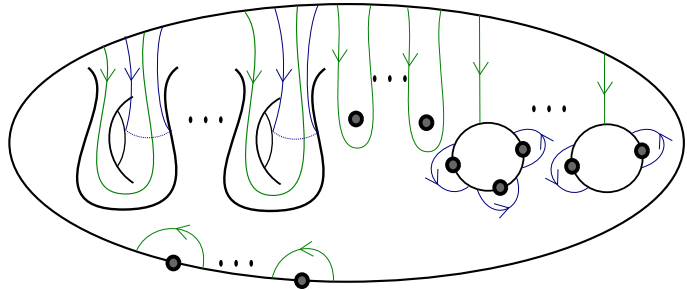} }
\caption{The geometric representatives of a set of generators for $\Pi_1(\Sigma_{\mathcal{P}}, \mathbb{V})$.} 
\label{fig_generators_final} 
\end{figure} 

\end{example}

 \begin{proposition}\label{prop_generators}[\cite[Proposition $3.4$]{KojuAzumayaSkein}]
If $\mathbb{G}$ is a set of generators of $\Pi_1(\Sigma_{\mathcal{P}}, \mathbb{V})$, then the set $ \mathcal{A}^{\mathbb{G}} $
generates $\mathcal{S}_{\omega}(\mathbf{\Sigma}) $ as an algebra.
\end{proposition}

The proof of Proposition \ref{prop_generators} is an easy consequence of the cutting arc relations illustrated in Figure \ref{fig_prop_generators}.

\begin{figure}[!h] 
\centerline{\includegraphics[width=12cm]{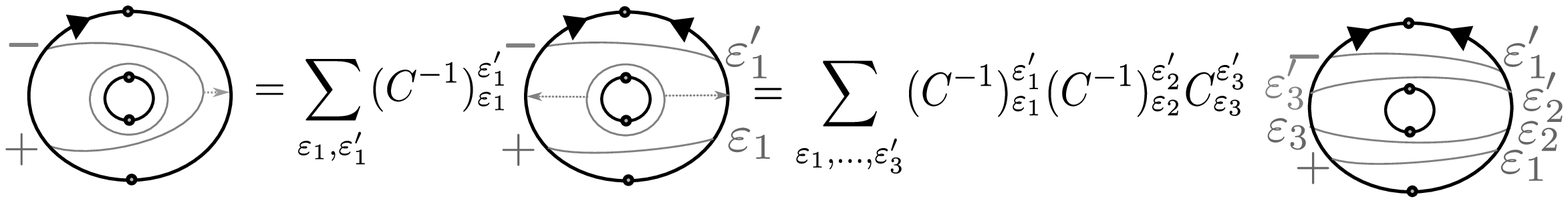} }
\caption{The figure illustrates how an application of the cutting arc relations permits to express any simple stated diagram in terms of the elements of $\mathcal{A}^{\mathbb{G}}$. Here $\mathbb{G}=\{\beta_1, \beta_2, \beta_3, \beta_4 \}$ are the generators of Figure \ref{fig_presentation}. We draw dotted arrows to exhibit where we perform the cutting arc relations.} 
\label{fig_prop_generators} 
\end{figure} 

\vspace{2mm}
\par We now define the notion of relations for a generating set $\mathbb{G}$. Let $\mathcal{F}(\mathbb{G})$ denote the free semi-group generated by the elements of $\mathbb{G}$ and let $\mathrm{Rel}_{\mathbb{G}}$ denote the subset of $\mathcal{F}(\mathbb{G})$ of elements of the form $R=\beta_{1}\star \ldots \star \beta_{n}$ such that $s(\beta_{i})= t(\beta_{i+1})$ and such that the path $\beta_1\ldots \beta_n$ is trivial. We write $R^{-1}:= \beta_n^{-1} \star \ldots \star \beta_1^{-1}$. A relation $R=\beta_{1}\star \ldots \star \beta_{n}\in  \mathrm{Rel}_{\mathbb{G}}$ is called \textit{simple} if the $\beta_i$ admit representative as embedded curves whose concatenation forms a contractible simple closed curve $\gamma$ in $\Sigma_{\mathcal{P}}$ whose orientation coincides with the orientation of the disc bounded by $\gamma$. Note that "being simple" depends on the choice of geometric representatives of the generators.

\begin{definition}
A finite subset $\mathbb{RL}\subset \mathrm{Rel}_{\mathbb{G}}$ is called a \textit{finite set of relations} if its elements are simple and  every word  $R\in \mathrm{Rel}_{\mathbb{G}}$ can be decomposed as $ R = \beta \star R_1^{\varepsilon_1} \star \ldots \star R_m^{\varepsilon_m}\star \beta^{-1}$, where $R_i \in \mathbb{RL}$, $\varepsilon_i \in \{ \pm 1 \}$ and $\beta=\beta_{1}\star \ldots \star \beta_{n}\in \mathcal{F}(\mathbb{G})$ is such that $s(\beta_{i})= t(\beta_{i+1})$.
The pair $\mathbb{P}:=(\mathbb{G}, \mathbb{RL})$ is called a \textit{finite presentation} of $\Pi_1(\Sigma_{\mathcal{P}}, \mathbb{V})$.
\end{definition}

 As illustrated in the introduction, the small fundamental groupoid of the triangle $\mathbb{T}$ admits the finite presentation with generating set $\mathbb{G}=\{\alpha, \beta, \gamma\}$ and unique relation $\mathbb{RL}=\{ \alpha \star \beta \star \gamma \}$.
 
 For a general connected open punctured surface $\mathbf{\Sigma}$, the set $\mathbb{G}$ of Example \ref{exemple_pres} is the generating set of a presentation of $\Pi_1(\Sigma_{\mathcal{P}}, \mathbb{V})$ with no relation.

 \subsection{Relations among the generators of the stated skein algebras}
 
 We fix a connected open punctured surface $\mathbf{\Sigma}$, a finite presentation $\mathbb{P}=(\mathbb{G}, \mathbb{RL})$ of $\Pi_1(\Sigma_{\mathcal{P}}, \mathbb{V})$, and look for relations in $\mathcal{S}_{\omega}(\mathbf{\Sigma})$ among the elements of $\mathcal{A}^{\mathbb{G}}$.

\begin{definition}
 An \textit{oriented arc} $\beta$ is a non-closed connected simple diagram of $\Sigma_{\mathcal{P}}$ together with an orientation plus an eventual height order of its endpoints in the case where they both lye in the same boundary arc.
 We will denote by $s(\beta)$ and $t(\beta)$ its endpoints so that $\beta$ is oriented from $s(\beta)$ towards $t(\beta)$. For $\varepsilon, \varepsilon' \in \{-,+\}$, we denote by $\beta_{\varepsilon \varepsilon'} \in \mathcal{S}_{\omega}(\mathbf{\Sigma})$ the class of the stated diagram $(\beta, \sigma)$ where $\sigma(s(\beta))=\varepsilon$ and $\sigma(t(\beta))=\varepsilon'$.
 \end{definition}
 
 Note that to each oriented arc one can associate a path in $\Pi_1(\Sigma_{\mathcal{P}}, \mathbb{V})$ by first isotoping its endpoints to $\mathbb{V}$ and then taking its homotopy class. However a path in $\Pi_1(\Sigma_{\mathcal{P}}, \mathbb{V})$ can be associated to several distinct oriented arcs, so an oriented arc contains more information that a path in the small fundamental groupoid.

We want to see the elements of $\mathbb{G}$ as pairwise non intersecting oriented arcs as illustrated in Figure \ref{fig_path_to_arc}. Recall that by Definition \ref{def_generators}, any path $\alpha \in \mathbb{G}$ is endowed with a geometric representative $\varphi_{\alpha}$ whose image is an oriented arc  $\underline{\alpha} \subset \Sigma_{\mathcal{P}}$ so that the $\underline{\alpha}$ pairwise do not intersect outside of $\mathbb{V}$ and their intersect transversally in $\mathbb{V}$. So each point $v_a \in \mathbb{V}$ is endowed with a total order $<_{v_a}$ on the set of its adjacent arcs (so the presenting graph has a ciliated ribbon graph structure). 
 
 The orientation of $\Sigma_{\mathcal{P}}$ induces an orientation of its boundary arcs which, in turn, induces a total order $<_a$ on each boundary arc $a$, where $v_1 <_a v_2$ if $a$ is oriented from $v_1$ towards $v_2$. After isotoping the $\underline{\alpha}$ in a small neighbourhood of each $v_a$ in such a way that the vertex order order $<_{v_a}$ matches with the boundary arc order $<_a$ as illustrated in Figure \ref{fig_path_to_arc}, we get a family of pairwise non-intersecting oriented arcs representing the elements of $\mathbb{G}$.
 
 \begin{figure}[!h] 
\centerline{\includegraphics[width=10cm]{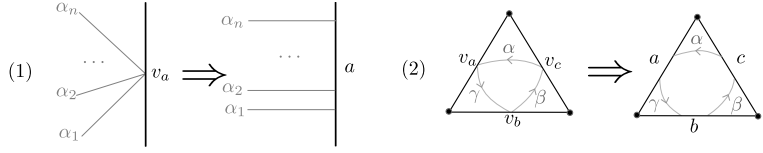} }
\caption{$(1)$ An illustration of the local isotopy we perform to turn the set of edges of a (ribbon) presenting graph into a set of pairwise non-intersecting oriented arcs. $(2)$ An example in the case of the triangle.} 
\label{fig_path_to_arc} 
\end{figure} 
 
 \begin{convention} From now on, we consider the elements of $\mathbb{G}$ as pairwise non-intersecting oriented arcs.
 \end{convention}

 \begin{definition}
 Let $\alpha$ be an oriented arc,  set $v_1:=s(\alpha)$ and $v_2:= t(\alpha)$ and denote by $u$ and $v$ the boundary arcs containing $v_1$ and $v_2$ respectively. The arc $\alpha$ is said
 \begin{itemize}
 \item \textit{of type} $a$ if $u\neq v$; 
 \item \textit{of type} $b$ if $u=v$,  $h(v_1)<h(v_2)$ and  $v_2 <_u v_1$; 
 \item \textit{of type} $c$ if $u=v$,  $h(v_1)<h(v_2)$ and  $v_1 <_u v_2$;  
   \item \textit{of type} $d$ if $u=v$,  $h(v_2)<h(v_1)$ and  $v_1 <_u v_2$; 
    \item \textit{of type} $e$ if $u=v$,  $h(v_2)<h(v_1)$ and  $v_2 <_u v_1$.
    \end{itemize}
    Here $h(v)$ represents the height of $v$ ($h$ is the second projection $\Sigma_{\mathcal{P}} \times (0,1) \rightarrow (0,1)$). Figure \ref{fig_type_arcs} illustrates the five types of oriented arcs.
    \end{definition}
    
    \begin{figure}[!h] 
\centerline{\includegraphics[width=8cm]{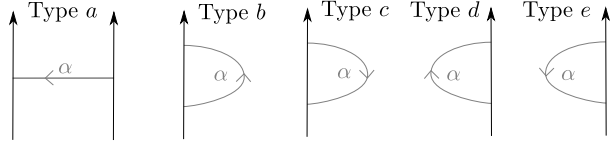} }
\caption{An illustration of the five types of oriented arcs.} 
\label{fig_type_arcs} 
\end{figure}

\begin{notations}
\begin{enumerate}
\item For $\alpha$ an oriented arc, write $M(\alpha) := \begin{pmatrix} \alpha_{++} & \alpha_{+-} \\ \alpha_{-+} & \alpha_{--} \end{pmatrix}$ the $2\times 2$ matrix with coefficients in $\mathcal{S}_{\omega}(\mathbf{\Sigma})$. The relations among the generators of $\mathcal{S}_{\omega}(\mathbf{\Sigma})$ that we will soon define are much more elegant when written using of the following matrix
$$
N(\alpha) := \left\{ 
\begin{array}{ll} 
M(\alpha) & \mbox{, if }\alpha \mbox{ is of type }a; \\
M(\alpha)C & \mbox{, if }\alpha \mbox{ is of type }b; \\
M(\alpha) ^tC & \mbox{, if }\alpha \mbox{ is of type }c; \\
C^{-1}M(\alpha) & \mbox{, if }\alpha \mbox{ is of type }d; \\
{}^t C^{-1} M(\alpha) & \mbox{, if }\alpha \mbox{ is of type }e;
\end{array} \right.
$$
 where ${}^tM$ denotes the transpose of $M$.
\item Let $M_{a,b}(R)$ the ring of $a\times b$ matrices with coefficients in some ring $R$ (here $R$ will be $\mathcal{S}_{\omega}(\mathbf{\Sigma})$). The \textit{Kronecker product} $\odot : M_{a,b}(R) \otimes M_{c,d}(R) \rightarrow M_{ac, bd}(R)$ is defined by $(A\odot B)_{j,l}^{i,k} = A^i_j B^k_l$. For instance 
$$ M(\alpha) \odot M(\beta) = 
\begin{pmatrix}
 \alpha_{++} \beta_{++} & \alpha_{++} \beta_{+-} & \alpha_{+-} \beta_{++} & \alpha_{+-} \beta_{+-} \\
 \alpha_{++} \beta_{-+} &\alpha_{++} \beta_{--} &\alpha_{+-} \beta_{-+} &\alpha_{+-} \beta_{--}  \\
 \alpha_{-+} \beta_{++} &\alpha_{-+} \beta_{+-} &\alpha_{--} \beta_{++} &\alpha_{--} \beta_{+-} \\
 \alpha_{-+} \beta_{-+} &\alpha_{-+} \beta_{--} &\alpha_{--} \beta_{-+} &\alpha_{--} \beta_{--} 
 \end{pmatrix}.
 $$

 \item By abuse of notations, we also denote by $\tau$ the matrix of the flip map $\tau: v_i\otimes v_j \mapsto v_j\otimes v_i , V^{\otimes 2} \rightarrow V^{\otimes 2}$, \textit{i.e.}
 $$ \tau = \begin{pmatrix}
 1 & 0 & 0 & 0 \\
 0 & 0 & 1 & 0 \\
 0 & 1 & 0 & 0 \\
 0 & 0 & 0 & 1
 \end{pmatrix}.
 $$
 \item For a $4\times 4$ matrix $X= (X^{ij}_{kl})_{ i,j, k, l =\pm}$, we define the $2\times 2$ matrices $\tr_L(X)$ and $\tr_R(X)$ by the formulas
 $$ \tr_L(X)_a^b := \sum_{i = \pm} X^{i b}_{i a} \quad \mbox{and} \quad \tr_R(X)_a^b := \sum_{i = \pm} X^{b i}_{a i}.$$

\item For $M= \begin{pmatrix} a & b \\ c & d \end{pmatrix}$, we set $ \mathrm{det}_q (M) := ad-q^{-1}bc$ and $\mathrm{det}_{q^2}(M) := ad - q^{-2}bc$.
\end{enumerate}
\end{notations}

\begin{lemma}[Orientation reversing formulas]\label{lemma_orientation_reversing}
Let $\alpha$ be an oriented arc and $\alpha^{-1}$ be the same arc with opposite orientation. Then one has 
$$ M(\alpha^{-1}) = {}^tM(\alpha).$$
Therefore, one has 
\begin{equation}\label{eq_inversion}
N(\alpha^{-1}) = 
\left\{ \begin{array}{ll}
{}^tN(\alpha) & \mbox{, if }\alpha \mbox{ is of type }a; \\
{}^tC^{-1} {}^tN(\alpha)^tC & \mbox{, if }\alpha \mbox{ is of type }b \mbox{ or }d; \\
C^{-1} {}^tN(\alpha) C & \mbox{, if }\alpha \mbox{ is of type }c \mbox{ or }e.
\end{array} \right.
\end{equation}
\end{lemma}

\begin{proof} This is a straightforward consequence of the definitions.
\end{proof}

\begin{lemma}[Height reversing formulas]\label{lemma_height_reversing}
Let $\alpha$ be an oriented arc with both endpoints in the same boundary arcs let $\alpha^0$ be the same arc with reversing height order for its endpoints. Then one has:
\begin{equation}\label{eq_height_reversing}
M(\alpha^0) = 
\left\{ \begin{array}{ll}
\tr_R\left( \mathscr{R}^{-1}( {}^tC^{-1} \odot M(\alpha){}^tC) \right) &  \mbox{, if }\alpha \mbox{ is of type }b; \\
\tr_L \left( \mathscr{R}^{-1}( M(\alpha)C \odot C^{-1}) \right) &  \mbox{, if }\alpha \mbox{ is of type }c; \\
\tr_L \left( ({}^tC^{-1}M(\alpha) \odot {}^tC) \mathscr{R} \right) &  \mbox{, if }\alpha \mbox{ is of type }d; \\
\tr_R \left( (C\odot C^{-1}M(\alpha)) \mathscr{R} \right)&  \mbox{, if }\alpha \mbox{ is of type }e.
\end{array}
\right.
\end{equation}
\end{lemma}

\begin{proof} Equations \eqref{eq_height_reversing} are obtained by using the boundary skein relations \eqref{boundary_skein_rel}. Figure \ref{fig_height_reversing} illustrates the proof in the case where $\alpha$ is of type e. The other cases are similar and left to the reader. 

In Figure \ref{fig_height_reversing}, we represented the curve $\alpha$ in blue to emphasize that, despite what the picture suggests, the curve can be arbitrarily complicated. Since the boundary arcs relation only involves the intersection of $\alpha$ with a small neighborhood (a bigon) of the boundary arc (colored in grey), how is the blue part of the figure does not matter.

\begin{figure}[!h] 
\centerline{\includegraphics[width=12cm]{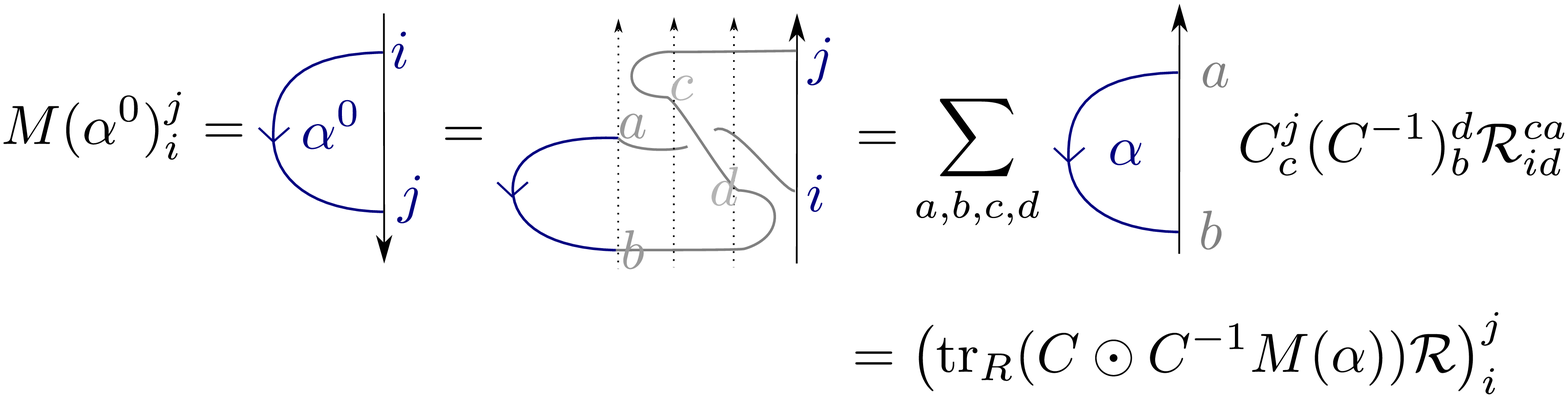} }
\caption{An illustration of the proof of  Equation \eqref{eq_height_reversing} in the case where $\alpha$ is of type e.} 
\label{fig_height_reversing} 
\end{figure}

\end{proof}

\begin{remark} Reversing the orientation of an arc exchanges (type b) $\leftrightarrow$ (type c) and (type d) $\leftrightarrow$ (type e) whereas reversing the height order exchanges (type b) $\leftrightarrow$ (type e) and (type c) $\leftrightarrow$ (type d). Therefore Equations \eqref{eq_inversion} and \eqref{eq_height_reversing} permit to switch between the types $b,c,d,e$; this will permit us to write the arcs exchange and trivial loops relations in a simpler form by specifying the type of arc.
\end{remark}

\begin{lemma}[Trivial loops relations]\label{lemma_trivial_loops_relations}
Let $R= \beta_k \star \ldots \star \beta_1$ be a simple relation. Suppose that all arcs $\beta_i$ are either of type $a$ or $d$. Then
\begin{equation}\label{eq_trivial_loops_rel}
\mathds{1}_2 = C M(\beta_k) C^{-1} M(\beta_{k-1}) C^{-1} \ldots C^{-1} M(\beta_1).
\end{equation}
\end{lemma}

\begin{proof} 
Equation \eqref{eq_trivial_loops_rel} is a consequence of the trivial arc and cutting arc relations illustrated in Figure \ref{fig_polygonrel} in the case of the triangle with presentation whose generators are the arcs $\{ \alpha, \beta, \gamma\}$ drawn in Figure \ref{fig_triangle} and the relation is $\alpha\star\beta\star\gamma=1$. Figure \ref{fig_polygonrel} shows the equality between the matrix coefficients of $C^{-1}$ and $M(\alpha) C^{-1} M(\beta) C^{-1} M(\gamma)$.
\begin{figure}[!h] 
\centerline{\includegraphics[width=12cm]{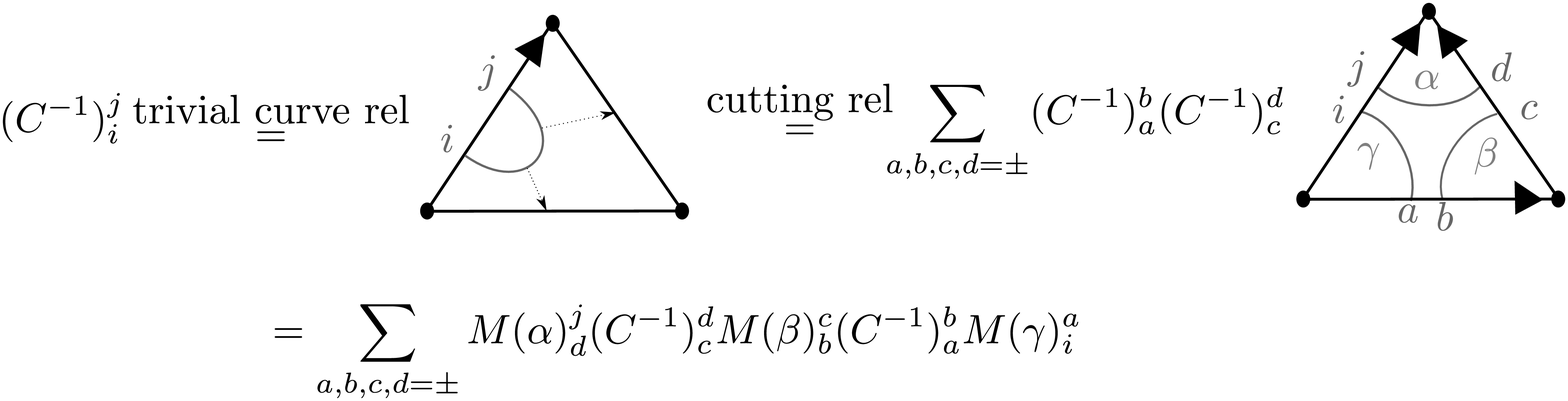} }
\caption{An illustration of the proof of  Equation \eqref{eq_trivial_loops_rel} in the case of the triangle.} 
\label{fig_polygonrel} 
\end{figure} 

\par Let us detail the proof in the general case. Since $\beta_i$ is either of type $a$ or $d$, it can be represented by a tangle $T(\beta_i)$ such that the height of the source endpoint of $\beta_i$ (say $v_i$) is smaller than the height of its target endpoint (say $w_i$); said differently $h(v_i)<h(w_i)$. One can further choose the $T(\beta_i)$ so that $T(\beta_{i+1})$ lies on the top of $T(\beta_i)$ (so $h(v_1)<h(w_1)<h(v_2)<\ldots <h(w_k)$). Let $T$ be the tangle made of the disjoint union of the $T(\beta_i)$. By the assumption that $R$ is a simple relation, we can suppose that $T$ is in generic position (in the sense of Section $2.1$) and that its projection diagram is simple. Fix $i,j \in \{-, +\}$ and let $\alpha^{0}$ be a trivial arc with endpoints $s(\alpha^{0})=v_1$ and $t(\alpha^{0})=w_k$ so that $\alpha^{0}$ can be isotoped (relatively to its boundary) to an arc inside $\partial \Sigma_{\mathcal{P}}$. One the one hand, the trivial arc relation \eqref{trivial_arc_rel} gives the equality $\alpha^{0}_{i j} = (C^{-1})_i^j$. On the other hand, the cutting arc relation \eqref{cutting_arc_rel} gives the equality

\begin{align*}
& (C^{-1})_i^j = \alpha^{0}_{i j } = \sum_{ s\in \mathrm{St}(T), s(v_1)= i, s(w_k)=j} [T,s] (C^{-1})_{s(w_1)}^{s(v_2)} (C^{-1})_{s(w_2)}^{s(v_3)} \ldots (C^{-1})_{s(w_{k-1})}^{v_k}
\\ &= 
 \sum_{\mu_1, \ldots \mu_{2k-2} = \pm} M(\beta_k)^j_{\mu_1} (C^{-1})_{\mu_2}^{\mu_1} M(\beta_{k-1})_{\mu_2}^{\mu_3} \ldots M(\beta_1)_i^{\mu_{2k-2}} = \left( M(\beta_k) C^{-1} M(\beta_{k-1}) C^{-1} \ldots M(\beta_1) \right)_i^j.
 \end{align*}
 This concludes the proof.

\end{proof}

\par Let $\alpha, \beta$ be two non-intersecting oriented arcs. Denote by $a,b,c,d$ the boundary arcs containing $s(\alpha), t(\alpha), s(\beta), t(\beta)$ respectively. 
Reversing the orientation and the height order of $\alpha$ or $\beta$ if necessary, we have ten different possibilities illustrated in Figure \ref{fig_arcrelations}.

\begin{figure}[!h] 
\centerline{\includegraphics[width=12cm]{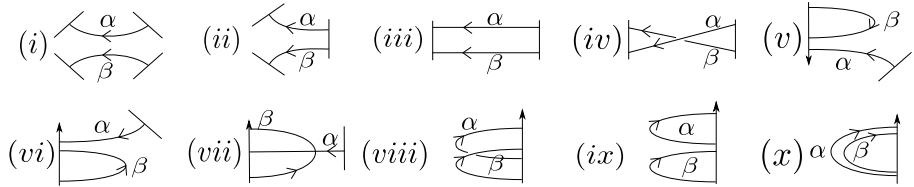} }
\caption{Ten configurations for two non-intersecting oriented arcs.} 
\label{fig_arcrelations} 
\end{figure}

\begin{lemma}\label{lemma_arcsrelations} 
\begin{itemize}

\item[(i)] If the elements of $\{a,b,c,d\}$ are pairwise distinct, one has 
\begin{equation}\label{arcrel1}
N(\alpha) \odot N(\beta) = \tau \left( N(\beta) \odot N(\alpha)\right) \tau.
\end{equation}

\item[(ii)] When $a=c$, $\{a,b,d\}$ has cardinal $3$ and $s(\beta)<_a s(\alpha)$, one has
\begin{equation}\label{arcrel2}
N(\alpha)\odot N(\beta) = \tau (N(\beta) \odot N(\alpha)) \mathscr{R}.
\end{equation}

\item[(iii)] When $a=c\neq b=d$ and $s(\beta)<_a s(\alpha), t(\alpha) <_b t(\beta)$, one has 
\begin{equation}\label{arcrel3}
N(\alpha)\odot N(\beta) = \mathscr{R}^{-1}(N(\beta) \odot N(\alpha)) \mathscr{R}.
\end{equation}

\item[(iv)] When $a=c\neq b=d$ and $s(\beta)<_a s(\alpha), t(\beta) <_b t(\alpha)$,
\begin{equation}\label{arcrel4}
N(\alpha)\odot N(\beta) = \mathscr{R}(N(\beta) \odot N(\alpha)) \mathscr{R}.
\end{equation}

\item[(v)] When $b=c=d\neq a$ and $s(\beta)<_a t(\beta) <_a t(\alpha)$ and $h(s(\beta))<h(t(\beta))$, one has
\begin{equation}\label{arcrel5}
N(\alpha) \odot N(\beta) = \mathscr{R}^{-1} \left( N(\beta) \odot \mathds{1}_2 \right) \mathscr{R} \left( N(\alpha) \odot \mathds{1}_2 \right).
\end{equation}

\item[(vi)] When $b=c=d\neq a$ and $t(\alpha)<_a t(\beta) <_a s(\beta)$ and $h(s(\beta))< h(t(\beta)) < h(t(\alpha))$, one has
\begin{equation}\label{arcrel6}
N(\alpha) \odot N(\beta) = \mathscr{R}^{-1}\left( N(\beta) \odot \mathds{1}_2 \right) \mathscr{R} \left( N(\alpha) \odot \mathds{1}_2 \right).
\end{equation}

\item[(vii)] When $b=c=d\neq a$ and $t(\beta)<_a t(\alpha) <_a s(\beta)$ and $h(s(\beta))<h(t(\alpha))<h(t(\beta))$, one has 
\begin{equation}\label{arcrel7}
N(\alpha) \odot N(\beta) = \mathscr{R} \left( N(\beta) \odot \mathds{1}_2 \right) \mathscr{R} \left( N(\alpha) \odot \mathds{1}_2 \right).
\end{equation}

\item[(viii)] When $a=b=c=d$ and $s(\beta)<_a s(\alpha) <_a t(\beta) <_a t(\alpha)$ and 
\\ $h(s(\beta))< h(s(\alpha))< h(t(\beta))< h(t(\alpha))$, one has
\begin{equation} \label{arcrel8}
\left( \mathds{1}_2 \odot N(\alpha) \right) \mathscr{R}^{-1} \left( \mathds{1}_2 \odot N(\beta) \right) \mathscr{R}^{-1} = 
\mathscr{R} \left( \mathds{1}_2 \odot N(\beta) \right) \mathscr{R}^{-1} \left( \mathds{1}_2 \odot N(\alpha) \right).
\end{equation}

\item[(ix)] When $a=b=c=d$ and $s(\beta) <_a t(\beta) <_a s(\alpha) <_a t(\alpha)$ and 
\\ $h(s(\beta))< h(t(\beta))<h(s(\alpha))<h(t(\alpha))$, one has
\begin{equation}\label{arcrel9}
\mathscr{R}^{-1} \left( \mathds{1}_2 \odot N(\alpha) \right) \mathscr{R} \left( \mathds{1}_2 \odot N(\beta) \right) =
\left( \mathds{1}_2 \odot N(\beta) \right) \mathscr{R}^{-1} \left( \mathds{1}_2 \odot N(\alpha) \right) \mathscr{R}.
\end{equation}

\item[(x)] When $a=b=c=d$ and $s(\alpha)<_a s(\beta) <_a t(\beta) <_a t(\alpha)$, and 
\\ $h(s(\alpha)) < h(s(\beta)) < h(t(\beta)) < h( t( \alpha))$,  one has
\begin{equation}\label{arcrel10}
\left( \mathds{1}_2 \odot N(\alpha) \right) \mathscr{R}^{-1} \left( \mathds{1}_2 \odot N(\beta) \right)\mathscr{R} =
\mathscr{R} \left( \mathds{1}_2 \odot N(\beta) \right) \mathscr{R}^{-1} \left( \mathds{1}_2 \odot N(\alpha) \right).
\end{equation}

\end{itemize}
\end{lemma}

\begin{proof}
Equation \eqref{arcrel1} says that in case $(i)$ any $\alpha_{i j }$ commutes with any $\beta_{k l}$, which is obvious. Equations \eqref{arcrel2}, \eqref{arcrel3}, \eqref{arcrel4} in cases $(ii), (iii)$ and $(iv)$ are straightforward consequences of the height exchange relation \eqref{height_exchange_rel}. All other cases will be derived using the boundary skein relations \eqref{boundary_skein_rel}. As in the proof of Lemma \ref{lemma_height_reversing}, we will color the arcs $\alpha$ and $\beta$ in red and blue to remind the reader that they might be much more complicated than what they look in the picture: in the computations we perform while using the boundary skein relation, we only care about the restriction of the diagrams (depicted in grey) in a small bigon in the neighborhood of the boundary arc $a$ and not of the actual shape of the blue and red parts.

Equations \eqref{arcrel5} and \eqref{arcrel6} in cases $(v)$ and $(vi)$ are proved in a very similar way; we detail the proof of \eqref{arcrel6} and leave \eqref{arcrel7} to the reader. In case $(vi)$, one has:

\begin{align*}
 &\left( M(\alpha) \odot M(\beta) \right)_{kl}^{ij} = \alpha_{k i} \beta_{l j} =  \adjustbox{valign=c}{\includegraphics[width=1.5cm]{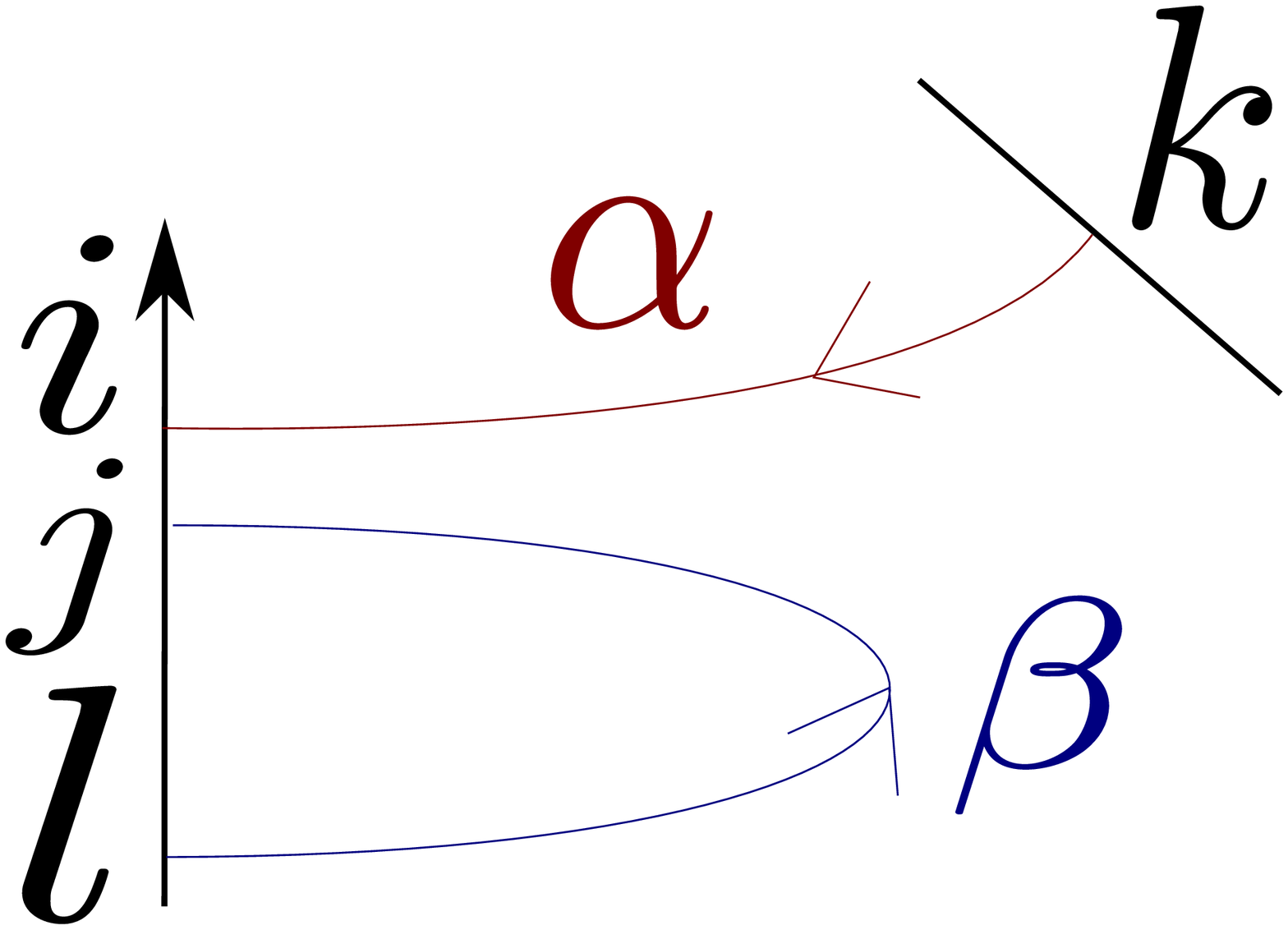}}  = \adjustbox{valign=c}{\includegraphics[width=3cm]{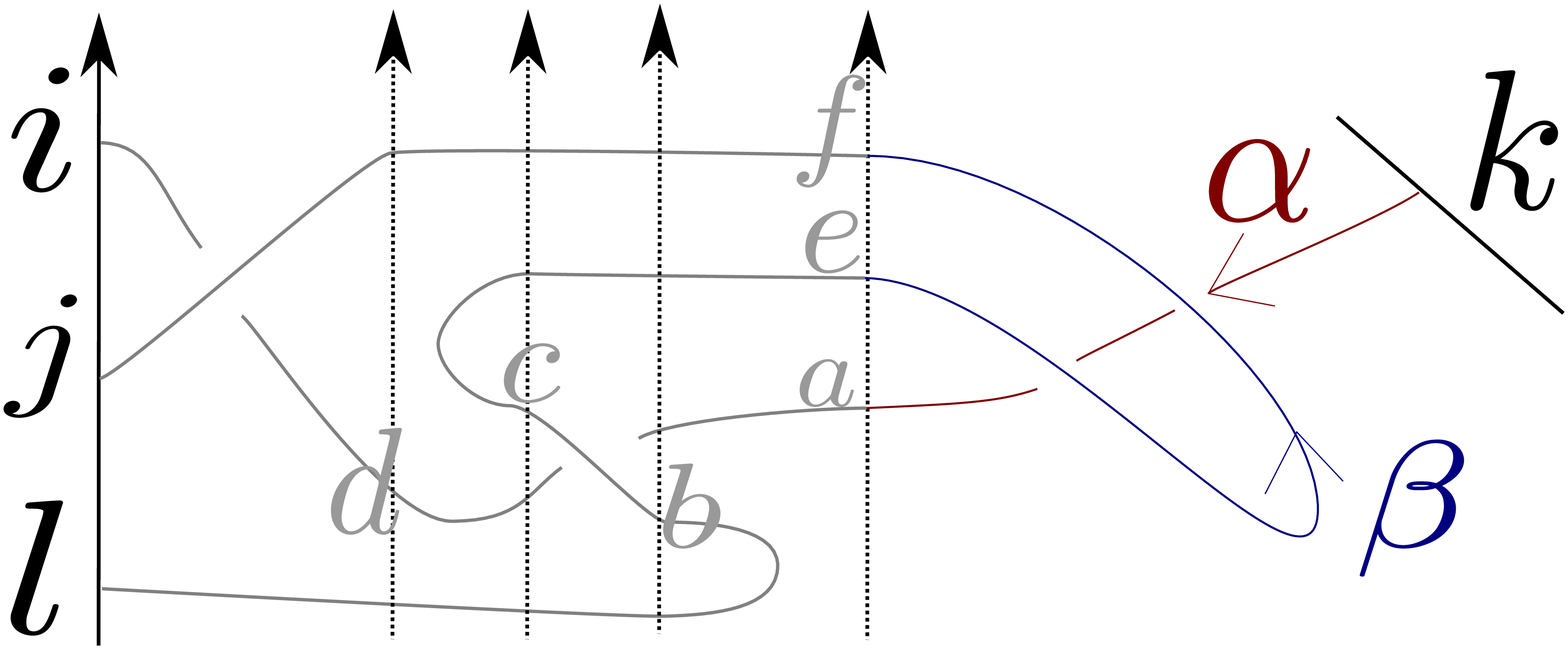}}  
 \\
& = \sum_{a,b,c,d,e,f=\pm} (\mathscr{R}^{-1})_{f d}^{i j}  M(\beta)_e^f C_c^e \mathscr{R}_{ab}^{cd} M(\alpha)_k^a (C^{-1})_l^b = \left( \mathscr{R}^{-1} (M(\beta) C \odot \mathds{1}_2) \mathscr{R}(M(\alpha)\odot C^{-1}) \right)_{kl}^{ij}.
\end{align*}

To handle cases $(vii)$ to $(x)$, we introduce the $4\times 4$ matrix $V=(V_{kl}^{ij})_{i,j,k,l\in \{-,+\}}$, where $V_{kl}^{ij}=[\alpha \cup \beta, \sigma_{ijkl}] \in \mathcal{S}_{\omega}(\mathbf{\Sigma})$ is the class of the simple diagram $\alpha \cup \beta$ with state $\sigma_{ijkl}$ sending $t(\alpha), t(\beta), s(\alpha)$ and $s(\beta)$ to $i,j,k$ and $l$ respectively. Here the height order of the points of $\partial (\alpha \cup \beta)$ is given by the boundary arc orientation drawn in Figure \ref{fig_arcrelations}. The trick is to compute $V$ in two different ways and then equating the two obtained formulas.

\vspace{2mm}
\par In case $(vii)$, on the one hand, we first prove the equality $V=\tau (M(\beta)C \odot \mathds{1}_2) \mathscr{R} (M(\alpha)\odot C^{-1})$ as follows:
$$ V^{ij}_{kl} =  \adjustbox{valign=c}{\includegraphics[width=1.5cm]{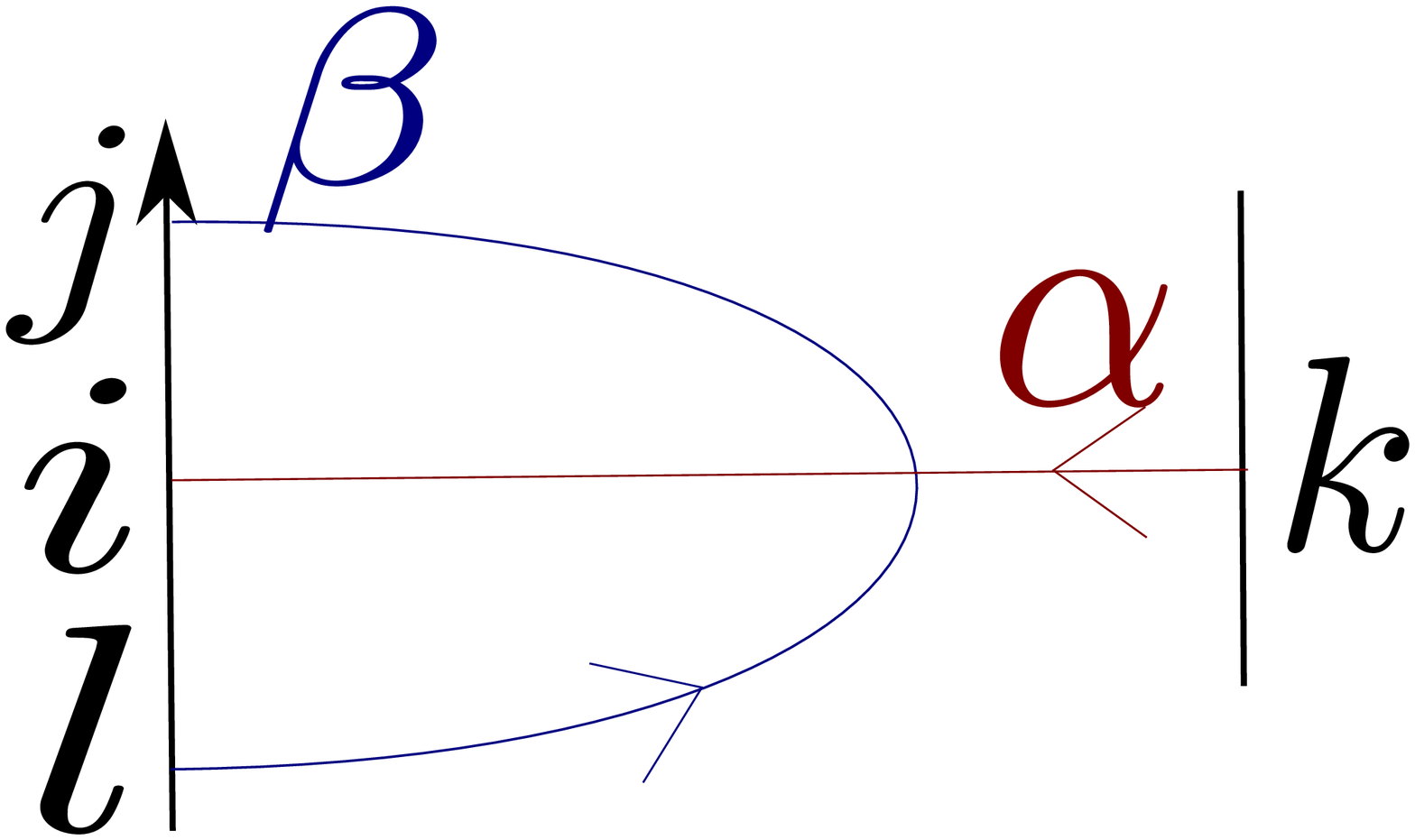}}  =\adjustbox{valign=c}{\includegraphics[width=3cm]{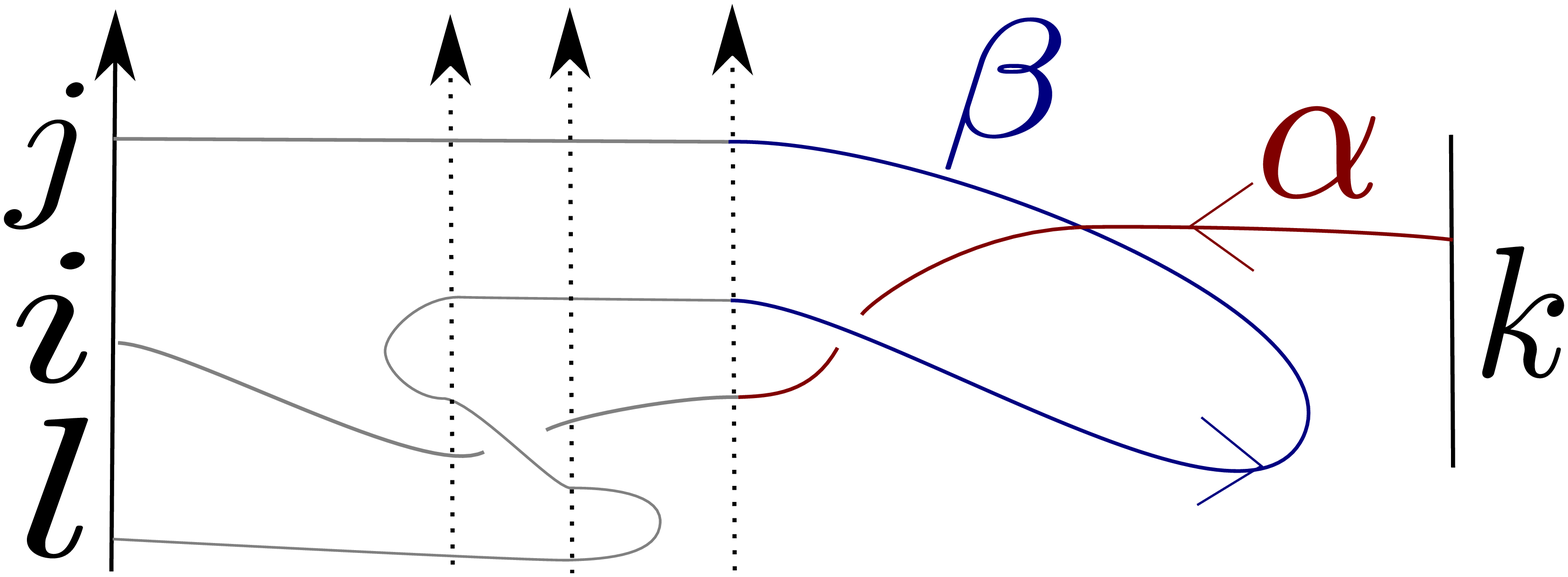}}= \left( (M(\beta)C \odot \mathds{1}_2) \mathscr{R} (M(\alpha)\odot C^{-1}) \right)_{kl}^{ji}.$$
On the other hand, we prove the equality $V= \tau \mathscr{R}^{-1} (M(\alpha) \odot M(\beta))$ as follows:
$$ V^{ij}_{kl} = \adjustbox{valign=c}{\includegraphics[width=1.5cm]{Case_vii_1.eps}}  =  \adjustbox{valign=c}{\includegraphics[width=3cm]{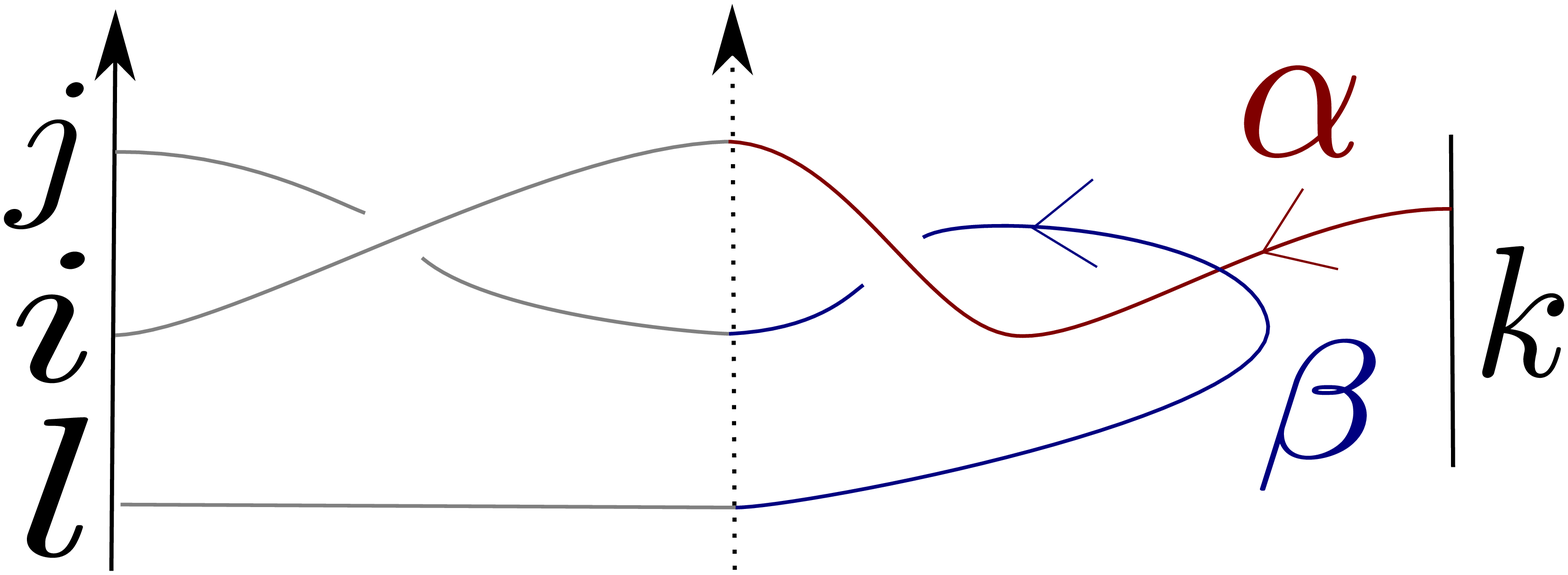}} =\left( \mathscr{R}^{-1} (M(\alpha) \odot M(\beta)) \right)_{kl}^{ji}.$$
So we get the equality $ \mathscr{R}^{-1} (M(\alpha) \odot M(\beta)) = (M(\beta)C \odot \mathds{1}_2) \mathscr{R} (M(\alpha)\odot C^{-1}) (=\tau V)$ and Equation \eqref{arcrel7} follows. 

\vspace{2mm}
\par In case $(vii)$, on the one hand, we first prove the equality $V=\tau (C\odot M(\alpha)) \mathscr{R}^{-1} (\mathds{1}_2 \odot C^{-1}M(\beta))$ as follows:
$$ V^{ij}_{kl} =  \adjustbox{valign=c}{\includegraphics[width=1.5cm]{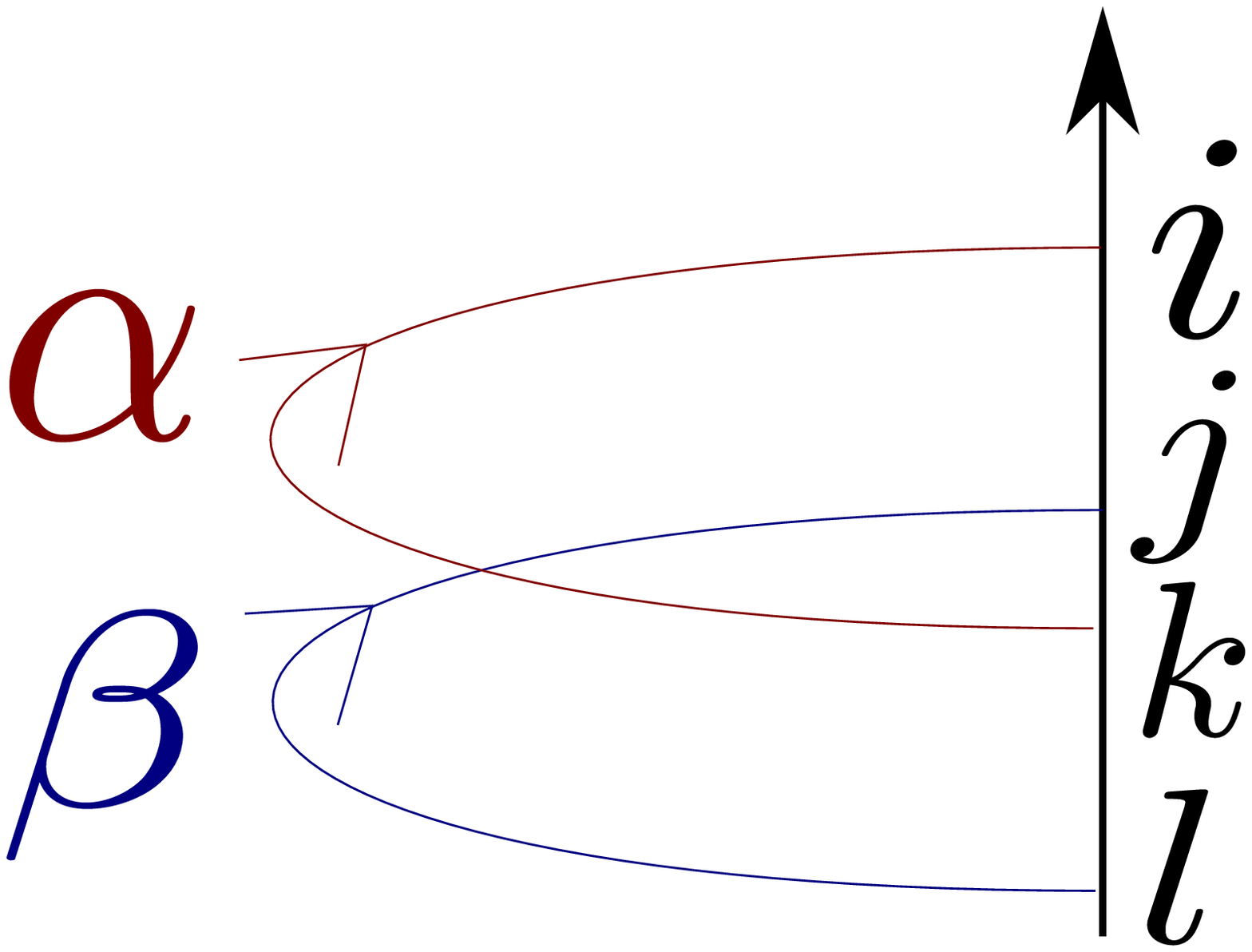}}  =\adjustbox{valign=c}{\includegraphics[width=3cm]{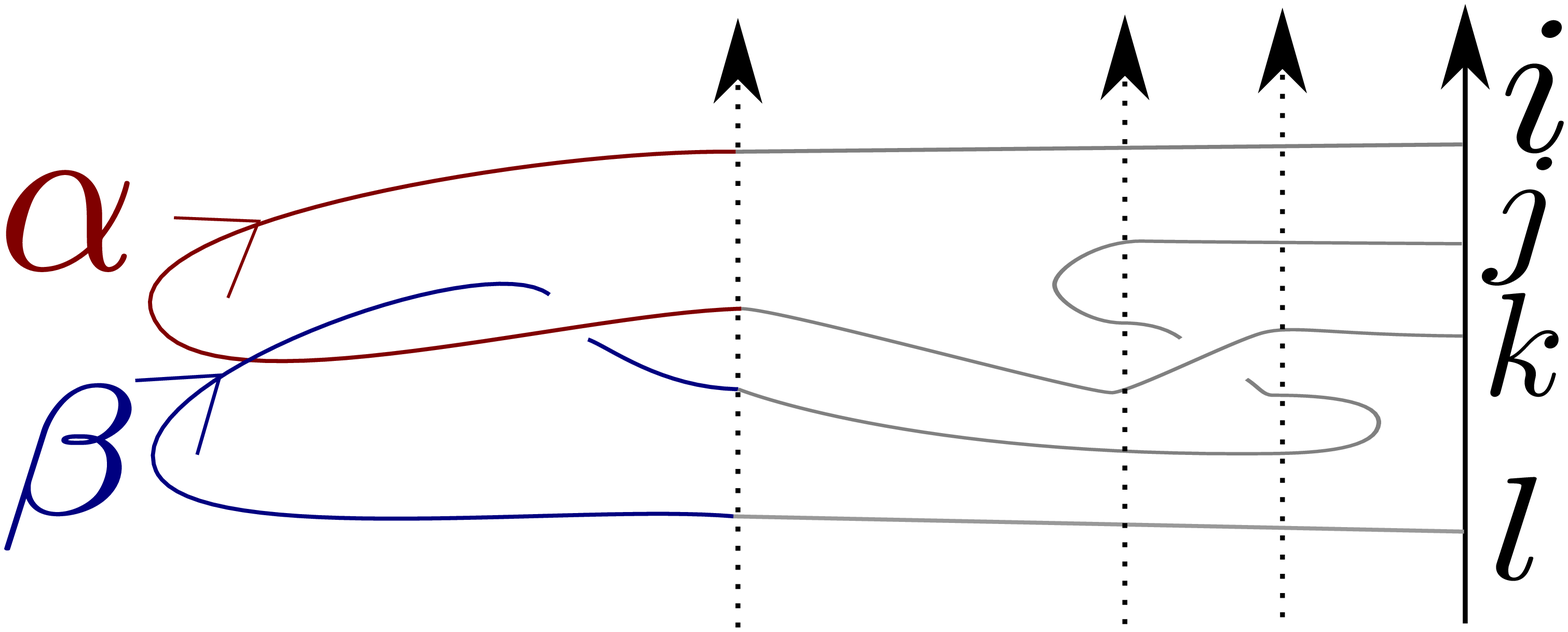}}= \left( (C\odot M(\alpha)) \mathscr{R}^{-1}(\mathds{1}_2 \odot C^{-1}M(\beta)) \right)_{kl}^{ji}.$$
On the other hand, we prove the equality $V= \tau (C\odot C) \mathscr{R} (\mathds{1}_2 \odot C^{-1}M(\beta)) \mathscr{R}^{-1} (\mathds{1}_2 \odot C^{-1}M(\alpha)) \mathscr{R}$ as follows:

$$ V^{ij}_{kl} =  \adjustbox{valign=c}{\includegraphics[width=1.5cm]{Case_viii_1.eps}}  =\adjustbox{valign=c}{\includegraphics[width=3.5cm]{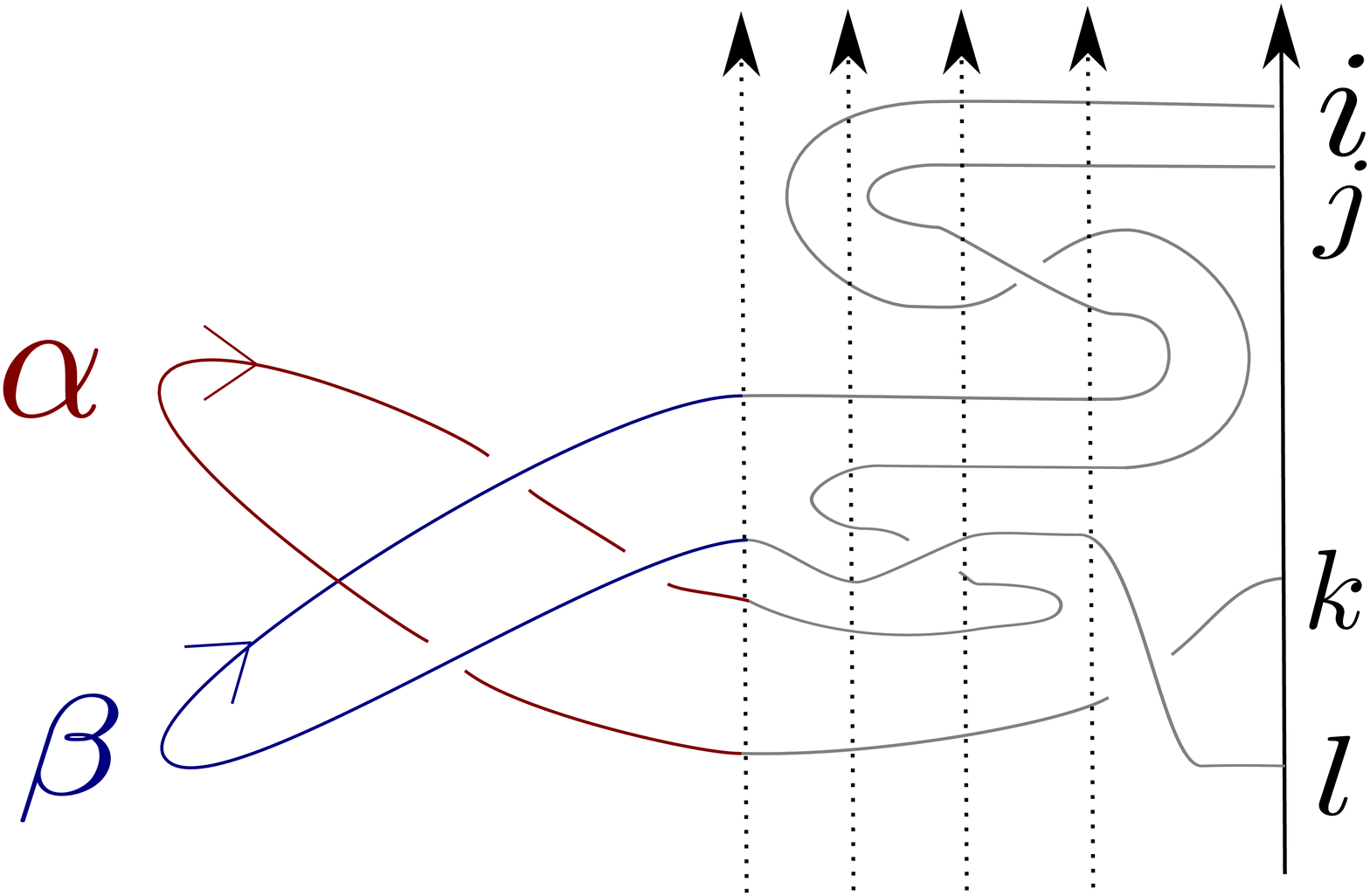}}= \left( (C\odot C) \mathscr{R} (\mathds{1}_2 \odot C^{-1}M(\beta)) \mathscr{R}^{-1} (\mathds{1}_2 \odot C^{-1}M(\alpha)) \mathscr{R} \right)_{kl}^{ji}.$$
Equation \eqref{arcrel8} follows by equating the two obtained expressions for $V$.

\vspace{2mm}
\par In case $(x)$, on the one hand, we first prove the equality $V= (C\odot M(\alpha))\mathscr{R}^{-1} (\mathds{1}_2 \odot C^{-1} M(\beta)) \mathscr{R}$ as follows:

$$ V_{kl}^{ij} =  \adjustbox{valign=c}{\includegraphics[width=1.5cm]{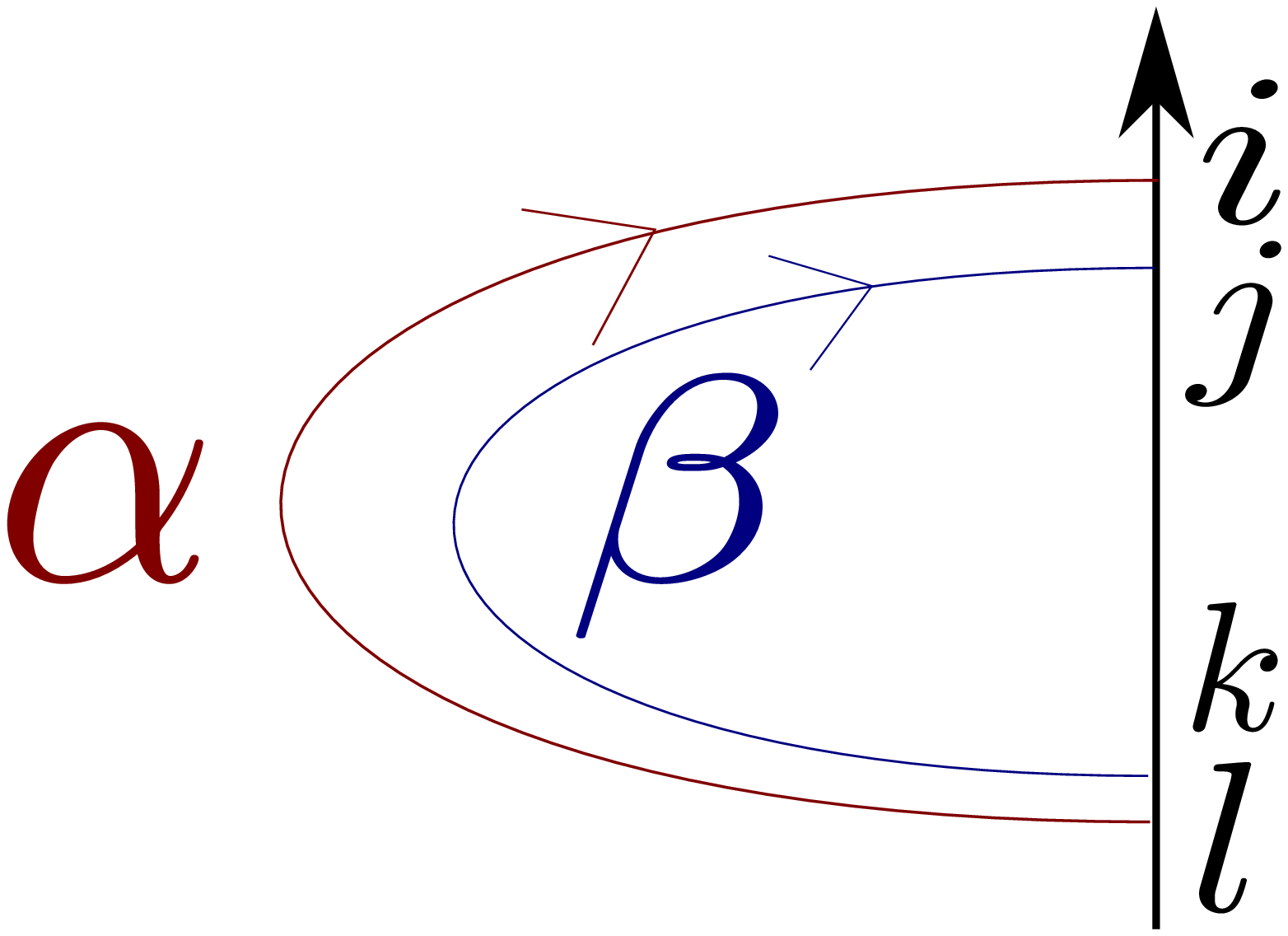}}  =\adjustbox{valign=c}{\includegraphics[width=3.5cm]{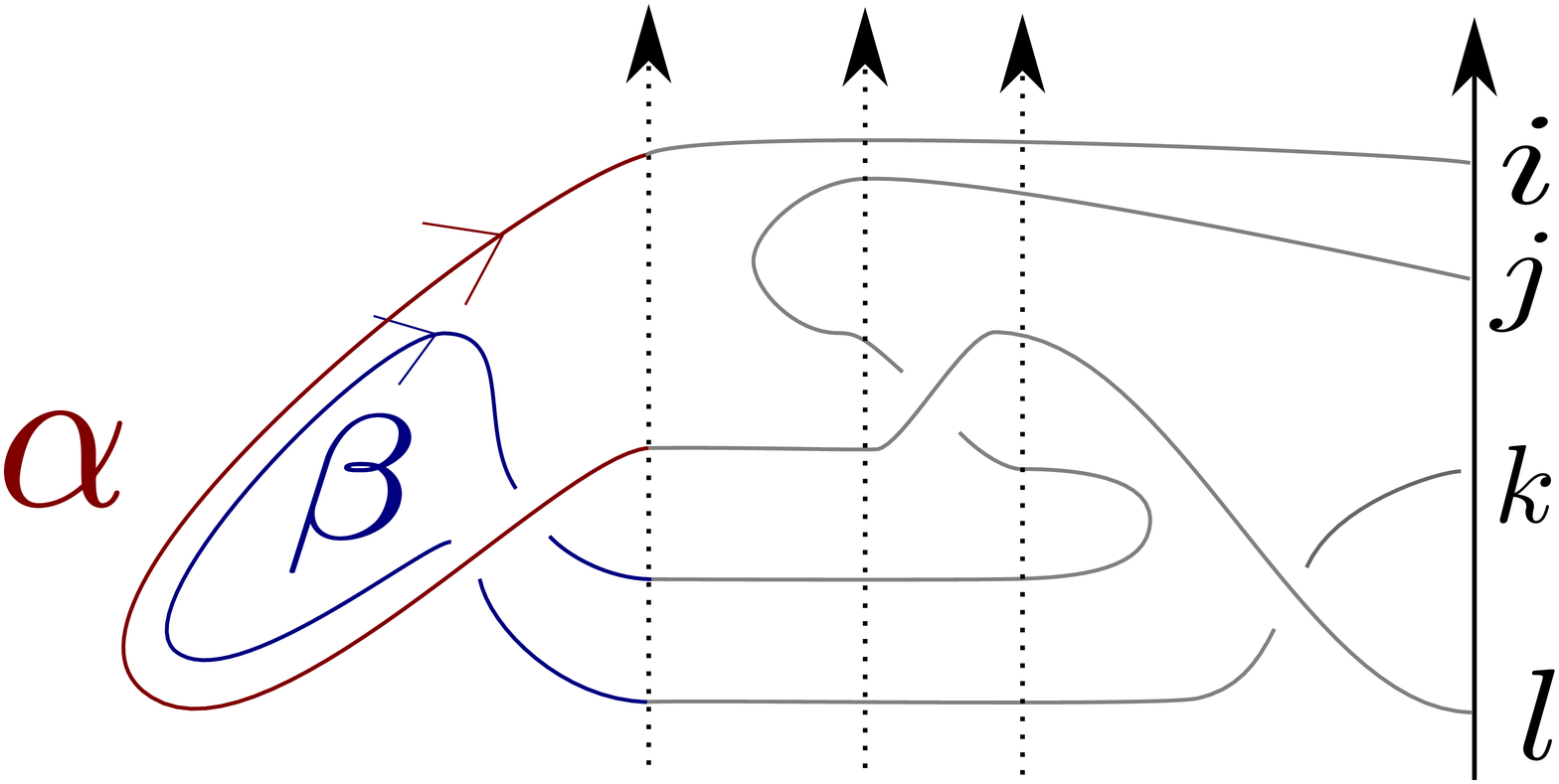}}= \left( (C\odot M(\alpha))\mathscr{R}^{-1} (\mathds{1}_2 \odot C^{-1} M(\beta)) \mathscr{R} \right)_{kl}^{ij}.$$

On the other hand, we prove the equality $V=(C\odot C) \mathscr{R} (\mathds{1}_2 \odot C^{-1}M(\beta)) \mathscr{R}^{-1}(\mathds{1}_2 \odot C^{-1}M(\alpha))$ as follows:

$$ V_{kl}^{ij} = \adjustbox{valign=c}{\includegraphics[width=1.5cm]{Case_x_1.eps}}  =\adjustbox{valign=c}{\includegraphics[width=4cm]{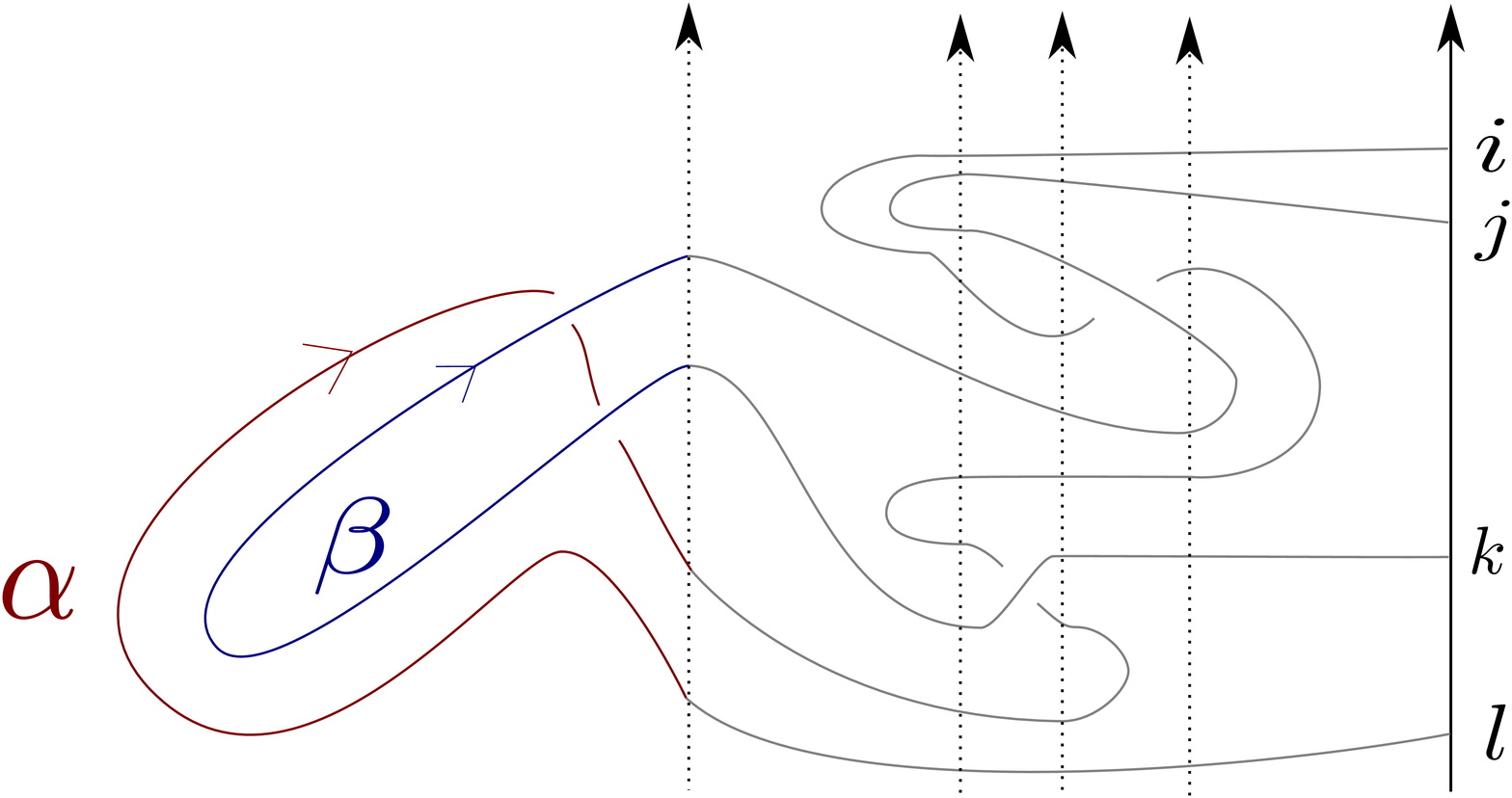}}= \left( (C\odot C) \mathscr{R} (\mathds{1}_2 \odot C^{-1}M(\beta)) \mathscr{R}^{-1}(\mathds{1}_2 \odot C^{-1}M(\alpha)) \right)_{kl}^{ij}.$$

Therefore, we obtain the following equality that will be used in the proof of Lemma \ref{lemma_qdet_rel}: 
\begin{equation}\label{eq_V}
 V = (C\odot M(\alpha))\mathscr{R}^{-1} (\mathds{1}_2 \odot C^{-1} M(\beta)) \mathscr{R} = (C\odot C) \mathscr{R} (\mathds{1}_2 \odot C^{-1}M(\beta)) \mathscr{R}^{-1}(\mathds{1}_2 \odot C^{-1}M(\alpha)).
 \end{equation}
 Equation \eqref{arcrel10} follows.
 
\vspace{2mm}
\par In case $(ix)$, we slightly change the strategy. We define the $4\times 4$ matrix $W= (W_{kl}^{ij})_{i,j,k,l\in \{-,+\}}$ by $W_{kl}^{ij}:=  \adjustbox{valign=c}{\includegraphics[width=1.5cm]{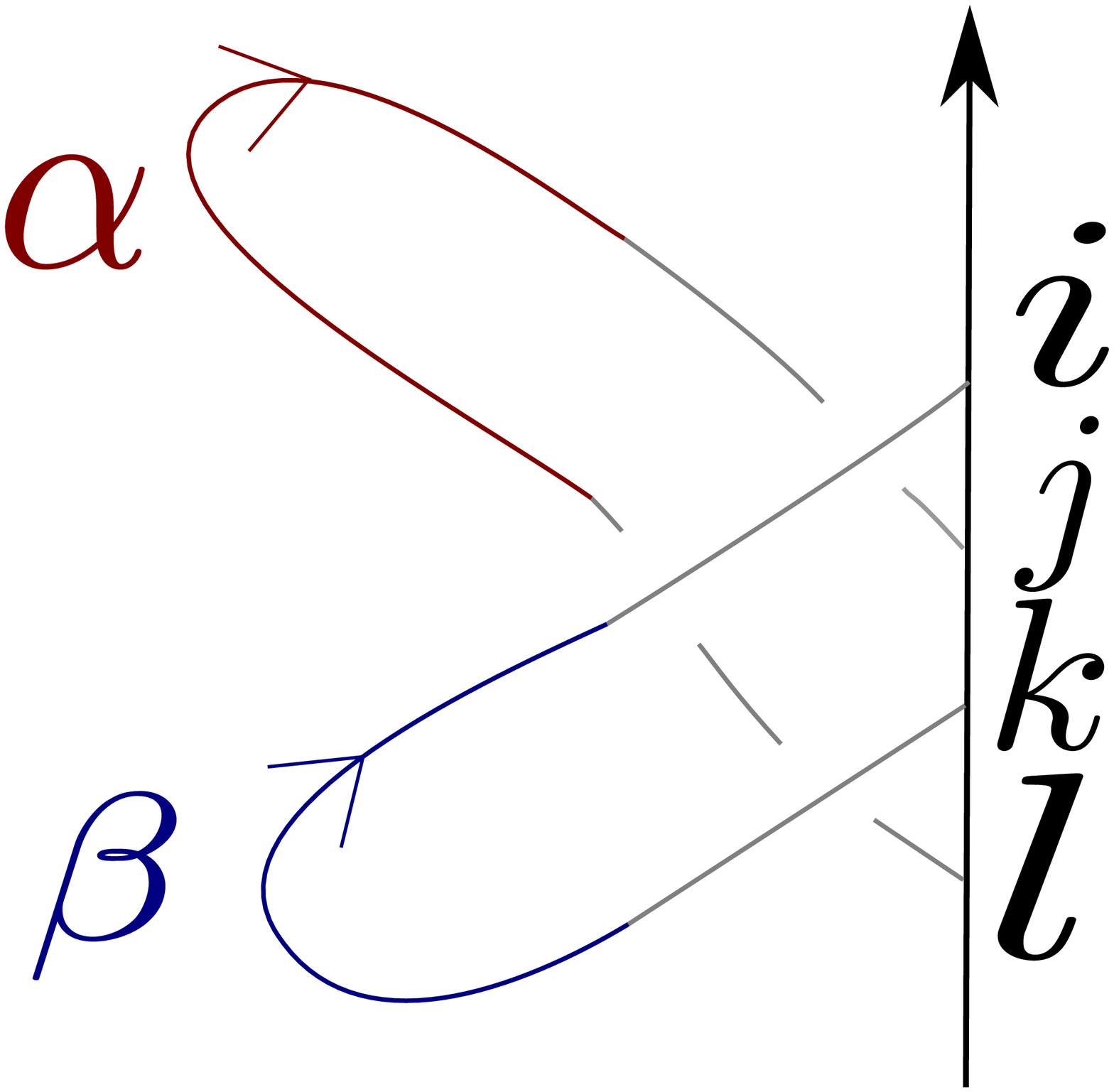}} $. We first prove the equality $W= (C\odot M(\beta)) \mathscr{R}^{-1}(\mathds{1}_2 \odot C^{-1}M(\alpha))$ as follows:

$$ W_{kl}^{ij}=  \adjustbox{valign=c}{\includegraphics[width=1.5cm]{Case_ix_1.eps}} = \adjustbox{valign=c}{\includegraphics[width=3.5cm]{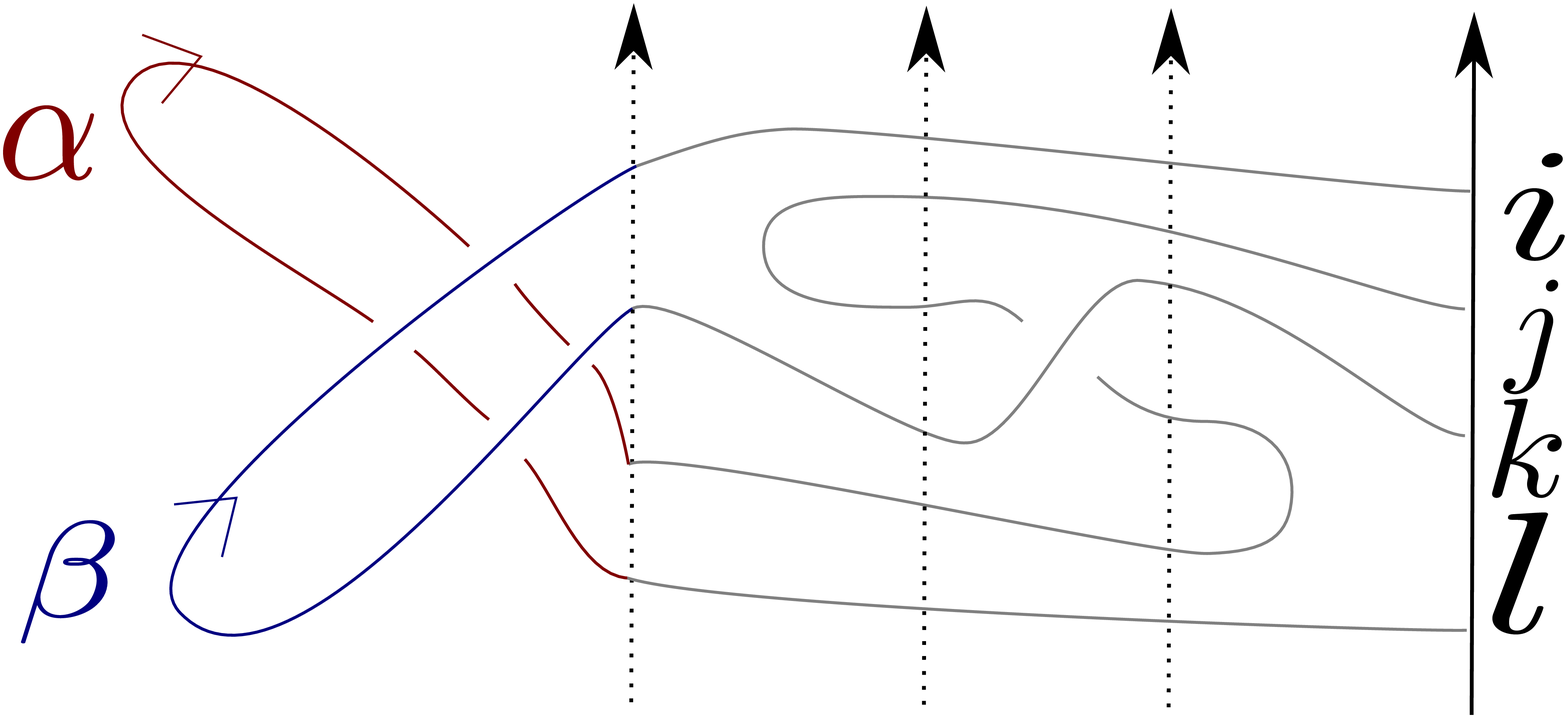}} = \left( (C\odot M(\beta)) \mathscr{R}^{-1}(\mathds{1}_2 \odot C^{-1}M(\alpha)) \right)_{kl}^{ij}.$$
Next, we prove the equality $W= (C\odot C) \mathscr{R}^{-1}(\mathds{1}_2\odot C^{-1}M(\alpha)) \mathscr{R} (\mathds{1}_2\odot C^{-1}M(\beta)) \mathscr{R}^{-1}$ as follows: 

$$ W_{kl}^{ij}=  \adjustbox{valign=c}{\includegraphics[width=1.5cm]{Case_ix_1.eps}} = \adjustbox{valign=c}{\includegraphics[width=3.5cm]{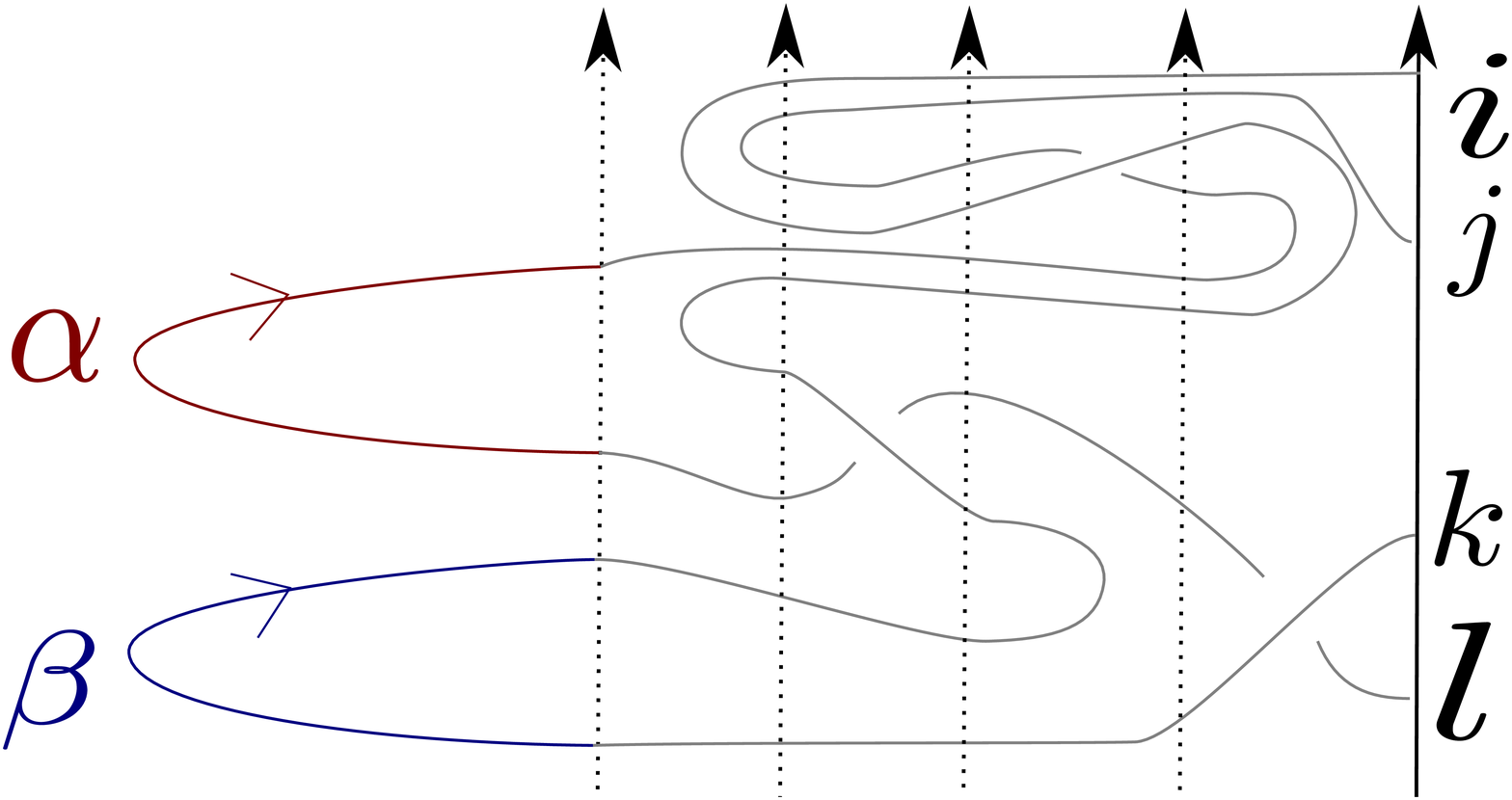}} = \left(  (C\odot C) \mathscr{R}^{-1}(\mathds{1}_2\odot C^{-1}M(\alpha)) \mathscr{R} (\mathds{1}_2\odot C^{-1}M(\beta)) \mathscr{R}^{-1} \right)_{kl}^{ij}.$$

Equation \eqref{arcrel9} follows by equating the two obtained expressions for $W$. This concludes the proof.

\end{proof}

\begin{lemma}[q-determinant relations]\label{lemma_qdet_rel}
Let $\alpha$ be an oriented arc. Then 
\begin{equation}\label{qdet_rel}
\mathrm{det}_q (N(\alpha))=1 \mbox{, if }\alpha \mbox{ is of type a, and } \mathrm{det}_{q^2}(N(\alpha))= 1 \mbox{, else}.
\end{equation}
\end{lemma}

\begin{proof}
First suppose that $\alpha$ is of type a. Applying the trivial arc and cutting arc relation, we obtain: 

$$ (C^{-1})_+^- =  \bigonheightcurveright{-}{+} = (C^{-1})^+_-   \adjustbox{valign=c}{\includegraphics[width=1.2cm]{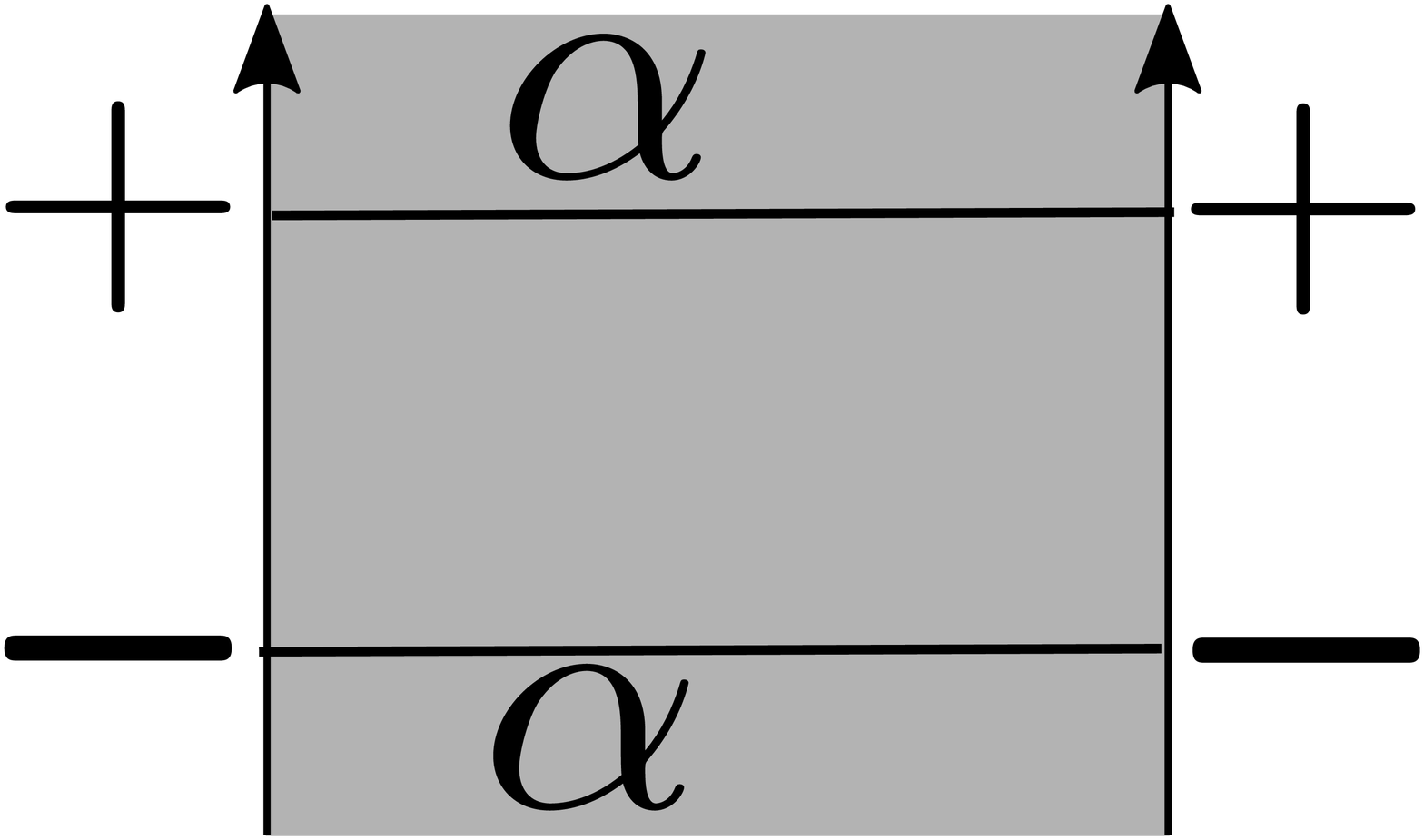}}   + (C^{-1})^-_+ \adjustbox{valign=c}{\includegraphics[width=1.2cm]{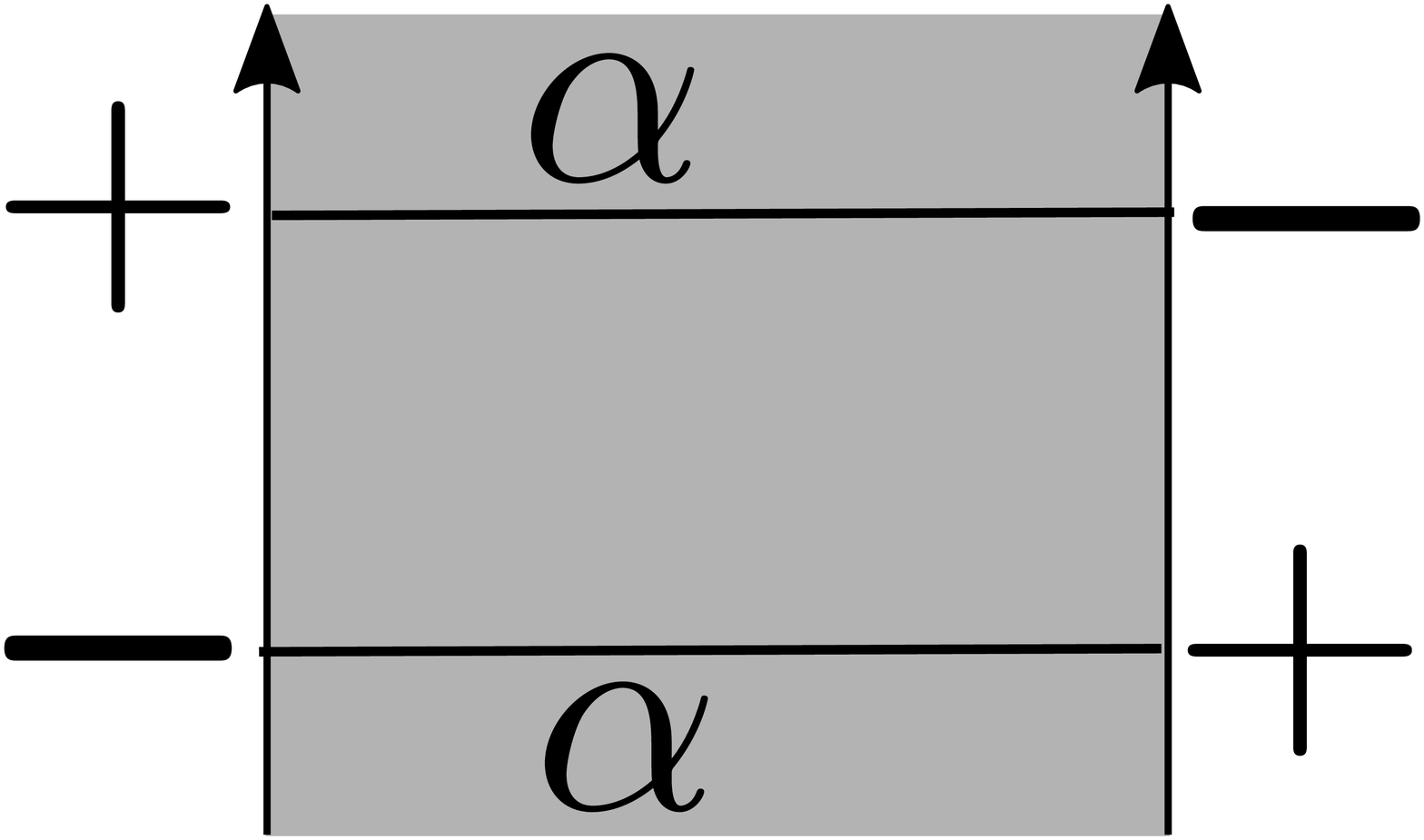}} , $$ 

which is equivalent to the equation $\alpha_{++} \alpha_{--} - q^{-1}\alpha_{+-} \alpha_{-+} = 1$ as claimed. Next we suppose that $\alpha$ is of type d. Let $\beta$ be an arc isotope to and disjoint from  $\alpha$, placed as in the configuration (x) of Figure \ref{fig_arcrelations}.  Consider the matrix $V=(V_{kl}^{ij})_{i,j,k,l\in \{-,+\}}$, where $V_{kl}^{ij}=[\alpha \cup \beta, \sigma_{ijkl}] \in \mathcal{S}_{\omega}(\mathbf{\Sigma})$ is the class of the simple diagram $\alpha \cup \beta$ with state $\sigma_{ijkl}$ sending $t(\alpha), t(\beta), s(\alpha)$ and $s(\beta)$ to $i,j,k$ and $l$ respectively, like in the proof of Lemma \ref{lemma_arcsrelations} (i.e. $V_{kl}^{ij}=  \adjustbox{valign=c}{\includegraphics[width=1.5cm]{Case_x_1.eps}}$). Again, using the trivial arc and cutting arc relation, we obtain: 

$$ C_-^+ =  \adjustbox{valign=c}{\includegraphics[width=1.5cm]{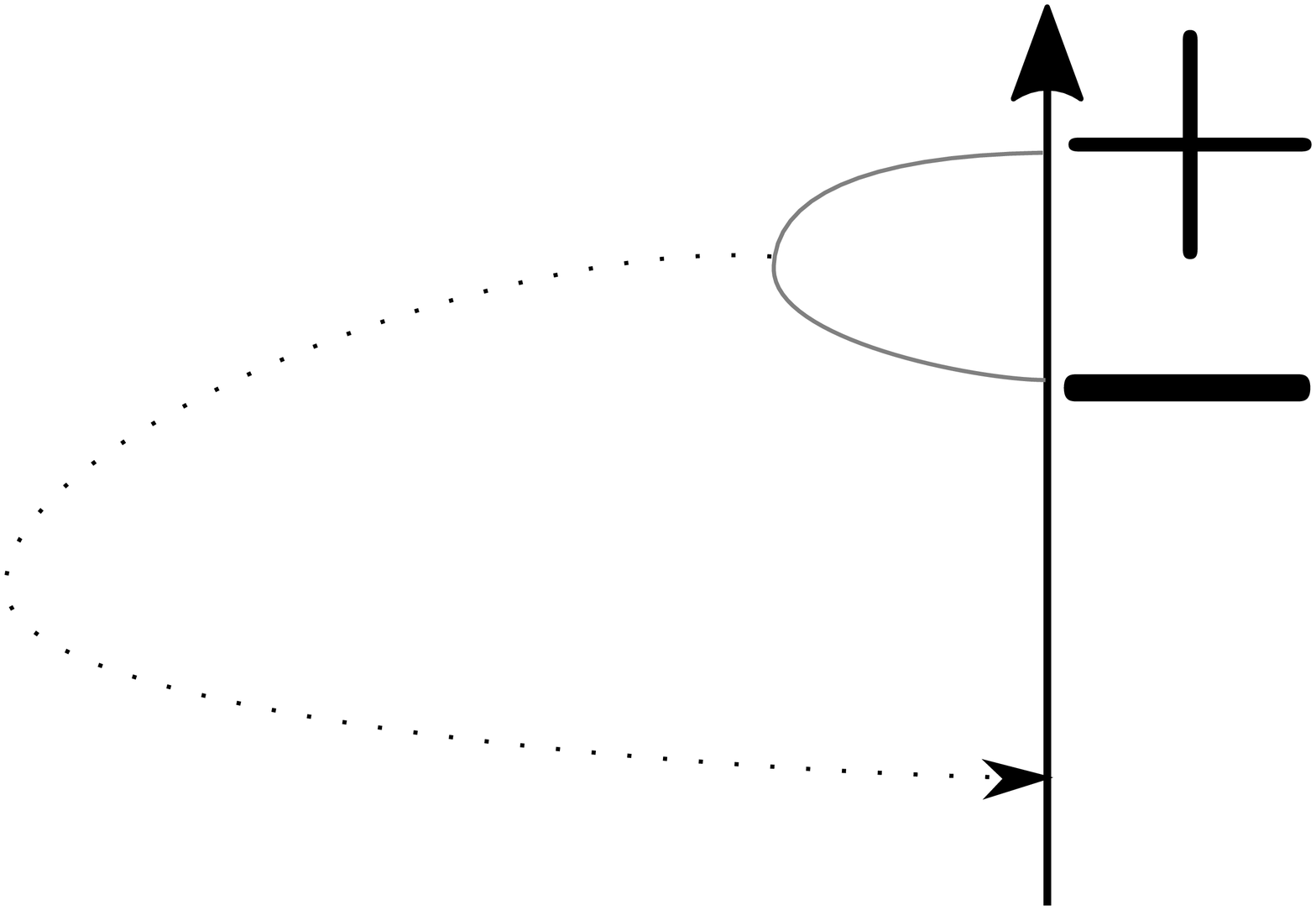}} = C_+^-  \adjustbox{valign=c}{\includegraphics[width=1.5cm]{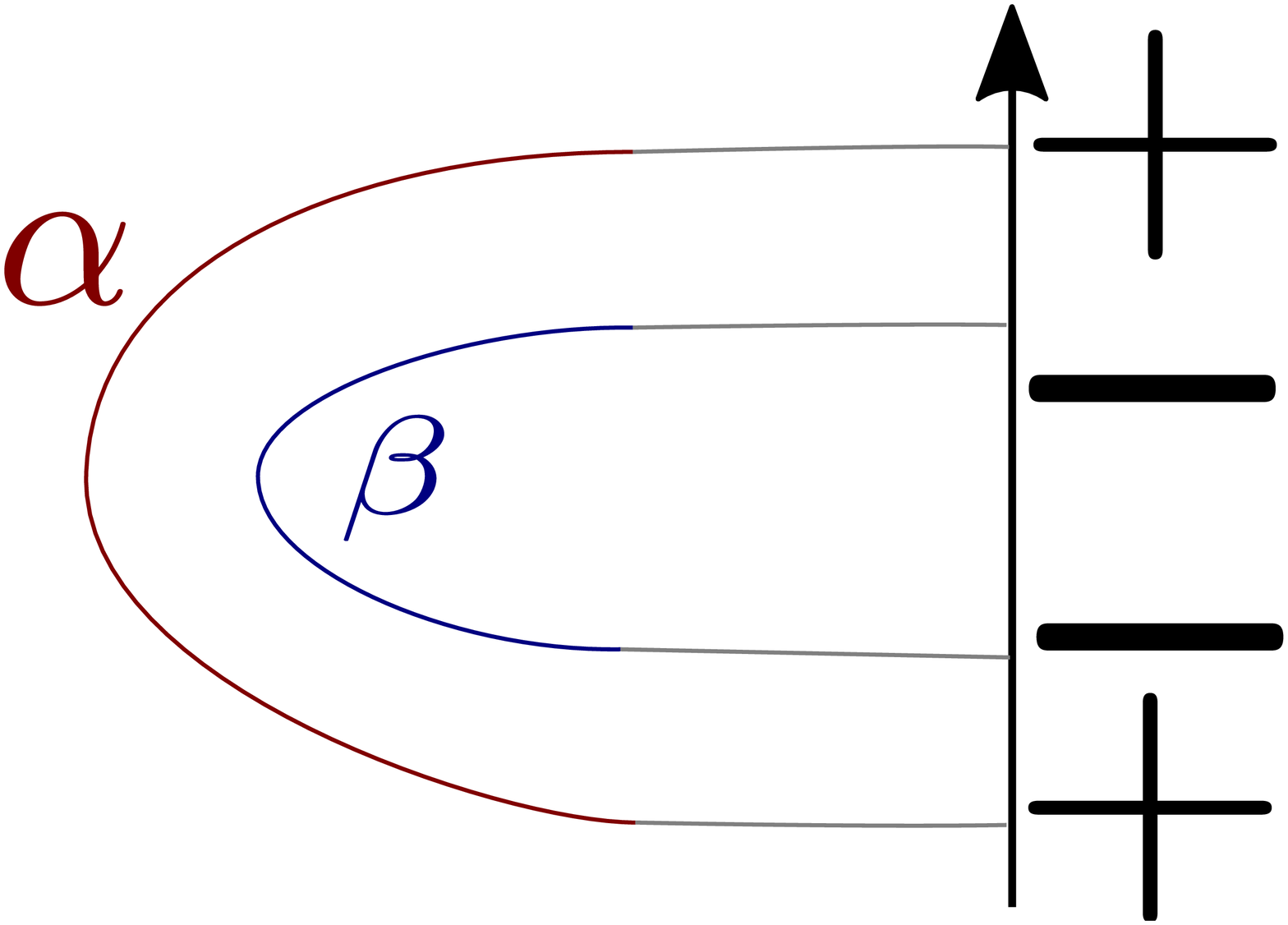}} + C_-^+  \adjustbox{valign=c}{\includegraphics[width=1.5cm]{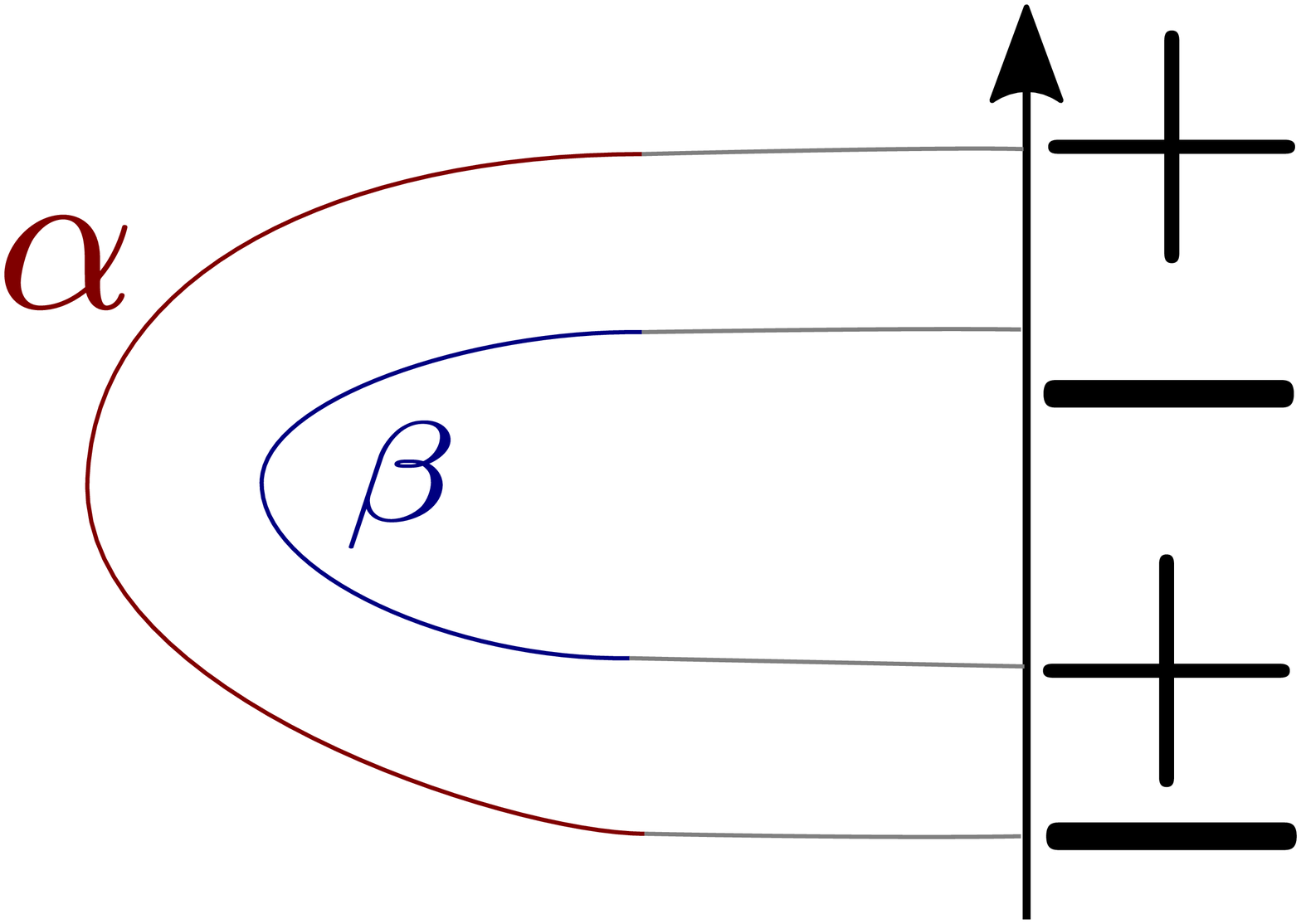}} 
 \quad \Leftrightarrow  V^{+-}_{-+} - q V^{+-}_{+-} = 1.
$$

Next, by developing the matrix coefficients in the equalities \eqref{eq_V}, we find the equalities
$$ V^{+-}_{-+} = q \alpha_{+-} \alpha_{-+} + A^{-1} \quad \mbox{and} \quad V^{+-}_{+-} = q^{-2} \alpha_{--}\alpha{++} - A^{-3}.$$ 
Putting these equalities together, we find

\begin{equation}\label{eq_qdet_typed}
 \alpha_{--} \alpha_{++} - q^2 \alpha_{+-} \alpha_{-+} = A.
 \end{equation}
 
Now developing Equation \eqref{arcrel10}, we obtain $\alpha_{+-} \alpha_{-+} = \alpha_{-+} \alpha_{+-}$, therefore: 
$$ \mathrm{det}_{q^2} (N(\alpha)) = \mathrm{det}_{q^2} 
\begin{pmatrix} - \omega^{-5} \alpha_{-+} & -\omega^{-5} \alpha_{--} \\ \omega^{-1} \alpha_{++} & \omega^{-1} \alpha_{+-} \end{pmatrix}
 = - A^3 \alpha_{-+} \alpha_{+-} + A^{-1} \alpha_{--} \alpha_{++} = 1.$$
 
\par Now, if $\alpha$ is of type e, then $\alpha^{-1}$ is of type d. A simple computation shows that if $M= \begin{pmatrix} a & b \\ c & d \end{pmatrix}$ is such that $ad=da$ then $\mathrm{det}_{q^2} ( M ) = \mathrm{det}_{q^2} \left( C^{-1} {}^tM C \right)$, so we deduce the q-determinant formula for $\alpha$ of type e from the facts that it holds for $\alpha^{-1}$, from the orientation reversing formula in Lemma \ref{lemma_orientation_reversing} and from the equality $\alpha_{+-} \alpha_{-+} = \alpha_{-+} \alpha_{+-}$.

\par Suppose that $\alpha$ is of type c and choose $\mathds{k}=\mathbb{Z}[\omega^{\pm 1}]$. Recall from Section $2.1$ the reflexion anti-involution $\theta$. The image $\theta(\alpha)$ is of type d, so applying $\theta$ to Equation \eqref{eq_qdet_typed}, we obtain that:
\begin{equation}\label{eq_qdet_typec}
  \alpha_{++} \alpha_{--} - q^{-2} \alpha_{-+} \alpha_{+-} = A^{-1}. 
  \end{equation} 
By Remark \ref{remark_change_basis}, since Equation \ref{eq_qdet_typec} holds for $\mathds{k}=\mathbb{Z}[\omega^{\pm 1}]$, it also holds for any other ring.  Also using $\theta$, we find that  $\alpha_{+-} \alpha_{-+} = \alpha_{-+} \alpha_{+-}$ and the equation $\mathrm{det}_{q^2}(N(\alpha))=1$ follows. Eventually, when $\alpha$ is of type b, we deduce the q-determinant relation from the facts that it holds for $\alpha^{-1}$ (of type c), from the orientation reversing formulas of Lemma \ref{lemma_orientation_reversing} and from the identity $\alpha_{+-} \alpha_{-+} = \alpha_{-+} \alpha_{+-}$.
\end{proof}

\begin{definition}
Let $\mathbb{P}=(\mathbb{G}, \mathbb{RL})$ be a finite presentation of $\Pi_1(\Sigma_{\mathcal{P}}, \mathbb{V})$. By Proposition \ref{prop_generators}, the set $\mathcal{A}^{\mathbb{G}}$ generates $\mathcal{S}_{\omega}(\mathbf{\Sigma})$ and we have found three families of relations: 
\begin{enumerate}
\item For each $\alpha\in \mathbb{G}$ we have the either the relation $\mathrm{det}_q(N(\alpha))=1$ or $\mathrm{det}_{q^2} (N(\alpha))=1$  by Equation \eqref{qdet_rel} in Lemma \ref{lemma_qdet_rel}; we call them the \textit{q-determinant relations}.
\item For each $R\in \mathbb{RL}$, we have four relations obtained by considering the matrix coefficients in Equation \eqref{eq_trivial_loops_rel} in Lemma \ref{lemma_trivial_loops_relations}; we call them \textit{trivial loops relations}.
\item For each pair $(\alpha, \beta)$ of elements in $\mathbb{G}$, we have $16$ relations obtained by considering the matrix coefficients in one of the Equations \eqref{arcrel1}, $\ldots$, \eqref{arcrel10} of Lemma \ref{lemma_arcsrelations} after having possibly replaced $\alpha$ or $\beta$ by $\alpha^{-1}$ or $\beta^{-1}$, if necessary, and using the inversion formula \eqref{eq_inversion}; we call them \textit{arcs exchange relations}.
\end{enumerate}
 \end{definition}
 
 \section{Proof of Theorems \ref{theorem1} and \ref{theorem2}}\label{sec_proof}
 
  In this section, we prove Theorems \ref{theorem1} and \ref{theorem2}.  Let $\mathcal{L}_{\omega}(\mathbb{P})$ be the algebra generated by the elements of $\mathbb{G}$ modulo the q-determinant, trivial loops and arcs exchange relations and write $\Psi : \mathcal{L}_{\omega}(\mathbb{P})\rightarrow \mathcal{S}_{\omega}(\mathbf{\Sigma})$ the obvious algebra morphism. By Proposition \ref{prop_generators}, $\Psi$ is surjective and we need to show that $\Psi$ is injective to prove Theorem \ref{theorem1}. We cut the proof of Theorem \ref{theorem1} in three steps: $(1)$ first in Step $1$, we show that it is sufficient to make the proof in the case where $\mathbb{P}$ has no relation (as in Example \ref{exemple_pres}); $(2)$ in this particular case, the finite presentation defining $\mathcal{L}_{\omega}(\mathbb{P})$ is inhomogeneous quadratic  and we will use the Diamond Lemma to extract PBW bases of $\mathcal{L}_{\omega}(\mathbb{P})$ and to prove it is Koszul; in Step $2$ we extract the re-written rules and their leading terms from the q-determinant and arc exchange relations and exhibit the associated spanning family $\underline{\mathcal{B}}^{\mathbb{G}}\subset\mathcal{L}_{\omega}(\mathbb{P})$; $(3)$ eventually in Step $3$, we show that the image by $\Psi$ of $\underline{\mathcal{B}}^{\mathbb{G}}$ is a basis,  this will prove both the injectivity of $\Psi$ and the fact that $\underline{\mathcal{B}}^{\mathbb{G}}$ is a Poincar\'e-Birkhoff-Witt basis and conclude the proofs of Theorems \ref{theorem1} and \ref{theorem2}.

\subsection{Step $1$: Reduction to the case where $\mathbb{P}$ has no relation}

 Let $\Gamma$ be the presenting graph of $\mathbb{P}$ and consider its fundamental groupoid $\Pi_1(\Gamma)$: the objects of $\Pi_1(\Gamma)$ are the vertices of $\Gamma$ (\textit{i.e.} the set $\mathbb{V}$) and the morphisms are compositions $\alpha_k^{\varepsilon_k} \ldots \alpha_1^{\varepsilon_1}$ where $\alpha_i \in \mathbb{G}$. The inclusion $\Gamma \subset \Sigma_{\mathcal{P}}$ induces a functor $F: \Pi_1(\Gamma) \rightarrow \Pi_1(\Sigma_{\mathcal{P}}, \mathbb{V})$ which is the identity on the objects. The fact that $\mathbb{G}$ is a set of generators implies that $F$ is full and $\mathbb{P}$ has no relations if and only if $F$ is faithful. Fix $v_0 \in \mathbb{V}$. For a relation $R\in \mathbb{RL}$ of the form $R= \beta_k \star \ldots \star \beta_1$, the \textit{base point of} $R$ is $s(\beta_1)=t(\beta_k)$. By inspecting the trivial loop relation \eqref{eq_trivial_loops_rel}, we see that  changing a relation $R$ by a relation $\beta \star R \star \beta^{-1}$ does not change the algebra $\mathcal{L}_{\omega}(\mathbb{P})$. Since $\Sigma_{\mathcal{P}}$ is assumed to be connected, we can suppose that  all relations in $\mathbb{RL}$ have the same base point $v_0$, so each relation $R= \beta_k \star \ldots \star \beta_1$ induces an element $[R]=\beta_k \ldots \beta_1 \in \pi_1(\Gamma, v_0)$.
 The functor $F$ induces a surjective group morphism $F_{v_0} : \pi_1(\Gamma, v_0) \rightarrow \pi_1(\Sigma_{\mathcal{P}}, v_0)$ and the fact that $\mathbb{RL}$ is a set of relations implies that $\{ [R], R\in \mathbb{RL} \}$ generates $\mathrm{ker} (F_{v_0})$. Since $\pi_1(\Gamma, v_0)$ is a free group, so is $\mathrm{ker} (F_{v_0})$. Let $R_1, \ldots, R_m\in \mathbb{RL}$ be such that $\{[R_1], \ldots, [R_m]\}$ is a minimal set of generators for the free group $\mathrm{ker} (F_{v_0})$. For each $R_i$, choose an element $\beta_i \in \mathbb{G}$ such that either $\beta_i$ or $\beta_i^{-1}$ appears in the expression of $R_i$ such that the set $\mathbb{G}'$ obtained from $\mathbb{G}$ by removing the $\beta_i$'s is a generating set. So if $\Gamma'$ is the presenting graph of $\mathbb{G}'$, the morphism $F'_{v_0} : \pi_1(\Gamma', v_0) \rightarrow \pi_1(\Sigma_{\mathcal{P}}, v_0)$ is injective, so the functor $F' : \Pi_1(\Gamma') \rightarrow \Pi_1(\Sigma_{\mathcal{P}}, \mathbb{V})$ is faithful and $\mathbb{P}':=(\mathbb{G}', \emptyset)$ is a finite presentation of $\Pi_1(\Sigma_{\mathcal{P}}, \mathbb{V})$ with no relations. 
 
 \par The inclusion $\mathbb{G}'\subset \mathbb{G}$ induces an algebra morphism $\widetilde{\varphi} : \mathcal{T}[\mathbb{G}'] \hookrightarrow \mathcal{T}[\mathbb{G}]$ on the free tensor algebras generated by $\mathbb{G}'$ and $\mathbb{G}$ respectively and $\widetilde{\varphi}$ sends q-determinant and arc exchange relations to q-determinant and arc exchange relations, so it induces an algebra morphism 
 $$\varphi : \mathcal{L}_{\omega}(\mathbb{P}') \rightarrow \mathcal{L}_{\omega}(\mathbb{P})$$.
 
 \begin{lemma}\label{lemma_reduction}
 The morphism $\varphi$ is an isomorphism.
 \end{lemma}
 
 \begin{proof} To prove the surjectivity, we need to show that for each removed path $\beta_i \in \mathbb{G} \setminus \mathbb{G}'$, the stated arcs $(\beta_i)_{\varepsilon \varepsilon'}$ can be expressed as a polynomial in the stated arcs $(\alpha^{\pm 1})_{\mu \mu'}$ for $\alpha \in \mathbb{G}'$. This follows from the trivial loop relation \eqref{eq_trivial_loops_rel} associated to the relation $R_i \in \mathbb{RL}$ containing $\beta_i^{\pm 1}$. The injectivity of $\varphi$ is a straightforward consequence of the definition. 
 \end{proof}

\subsection{Step $2$: Poincar\'e-Birkhoff-Witt bases and Koszulness}\label{sec_Koszul}

\begin{convention}
In the rest of the section, we now suppose that $\mathbb{P}=(\mathbb{G}, \emptyset)$ is a presentation with no relations and that every arc in $\mathbb{G}$ is either of type $a,c$ or $d$.
\end{convention}
Note that the convention on the type of the generators is not restrictive but purely conventional since we can always replace a generator $\alpha$ by $\alpha^{-1}$ without changing the set $\mathcal{A}^{\mathbb{G}}$ of generators of $\mathcal{S}_{\omega}(\mathbf{\Sigma})$.

Since $\mathbb{P}$ has no relation,  the defining presentation of $\mathcal{L}_{\omega}(\mathbb{P})$ contains only q-determinant and arc exchange relations. All these relations are quadratic (inhomogeneous) in the generators $\mathcal{A}^{\mathbb{G}}$ and we want to apply the Diamond Lemma to prove that $\mathcal{L}_{\omega}(\mathbb{P})$ is Koszul.

\vspace{2mm} \par 
\textbf{Reminder on the Diamond Lemma for PBW bases}
\vspace{2mm} \par 
Following the exposition in Section $4$ of \cite{LodayValletteOperads}, we briefly recall the statement of the Diamond Lemma for PBW bases.

 Let $V$ be a free finite rank $\mathds{k}$-module, denote by $T(V):=\oplus_{n\geq 0} V^{\otimes n}$ the tensor algebra and fix $R \subset V^{\otimes 2}$ a finite subset. The quotient algebra $\mathcal{A}:= \quotient{T(V)}{(R)}$ is called a \textit{quadratic algebra}. Let $\{v_i\}_{i\in I}$ be a totally ordered basis of $V$ and write $I=\{1, \ldots, k\}$ so that $v_i < v_{i+1}$. Then the set $J:= \bigsqcup_{n\geq 0} I^n$ (where $I^0=\{0\}$) is totally ordered by the lexicographic order and the set of elements $v_{\mathbf{i}}=v_{i_1}\ldots v_{i_n}$, for $\mathbf{i}=(i_1, \ldots, i_k)$, forms a basis of $T(V)$. We suppose that the elements $r\in R$ (named relators) have the form
 $$ r= v_i v_j - \sum_{(k,l)<(i,j)} \lambda_{kl}^{ij} v_k v_l.$$
 The term $v_iv_j$ is called the \textit{leading term} of $r$.
  We assume that two distinct relators have distinct leading terms.
  Define the family
 \begin{equation}\label{def_PBW_basis}
  \mathcal{B} :=  \{ v_{i_1} \ldots v_{i_n}| \quad \mbox{so that }v_{i_k}v_{i_{k+1}} \mbox{is not a leading term}, \quad \forall 1\leq k \leq n-1\}, 
  \end{equation}
 and denote by $\mathcal{B}^{(3)}\subset \mathcal{B}$ the subset of elements of length $3$ (of the form $v_{i_1}v_{i_2}v_{i_3}$). Obviously the set $\mathcal{B}$ spans $\mathcal{A}$. 

\begin{theorem}[Diamond Lemma for PBW bases: Bergman \cite{Bergman_DiamondLemma}, see also \cite{LodayValletteOperads} Theorem $4.3.10$]\label{theorem_PBWbases1}
If $\mathcal{B}^{(3)}$ is free, then $\mathcal{B}$ is a (Poincar\'e-Birkhoff-Witt) basis and $\mathcal{A}$ is Koszul.
\end{theorem}

The arc exchange relations defining $\mathcal{L}_{\omega}(\mathbb{P})$ are quadratic, however the $q$-determinant relations are not (because of the $1$ in $\mathrm{det}_q(N(\alpha))=1$), so $\mathcal{L}_{\omega}(\mathbb{P})$ is not quadratic but rather inhomogeneous quadratic. An \textit{inhomogeneous quadratic algebra} is an algebra of the form  $\mathcal{A}:= \quotient{T(V)}{(R)}$, where $R\subset V^{\otimes 2}\oplus V\oplus \mathds{k} \subset T(V)$. We further make the assumptions that $(ql_1): R\cap V =\{0\}$ and $(ql_2): (R\otimes V +V\otimes R)\cap V^{\otimes 2} \subset R\cap V^{\otimes 2}$. The hypothesis $(ql_2)$ says that one cannot create new relations by adding an element to $R$, so it is not restrictive. Like before, we fix an ordered basis $\{v_i\}_{i\in I}$ of $V$ and suppose that the relators of $R$ have the form
\begin{equation}\label{eq_relators}
 r = v_i v_j - \sum_{(k,l)<(i,j)} \lambda_{kl}^{ij} v_k v_l - c_{i,j}, 
 \end{equation}
where $c_{i,j}$ are some scalars and we suppose that two distinct relators have distinct leading terms.
The associated quadratic algebra $q\mathcal{A}$ is the algebra with same generators $v_i$ but where the relators have been changed by replacing the scalars $c_{i,j}$ by $0$. 
Let $q\mathcal{B} \subset q\mathcal{A}$ and $\mathcal{B} \subset \mathcal{A}$ be the two generating families defined by Equation \eqref{def_PBW_basis}. 

\begin{theorem}[\cite{LodayValletteOperads} Theorem $4.3.18$]\label{theorem_PBWbases2}
Suppose that $q\mathcal{B}^{(3)}\subset q\mathcal{A}$ is free, then both $q\mathcal{B}$ and $\mathcal{B}$ are (PBW) bases of $q\mathcal{A}$ and $\mathcal{A}$ respectively and both $q\mathcal{A}$ and $\mathcal{A}$ are Koszul.
 \end{theorem}
 
 There exists a linear surjective morphism $\varphi : q\mathcal{A} \rightarrow \mathcal{A}$ sending the generating family $q\mathcal{B}$ to $\mathcal{B}$ (see \cite[Section $4.2.9$]{LodayValletteOperads}). So, if $\mathcal{B}$ is a basis of $\mathcal{A}$, then $q\mathcal{B}$ is free, therefore Theorem \ref{theorem_PBWbases2} implies that $\mathcal{A}$ is Koszul. Therefore, we have the 
 
 \begin{theorem}\label{theorem_Koszul}
 If $\mathcal{B}$ is a basis of $\mathcal{A}$, then $\mathcal{A}$ is Koszul.
 \end{theorem}

  \vspace{2mm} \par 
\textbf{The relators of the stated skein presentations and PBW bases}
\vspace{2mm} \par
For $\alpha \in \mathbb{G}$, we write 
 $$\mathcal{B}(\alpha) = \{ (\alpha_{++})^a (\alpha_{+-})^b (\alpha_{--})^c, a,b,c \geq 0\} \cup \{ (\alpha_{++})^a (\alpha_{-+})^b (\alpha_{--})^c, a,b,c \geq 0\}\subset \mathcal{L}_{\omega}(\mathbb{P}).$$
 Fix a total order $<$ on the set $\mathbb{G}$ of generators and index its elements as $\mathbb{G}=\{ \alpha_1, \ldots, \alpha_n \}$, where $\alpha_i < \alpha_{i+1}$. Let 
 $$ \underline{\mathcal{B}}^{\mathbb{G}}:= \{ m_1 m_2\ldots  m_n | \quad m_i \in \mathcal{B}(\alpha_i) \} \subset \mathcal{L}_{\omega}(\mathbb{P}).$$
 
 We want to apply Theorem \ref{theorem_Koszul} to prove that $\mathcal{L}_{\omega}(\mathbb{P})$ is Koszul. By definition, $\mathcal{L}_{\omega}(\mathbb{P})$ is an inhomogeneous quadratic algebra whose set of generators is $\mathcal{A}^{\mathbb{G}}= \{ \alpha_{ij} | \alpha \in \mathbb{G}, i,j=\pm \}$ and whose relations are the arcs exchange and $q$-determinant relations. 
 
 We first define a total order $\prec$ on $\mathcal{A}^{\mathbb{G}}$ by imposing that 
  $ \alpha_{ab} \prec \beta_{cd}$ if  $\alpha < \beta$ and that $\alpha_{++}\prec \alpha_{+-} \prec \alpha_{-+} \prec \alpha_{--}$. 

 The goal of this subsection is to rewrite the q-determinant and arc exchange relations such that they define a set of  relators of the form \eqref{eq_relators} whose leading terms are pairwise distinct, satisfying $(ql_1)$ and $(ql_2)$  and such that the set of leading terms is 
 
 \begin{equation}\label{eq_leadingterms}
  \mathrm{Leading Terms}:= \{ \alpha_{ab}\beta_{cd} | \mbox{ such that either }(\alpha>\beta) \mbox{ or }( \alpha=\beta \mbox{ and either }a<c\mbox{ or }b<d) \}. 
  \end{equation}
 
 The set $ \underline{\mathcal{B}}^{\mathbb{G}}$ is the generating set defined by \eqref{def_PBW_basis} with this set of leading terms (\textit{i.e.} $ \underline{\mathcal{B}}^{\mathbb{G}}$ is the set of elements $v_1\ldots v_n$ where $v_i \in \mathcal{A}^{\mathbb{G}}$ and $v_iv_{i+1}$ is not in $\mathrm{Leading Terms}$). At this stage, it will become clear that $ \underline{\mathcal{B}}^{\mathbb{G}}$ spans $\mathcal{L}_{\omega}(\mathbb{P})$.
 Once we'll perform this task, we will prove in Step $3$ that $ \underline{\mathcal{B}}^{\mathbb{G}}$ if free by showing that its image through $\Psi : \mathcal{L}_{\omega}(\mathbb{P})\rightarrow \mathcal{S}_{\omega}(\mathbf{\Sigma})$ is a basis of $\mathcal{S}_{\omega}(\mathbf{\Sigma})$. This will imply that $\Psi$ is an isomorphism (so will prove Theorem \ref{theorem1}) and Theorem \ref{theorem_Koszul} will imply that $\mathcal{L}_{\omega}(\mathbb{P})$ is Koszul (so it will prove Theorem \ref{theorem2}). 
 
 \vspace{2mm}
 \par Consider two distinct generators $\alpha, \beta \in \mathbb{G}$ such that $\alpha>\beta$. For each $a,b,c,d \in \{\pm \}$, we have an arc exchange relation of the form 
 $$ \alpha_{ab} \beta_{cd} = \sum_{ijkl = \pm } c_{a,b,c,d}^{i,j,k,l} \beta_{ij} \alpha_{kl}, $$
 where $c_{a,b,c,d}^{i,j,k,l}$ are some scalars. We associate the relator $r= \alpha_{ab}\beta_{cd} - \sum_{ijkl = \pm } c_{a,b,c,d}^{i,j,k,l} \beta_{ij} \alpha_{kl}$, whose leading term is $\alpha_{ab}\beta_{cd}$ (because $\alpha>\beta$ implies that $ \alpha_{ab}\beta_{cd} \succ \beta_{ij}\alpha_{kl}$)  and denote by $R_{\alpha, \beta}$ the set (of cardinal $16$) of such relators.
 
 \vspace{2mm}
 \par Now suppose that $\alpha \in \mathbb{G}$ is of type $a$. The set of relations between the generators $\alpha_{ij}$ are given by
 $$M(\alpha)\odot M(\alpha) = \mathscr{R}^{-1}(M(\alpha) \odot M(\alpha)) \mathscr{R}, \quad \mbox{ and }\mathrm{det}_q(M(\alpha))=1.$$
 Note that in this case, the subalgebra of $\mathcal{L}_{\omega}(\mathbb{P})$ generated by the $\alpha_{ij}$  is isomorphic to $\mathcal{O}_q[\SL_2]\cong \mathcal{S}_{\omega}(\mathbb{B})$. We rewrite those relations as follows:
 
\begin{equation}\tag{Ra}
\left\{
\begin{array}{ll}
 \alpha_{+-}\alpha_{++} = q\alpha_{++}\alpha_{+-}, & \alpha_{-+}\alpha_{++}=q\alpha_{++}\alpha_{-+}, \\
  \alpha_{--}\alpha_{+-}=q\alpha_{+-}\alpha_{--}, &  \alpha_{--}\alpha_{-+} = q\alpha_{-+}\alpha_{--}, \\
   \alpha_{+-}\alpha_{-+} = q \alpha_{++}\alpha_{--} - q, & \alpha_{-+}\alpha_{+-}=q \alpha_{++}\alpha_{--} - q, \\ 
 \alpha_{--}\alpha_{++} = q^2 \alpha_{++}\alpha_{--}  + 1 - q^2. &
\end{array}
\right.
\end{equation}

The associated set of relators $R_{\alpha}$ is defined by assigning to each of the seven equalities of the form $x=y$ in the system $Ra$, the relator $r:= x-y$ with leading term $x$. 
Note that the set of leading terms of the elements of $R_{\alpha}$ is the set of elements $\alpha_{ab}\alpha_{cd}$ such that either $a<c$ or $b<d$.

 \vspace{2mm}
 \par Now suppose that $\alpha \in \mathbb{G}$ is of type $d$. The set of relations between the generators $\alpha_{ij}$ are given by

$$\left( \mathds{1}_2 \odot N(\alpha) \right) \mathscr{R}^{-1} \left( \mathds{1}_2 \odot N(\alpha) \right)\mathscr{R} =
\mathscr{R} \left( \mathds{1}_2 \odot N(\alpha) \right) \mathscr{R}^{-1} \left( \mathds{1}_2 \odot N(\alpha) \right), \quad \mbox{ and }\mathrm{det}_{q^2}(N(\alpha))=1, $$
where $N(\alpha)=C^{-1}M(\alpha)$. These relations generate the same ideal as the following set of relations:

\begin{equation}\tag{Rd}
\left\{
\begin{array}{ll}
\alpha_{-+} \alpha_{++} = \alpha_{++}\alpha_{-+} + (q-q^{-1})q^2 \alpha_{+-}\alpha_{--}, & \alpha_{+-}\alpha_{++} = q^2 \alpha_{++} \alpha_{+-},  \\
 \alpha_{--}\alpha_{-+} = \alpha_{-+}\alpha_{--} + (q-q^{-1})q^2\alpha_{+-}\alpha_{--},  &  \alpha_{--}\alpha_{+-} = q^2 \alpha_{+-} \alpha_{--}, \\
 \alpha_{+-}\alpha_{-+} = \alpha_{++}\alpha_{--}- (q- q^{-1})^2 \alpha_{+-}^2 - A,  &  \alpha_{-+}\alpha_{+-}= \alpha_{++}\alpha_{--}- (q- q^{-1})^2 \alpha_{+-}^2 - A, \\
\alpha_{--}\alpha_{++} = q^2 \alpha_{++}\alpha_{--} - q^2(q-q^{-1})^2 \alpha_{+-}^2 +A(1-q^2).&
\end{array}
\right.
\end{equation}

As before, we denote by $R_{\alpha}$ the set of relators obtained from system (Rd) by assigning to each of the seven equalities of the form $x=y$ in the system $Ra$, the relator $r:= x-y$ with leading term $x$. Again,  the set of leading terms of the elements of $R_{\alpha}$ is the set of elements $\alpha_{ab}\alpha_{cd}$ such that either $a<c$ or $b<d$.

 \vspace{2mm}
\par For $\alpha\in \mathbb{G}$ of type $c$, the set of relations between the elements $\alpha_{ij}$ can be obtained from the system (Rd) using the reflection anti-involution. Once re-arranging the terms, we get the system of relations:

\begin{equation}\tag{Rc}
\left\{
\begin{array}{ll}
\alpha_{-+} \alpha_{++} = \alpha_{++}\alpha_{-+} + (q-q^{-1}) \alpha_{+-}\alpha_{--}, & \alpha_{+-}\alpha_{++} = q^2 \alpha_{++} \alpha_{+-},  \\
 \alpha_{--}\alpha_{-+} = \alpha_{-+}\alpha_{--} + (q-q^{-1})\alpha_{+-}\alpha_{--},  &  \alpha_{--}\alpha_{+-} = q^2 \alpha_{+-} \alpha_{--}, \\
 \alpha_{+-}\alpha_{-+} = q^2\alpha_{++}\alpha_{--}-  A^3,  &  \alpha_{-+}\alpha_{+-}= q^2\alpha_{++}\alpha_{--}-  A^3, \\
\alpha_{--}\alpha_{++} = q^2 \alpha_{++}\alpha_{--} +(q-q^{-1})^2 \alpha_{+-}^2 +A^{-1}(1-q^2).&
\end{array}
\right.
\end{equation}

Like previously, we denote by  $R_{\alpha}$ the associated set of relators and note that the  set of leading terms is the set of elements $\alpha_{ab}\alpha_{cd}$ such that either $a<c$ or $b<d$.

 \vspace{2mm}
\par Let $V$ be the free $\mathds{k}$-module with basis $\mathcal{A}^{\mathbb{G}}$ and $R \subset \mathds{k}\oplus V^{\otimes 2} \subset  T(V)$ be the union of the sets of relators $R_{\alpha, \beta}$ and $R_{\alpha}$, where $\alpha, \beta \in \mathbb{G}$ and $\alpha>\beta$. Then $\mathcal{L}_{\omega}(\mathbb{P})=\quotient{T(V)}{(R)}$, the leading terms of $R$ are pairwise distinct and they form the set $\mathrm{LeadingTerms}$ of Equation \eqref{eq_leadingterms} and the hypotheses $(ql_1)$ and $(ql_2)$ are obviously satisfied. Therefore, if we prove that $\underline{\mathcal{B}}^{\mathbb{G}}$ is a basis of $\mathcal{L}_{\omega}(\mathbb{P})$ then Theorem \ref{theorem_Koszul} would imply that $\mathcal{L}_{\omega}(\mathbb{P})$ is Koszul.

\subsection{Step $3$: Injectivity of $\Psi$}

 Denote by $\mathcal{B}^{\mathbb{G}}\subset \mathcal{S}_{\omega}(\mathbb{P})$ the image of $\underline{\mathcal{B}}^{\mathbb{G}}$ by $\Psi : \mathcal{L}_{\omega}(\mathbb{P})\rightarrow \mathcal{S}_{\omega}(\mathbf{\Sigma})$.

\begin{theorem}\label{prop_new_basis}
The set $\mathcal{B}^{\mathbb{G}}$ is a basis of $\mathcal{S}_{\omega}(\mathbf{\Sigma})$.
\end{theorem}

\begin{corollary}
\begin{enumerate}
\item The morphism $\Psi : \mathcal{L}_{\omega}(\mathbb{P})\rightarrow \mathcal{S}_{\omega}(\mathbf{\Sigma})$ is an isomorphism.
\item The family $\mathcal{B}^{\mathbb{G}}$ is a PBW basis and $\mathcal{S}_{\omega}(\mathbf{\Sigma})$ is Koszul.
\end{enumerate}
\end{corollary}

The fact that $\mathcal{B}^{\mathbb{G}}$ spans linearly $\mathcal{S}_{\omega}(\mathbf{\Sigma})$ follows from the surjectivity of $\Psi$ (so follows from Proposition \ref{prop_generators}), however we will reprove this fact.  The proof of Theorem \ref{prop_new_basis} is divided in two steps: first we introduce another family $\mathcal{B}^{\mathbb{G}}_{+} \subset \mathcal{S}_{\omega}(\mathbf{\Sigma})$ and prove that $\mathcal{B}^{\mathbb{G}}_+$ is free by relating it to the basis $\mathcal{B}$. Next we use a filtration of $\mathcal{S}_{\omega}(\mathbf{\Sigma})$ to deduce that  $\mathcal{B}^{\mathbb{G}}$ is free from the fact that  $\mathcal{B}^{\mathbb{G}}_+$ is free.

\vspace{2mm}
\par For $\alpha \in \mathbb{G}$ and $n\geq 0$, we denote by $\alpha^{\left<n \right>}$ the simple diagram made of $n$ pairwise non-intersecting copies of $\alpha$. For $\mathbf{n}\in \mathbb{N}^{\mathbb{G}}$, we denote by $D(\mathbf{n})$ the simple diagram $\bigsqcup_{\alpha \in \mathbb{G}} \alpha^{\left<n(\alpha)\right>}$. Denote by $v$ and $w$ the two endpoints of $\alpha$ and by $a$ and $b$ the (non necessary distinct) boundary arcs containing $v$ and $w$ respectively. Write $v_1, \ldots, v_n$ and $w_1, \ldots, w_n$ the endpoints of $\alpha^{\left< n\right>}$ such that $v_i<_a v_{i+1}$ and $w_i <_b w_{i+1}$ (so $v_i$ and $w_i$ are not necessary the boundary points of the same component of $\alpha^{\left<n \right>}$). A state $s\in \mathrm{St}(D(\mathbf{n}))$ is \textit{positive} if for all $\alpha \in \mathbb{G}$ and for all $i\leq j$ one has $s(v_i)\leq s(v_j)$ and $s(w_i)\leq s(w_j)$; we let $\mathrm{St}^+(D(\mathbf{n}))$ denote the set of positive states.

\begin{definition}
We denote by $\mathcal{B}^{\mathbb{G}}_{+} \subset \mathcal{S}_{\omega}(\mathbf{\Sigma})$ the set of classes $[D(\mathbf{n}), s]$ for $\mathbf{n}\in \mathbb{N}^{\mathbb{G}}$ and $s\in \mathrm{St}^+(D(\mathbf{n}))$.
\end{definition}

\begin{notations}
Let $(D,s)$ be a stated diagram and $a$ a boundary arc. We denote by $d_a([D,s]) \in \mathbb{N}$ the number of pairs $(v,w)$ in $\partial_a D$ such that $v<_a w$ and $(s(v), s(w)) = (+, -)$ (recall that the orientation of $\Sigma_{\mathcal{P}}$ induces an orientation of $a$ which, in turns, induces the order $<_a$). 
We also write $d([D,s])= \sum_a d_a([D,s])$. Note that $s$ is $\mathfrak{o}^+$-increasing if and only if $d([D,s])=0$.
\end{notations}

Let us develop an element $b \in \mathcal{B}^{\mathbb{G}}_+$ in the basis $\mathcal{B}$ as $b=\sum_i a_i [D_i,s_i]$. We denote by $S(b)\subset \mathcal{B}$ the set of basis elements $[D_i,s_i]$ such that $a_i\neq 0$.

\begin{lemma}\label{lemma_basis+}
\begin{enumerate}
\item There exists a unique element $m(b)=[D_0,s_0] \in S(b)$ such that for all $[D_i,s_i]\in S(b)$ such that $[D_i,s_i]\neq [D_0,s_0]$, one has $|\partial D_i| > |\partial D_0|$.
\item The map $m : \mathcal{B}^{\mathbb{G}}_{+} \rightarrow \mathcal{B}$ is injective.
\item The family $\mathcal{B}^{\mathbb{G}}_+$ is a basis.
\end{enumerate}
\end{lemma}

\begin{proof}
$(1)$ Let $b=[D(\mathbf{n}), s] \in \mathcal{B}^{\mathbb{G}}_+$ and let us define $m(b)$. If $s$ is $\mathfrak{o}^+$-increasing, then $b\in \mathcal{B}$ so $m(b)=b$ satisfies the property. Else $d:=d(b)>0$ and we can find a boundary arc $a$ and a pair $(v,w)$ of consecutive points of $\partial_a D(\mathbf{n})$ such that $(s(v),s(w))=(+,-)$ and $v<_a w$. By gluing the points $v$ and $w$ together and then pushing it in the interior of $\Sigma_{\mathcal{P}}$ (that is by performing the local move $\heightexch{->}{-}{+} \mapsto \heightcurve$), we get a new stated diagram $(D_2,s_2)$ such that $d([D_2,s_2])= d(b)-1$. By performing the above move consecutively, we get a sequence of stated diagrams 
$(D(\mathbf{n}),s)=(D_1,s_1) \mapsto (D_2,s_2) \mapsto \ldots \mapsto (D_d, s_d)$ with $d[D_{i+1},s_{i+1}] = d(D_i,s_i)-1$, so $m(b):= [D_d,s_d] \in \mathcal{B}$ (see Figure \ref{fig_application_m} for an example). The skein relation

$$
\heightexch{->}{-}{+} =
q \heightexch{->}{+}{-}
+ \omega
\heightcurve, 
$$
shows that $[D_i, s_i] = q [D_i',s'_i] +  \omega [D_{i+1},s_{i+1}]$, where $|\partial D'_i| =| \partial D_i | > | \partial D_{i+1}|$ so the first assertion follows by induction.

\vspace{2mm}
\par $(2)$ To prove that  $m : \mathcal{B}^{\mathbb{G}}_{+} \rightarrow \mathcal{B}$ is injective, we construct a right inverse $g: \mathcal{B} \rightarrow \mathcal{B}^{\mathbb{G}}_{+}$ such that $g\circ m =\id$. Let $[D,s] \in \mathcal{B}$ and let us define $g([D,s])$. Obviously, $g$ sends the class of the empty diagram to itself, so we suppose that $D$ is not empty.
First suppose that $D$ is connected and consider a path $\alpha_D \in \Pi_1(\Sigma_{\mathcal{P}}, \mathbb{V})$ representing $D$. Since $\mathbb{P}=(\mathbb{G}, \emptyset)$ is a finite presentation of $\Pi_1(\Sigma_{\mathcal{P}}, \mathbb{V})$ without relation, the path $\alpha_D$ decomposes in a unique way as $\alpha_D = \alpha_{i_1}^{\varepsilon_1} \ldots \alpha_{i_n}^{\varepsilon_n}$, where $\alpha_{i_k} \in \mathbb{G}$. Let  $\mathbf{n}_D \in \mathbb{N}^{\mathbb{G}}$ be such that  $\mathbf{n}_D(\alpha)$ is the number of times $\alpha$ appears in the decomposition of $\alpha_D$. Now, if $D$ is not connected and has connected components $D_1, \ldots, D_n$, we set $\mathbf{n}_D:= \sum_{i=1}^n \mathbf{n}_{D_i}$. By applying the cutting arc relation
$$
\heightcurve = \omega^{-1} \heightexch{->}{-}{+}-q \heightexch{->}{+}{-}
$$
 consecutively, we write $[D,s]$ as a linear combination: 
$$ [D,s] = \sum_{s \in \mathrm{St}(D(\mathbf{n}_D))} \alpha_s [D(\mathbf{n}_D), s].$$ 
Note that because $[D,s] \in \mathcal{B}$, the stated diagram $(D,s)$ contains no trivial arc by definition, so only the positive states $s\in \mathrm{St}^+(D(\mathbf{n}_D))$ have a non-vanishing coefficient $\alpha_s\neq 0$. In particular, we have proved that $\mathcal{B}^{\mathbb{G}}_+$ generates $\mathcal{S}_{\omega}(\mathbf{\Sigma})$. We define $g([D,s])$ as the element $[D(\mathbf{n}_D), s_0]$ such that $\alpha_s\neq 0$ and $d([D(\mathbf{n}_D), s_0])$ is minimal. Note that it obtained from $[D,s]$ by a series of local moves $\heightcurve \mapsto \heightexch{->}{-}{+}$ which are the inverse of the local moves used to define $m([D(\mathbf{n}_D), s_0])$, so $g \circ m = \id$. Compare the Examples of Figure \ref{fig_prop_generators} and \ref{fig_application_m} for an illustration.

\vspace{2mm}
\par $(3)$ It remains to prove that $\mathcal{B}^{\mathbb{G}}_+$ is free. Consider an arbitrary total order $\prec$ on $\mathcal{B}$ such that if $| \partial D | < | \partial D'|$ then $[D,s] \prec [D',s']$ for any $[D,s], [D',s']$ in $\mathcal{B}$. By contradiction, suppose there exists a non empty finite family $\{ b_i^+ \}_{i \in I}$ of elements of $\mathcal{B}^{\mathbb{G}}_+$ and a family of non-vanishing scalars $\{ x_i \}_{i \in I}$ such that 
\begin{equation}\label{eq_cl}
 \sum_{i\in I} x_i b_i^+ = 0.
 \end{equation}
Let $i_0 \in I$ be such that $m(b_{i_0})$ is the minimum for $\prec$ of the set $\{ m(b_i), i\in I\}$. By developing each $b_i^+$ in Equation \eqref{eq_cl} in the basis $\mathcal{B}$, we get a vanishing linear combination $\sum_j y_j b_j =0$ of elements $b_j \in \mathcal{B}$. Since $\mathcal{B}$ is free, we have $y_j=0$ for all $j$.
 Let $j_0$ be such that $b_{j_0}=m(b_{i_0})$. It follows from assertions $(1)$ and $(2)$ that $y_{j_0}= x_{i_0}$, so $x_{i_0}=0$ and we have a contradiction.

\begin{figure}[!h] 
\centerline{\includegraphics[width=8cm]{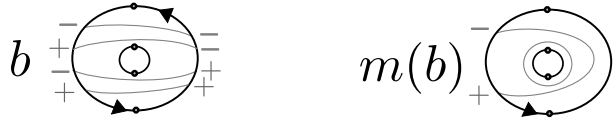} }
\caption{An element $b\in \mathcal{B}^{\mathbb{G}}_+$ and its associated element $m(b)\in \mathcal{B}$. Here $\mathbb{G}=\{ \beta_1,\beta_2, \beta_3, \beta_4\}$ are the generators of Figure \ref{fig_presentation}} 
\label{fig_application_m} 
\end{figure}

\end{proof}

We now want to deduce that $\mathcal{B}^{\mathbb{G}}$ is a basis from the fact that $\mathcal{B}^{\mathbb{G}}_+$ is a basis. The argument is based on the use of an algebra filtration of $\mathcal{S}_{\omega}(\mathbf{\Sigma})$ that we now introduce. 

\begin{definition}\label{def_filtration_cheloud}
For $\mathbf{n}\in \mathbb{N}^{\mathbb{G}}$, we let $|\mathbf{n}|:= \sum_{\alpha \in \mathbb{G}} \mathbf{n}(\alpha)$. For a class $[D(\mathbf{n}), s]$, we set $\lVert [D(\mathbf{n}), s]\rVert := (| \mathbf{n}|, - d([D(\mathbf{n}),s]) ) \in \mathbb{N}\times \mathbb{Z}$. 
Denote by $<$ the lexicographic order on $\mathbb{N}\times \mathbb{Z}$, \textit{i.e.} $(k_1,k_2)<(k'_1,k'_2)$ if either $k_1<k'_1$ or $k_1=k'_1$ and $k_2<k'_2$.
Eventually, to $\mathbf{k}=(k_1,k_2) \in \mathbb{N}\times \mathbb{Z}$ we associate the submodule
$$ \mathcal{F}_{\mathbf{k}} := \mathrm{Span} \left( [D(\mathbf{n}),s], \mbox{ such that } \lVert [D(\mathbf{n}), s] \rVert \leq \mathbf{k} \right).$$
\end{definition}

In order to prove that the $\{ \mathcal{F}_{\mathbf{k}} \}$ form an algebra filtration, the following elementary observation will be quite useful:

\begin{lemma}\label{lemma_easy}
Let $T, T'$ be two tangles in $\Sigma_{\mathcal{P}} \times (0,1)$ which are isotopic through an isotopy that does not preserves the height orders. Let $s\in \mathrm{St}(T)$ and $s'\in \mathrm{St}(T')$ be two states such that for all boundary arc $a$, if $\partial_a T = \{v_1, \ldots, v_n \}$ and $\partial_aT' = \{w_1, \ldots, w_n \}$ are ordered such that $h(v_i)<h(v_{i+1})$ and $h(w_i)< h(w_{i+1})$, then one has $s(v_i)=s'(w_i)$ for all $i\in \{1, \ldots, n\}$. Then one has 

\begin{equation}\label{eq_developement}
 [T,s]= \omega^n [T',s'] + \sum_{\sigma \in \mathrm{St}(T'), d([T',\sigma])<d([T',s'])} x_{\sigma} [T', \sigma], 
 \end{equation}
 
where $n\in \mathbb{Z}$, $x_{\sigma} \in \mathds{k}$ and the sum in the right-hand-side is over states $\sigma$ of $T'$ such that $d([T',\sigma])<d([T',s'])$.
\end{lemma}

\begin{proof}
We say that a tangle $T_i$ is obtained from another one $T_{i+1}$ by an elementary height exchange if there exists a boundary arc $a$ and two consecutive points $v$ and $w$ in $\partial_a T_i$ with $h(v)<h(w)$ ('consecutive' means that there does not exists any $p\in \partial_a T_i$ such that $h(v)<h(p)<h(w)$) such that $T_{i+1}$ is the tangle obtained from $T_i$ by exchanging the heights of $v$ and $w$. 
Since $T$ and $T'$ are isotopic, through an isotopy that does not preserve the height orders, we can obtain $T'$ from $T$ by a finite sequence $T=T_1 \mapsto T_2 \mapsto \ldots \mapsto T_n=T'$ of elementary height exchanges. It is clear that if one has a development \eqref{eq_developement} when the pair $(T,T')$ is equal to a pair $(T_i, T_{i+1})$ and a pair $(T_{i+1}, T_{i+2})$, then it holds for the pair $(T_i, T_{i+2})$, so by induction on the size $n$ the finite sequence, it is sufficient to prove the lemma in the particular case where $T$ and $T'$ differ by an elementary height exchange. In this case, Equation \eqref{eq_developement} follows from the height exchange relations
\begin{align*}
&  \adjustbox{valign=c}{\includegraphics[width=1cm]{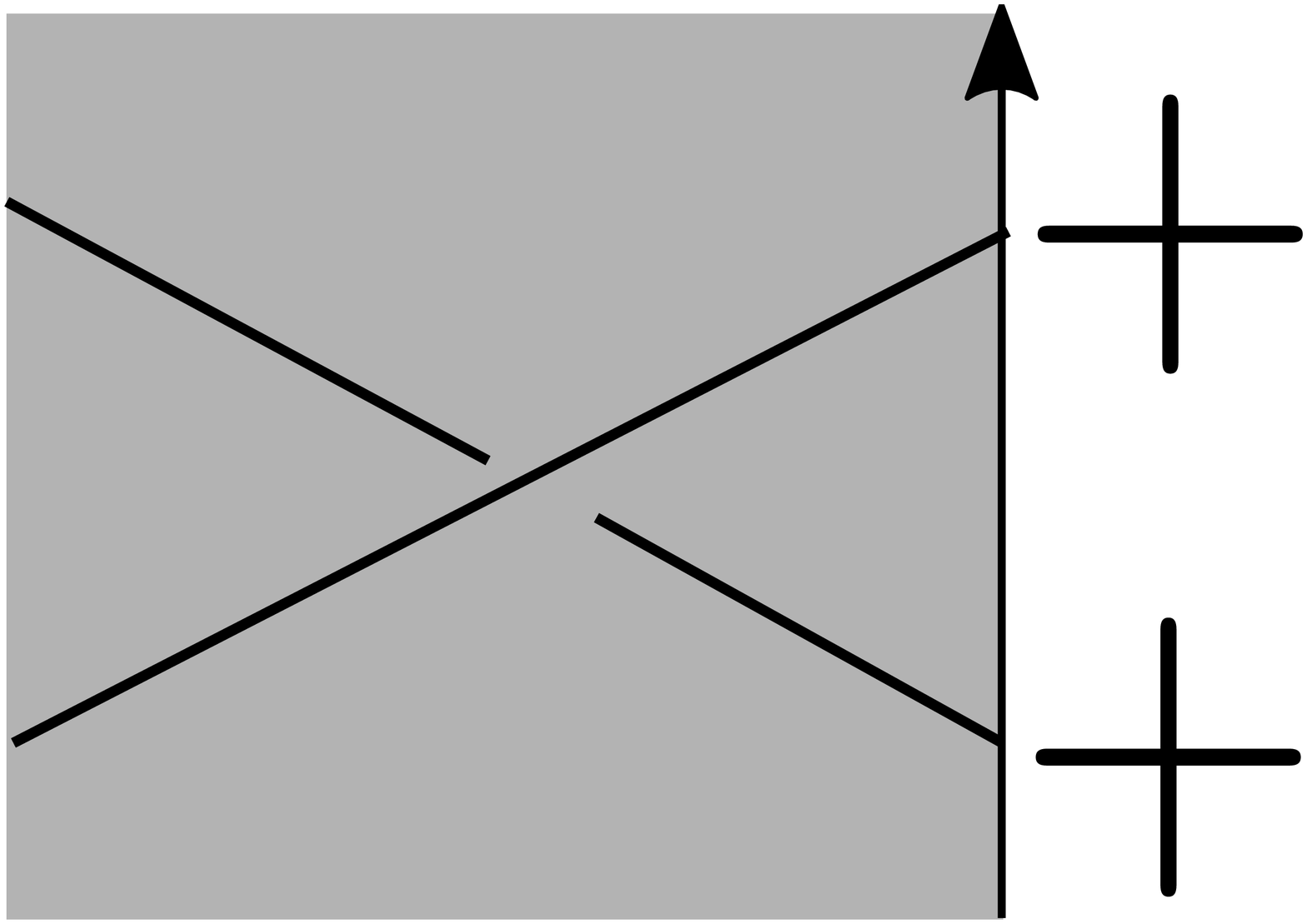}} = A \heightexch{->}{+}{+}, \quad   \adjustbox{valign=c}{\includegraphics[width=1cm]{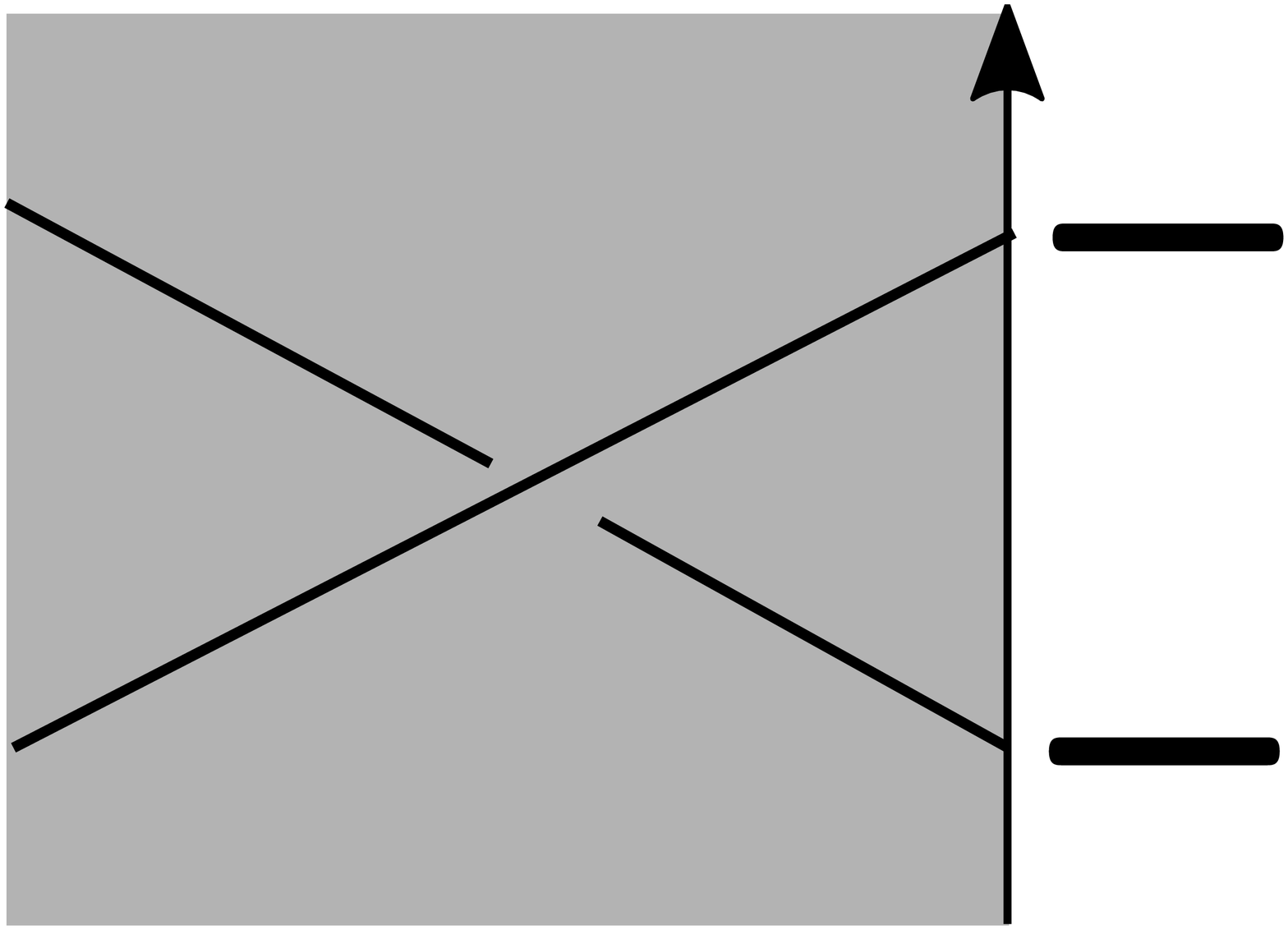}} = A \heightexch{->}{-}{-},  \quad \adjustbox{valign=c}{\includegraphics[width=1cm]{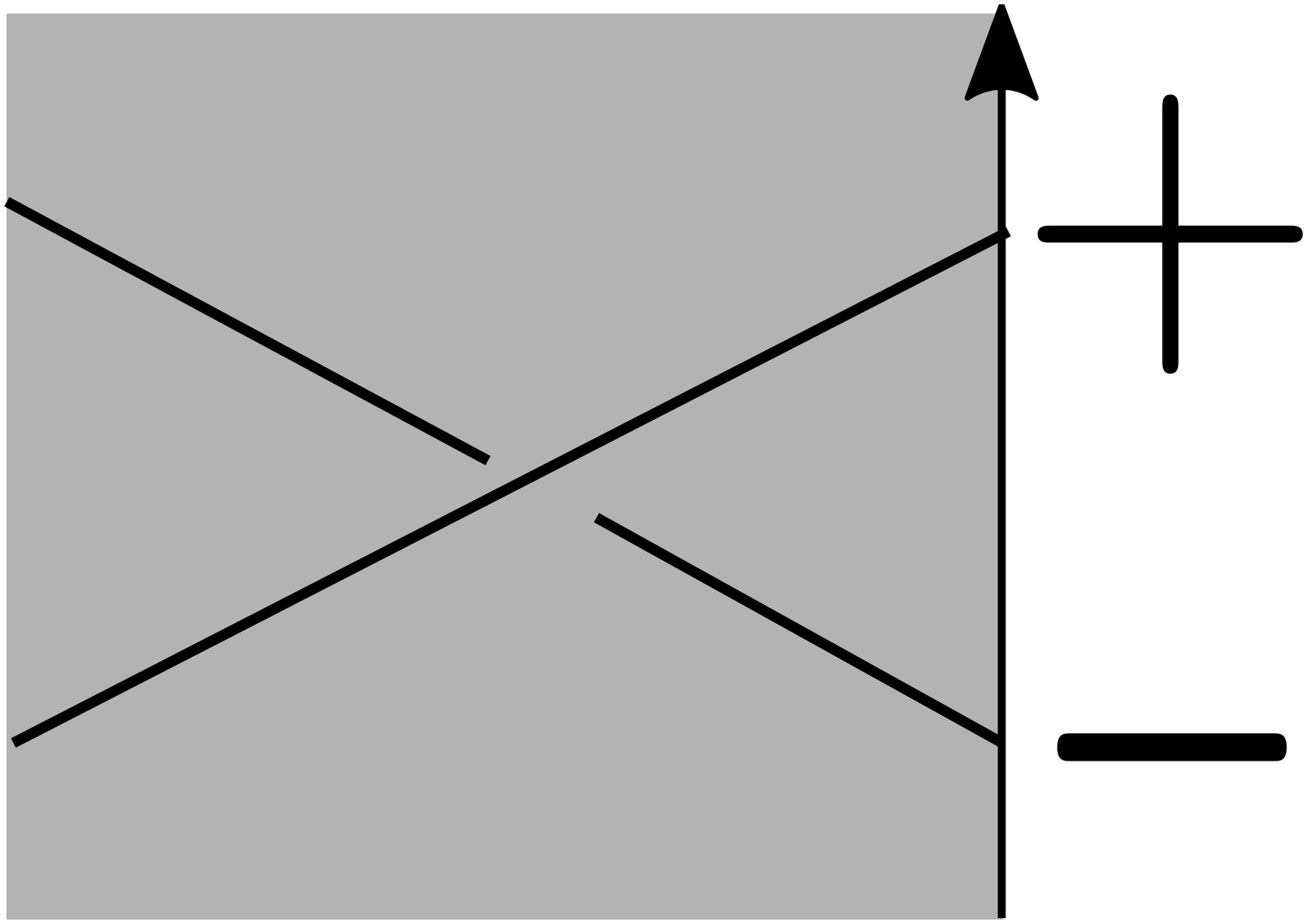}} = A^{-1} \heightexch{->}{-}{+}, \\
& \adjustbox{valign=c}{\includegraphics[width=1cm]{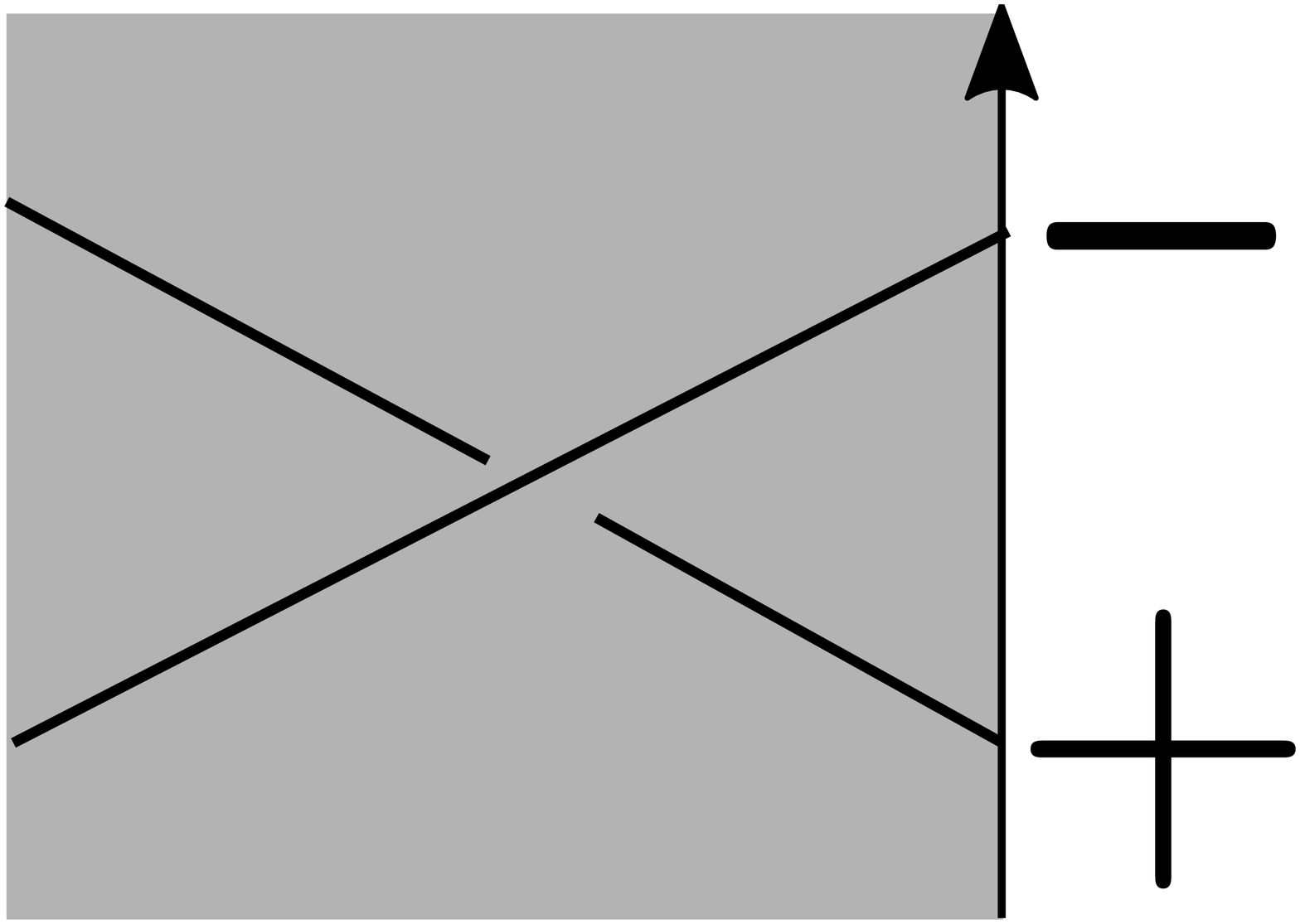}} = A^{-1} \heightexch{->}{+}{-} +(A-A^{-3})  \heightexch{->}{-}{+}.
\end{align*}

\end{proof}

\begin{notations}\label{notations_reloud}
Let $b\in \mathcal{B}^{\mathbb{G}}$, so by definition $b=b_{\alpha_1}\ldots b_{\alpha_n}$, where $b_{\alpha_i}\in \mathcal{B}(\alpha_i)$, that is one has either $b_{\alpha_i}= \alpha_{++}^{a_i}\alpha_{+-}^{b_i}\alpha_{--}^{c_i}$ or $b_{\alpha_i}= \alpha_{++}^{a_i}\alpha_{-+}^{b_i}\alpha_{--}^{c_i}$, for some $a_i, b_i, c_i \geq 0$. Let $\mathbf{n} \in \mathbb{N}^{\mathbb{G}}$ be defined by $\mathbf{n}(\alpha_i) := a_i +b_i+c_i$. Let $T(\mathbf{n})$ be the tangle underlying $D(\mathbf{n})$. 
Let $(T,s)$ be a stated tangle (unique up to isotopy) such that  $b=[T,s]$, so that $T(\mathbf{n})$ is obtained from $T$ by an isotopy that does not necessary preserve the height order. Eventually we define the element $b^+ := [T(\mathbf{n}), s^+] \in \mathcal{B}^{\mathbb{G}}_+$, where $s^+ \in \mathrm{St}^+(T(\mathbf{n}))$ is the unique state such that $(T,s)$ and $(T(\mathbf{n}),s^+)$ satisfies the assumption of Lemma \ref{eq_developement}. Note that the induced map $(\cdot)^+ : \mathcal{B}^{\mathbb{G}} \rightarrow \mathcal{B}^{\mathbb{G}}_+$, sending $b$ to $b^+$, is a bijection.
\end{notations}

\begin{lemma}\label{lemma_dev2}
\begin{enumerate}
\item For $\mathbf{k}, \mathbf{k}' \in \mathbb{N} \times \mathbb{Z}$, one has $\mathcal{F}_{\mathbf{k}} \cdot \mathcal{F}_{\mathbf{k}'} \subset \mathcal{F}_{\mathbf{k}+\mathbf{k}'}$.
\item For $b\in \mathcal{B}^{\mathbb{G}}$, one has 
\begin{equation}\label{eq_developement2}
 b = \omega^n b^+ + \mbox{lower terms}, 
 \end{equation}
where $n\in \mathbb{Z}$ and 'lower terms' is a linear combination of basis elements $b_i^+ \in \mathcal{B}^{\mathbb{G}}_+$ such that $\lVert b_i^+ \rVert < \lVert b^+ \rVert$.
\end{enumerate}
\end{lemma}

Note that the second assertion of Lemma \ref{lemma_dev2} implies that $\mathcal{B}^{\mathbb{G}}$ spans $\mathcal{S}_{\omega}(\mathbf{\Sigma})$ so reproves Proposition \ref{prop_generators}.

\begin{proof}
$(1)$ Let $x:=[T(\mathbf{n}), s]$ and $y:=[T(\mathbf{n}'),s']$ and denote by $(T(\mathbf{n}) \cup T(\mathbf{n}'), s\cup s')$ the stated tangle obtained by stacking $(T(\mathbf{n}),s)$ on top of $(T(\mathbf{n}'),s')$, so that $x\cdot y = [T(\mathbf{n}) \cup T(\mathbf{n}'), s\cup s']$. The tangles $T(\mathbf{n}) \cup T(\mathbf{n}')$ and $T(\mathbf{n} + \mathbf{n}')$ differ by an isotopy that does not necessary preserve the height orders, so Lemma \ref{lemma_easy} implies that $x\cdot y$ is a linear combination of elements of the form $[D(\mathbf{n}+\mathbf{n}'), \sigma]$ such  that $\lVert [D(\mathbf{n}+\mathbf{n}'), \sigma] \rVert \leq \lVert x \rVert + \lVert y\rVert$. This proves the first assertion.

\vspace{2mm}
\par $(2)$ Using Notations \ref{notations_reloud}, we apply Lemma \ref{lemma_easy} to  $b=[T,s]$ and $b^+= [T(\mathbf{n}),s^+]$, and Equation \ref{eq_developement2} is just a rewriting of Equation \eqref{eq_developement}.
\end{proof}

\begin{proof}[Proof of Theorem \ref{prop_new_basis}]
The fact that $\mathcal{B}^{\mathbb{G}}$ generates $\mathcal{S}_{\omega}(\mathbf{\Sigma})$ follows both from Proposition \ref{prop_generators} or from the second assertion of Lemma \ref{lemma_dev2}. Let us prove that $\mathcal{B}^{\mathbb{G}}$ is free.
Consider a vanishing linear combination $\sum_{i\in I} x_i b_i =0$, where $b_i \in \mathcal{B}^{\mathbb{G}}$ and the $x_i$ are non-vanishing scalars and suppose by contradiction that $I$ is not empty. Let $i_0 \in I$ be such that $\lVert b_{i_0}^+ \rVert$ is the maximum, for $<$, of the set $\{ \lVert b_i^+ \rVert \} \subset \mathbb{N}\times \mathbb{Z}$. The fact that $(\cdot)^+ : \mathcal{B}^{\mathbb{G}} \rightarrow \mathcal{B}^{\mathbb{G}}_+$ is a bijection implies that this maximum is unique. Developing each $b_i$ in the basis $\mathcal{B}^{\mathbb{G}}_+$ (using Equation \eqref{eq_developement2}), the equation $\sum_{i\in I} x_i b_i =0$ induces an equation of the form $\sum_{b^+ \in \mathcal{B}^{\mathbb{G}}_+} y_{b^+} b^+$ and the fact that $\mathcal{B}^{\mathbb{G}}_+$ is free (Lemma \ref{lemma_basis+}) implies that each $y_{b^+}$ vanishes. Lemma \ref{lemma_dev2} implies that $x_{i_0}=y_{b^+_{i_0}}$. Therefore $x_{i_0}=0$ and we have a contradiction.

\end{proof}

\section{Lattice gauge field theory}

\subsection{Ciliated graphs and quantum gauge group coaction}\label{sec_graphs}

Since the pioneer work of Fock and Rosly \cite{FockRosly}, constructions in lattice gauge field theory are based on ciliated graphs. As we now explain, to a ciliated graph $(\Gamma,c)$ one can  associate a punctured surface $\mathbf{\Sigma}^0$ together with a finite presentation $\mathbb{P}$ of its associated groupoid. 
\begin{definition}
\begin{enumerate}
\item A \textit{ribbon graph} $\Gamma$ is a finite graph together with the data, for each vertex, of a cyclic ordering of its adjacent half-edges. An \textit{orientation} for a ribbon graph is the choice of an orientation for each of its edges. 
\item  A \textit{ciliated ribbon graph} $(\Gamma, c)$ is a ribbon graph $\Gamma$ together with a lift, for each vertex, of the cyclic ordering of the adjacent half-edges, to a linear ordering. In pictures, if the half-edges adjacent to a vertex have the cyclic ordering $e_1<e_2<\ldots <e_n <e_1$ that we lift to the linear ordering $e_1<e_2<\ldots <e_n$, we draw a \textit{cilium} between $e_n$ and $e_1$. 
\item We associate surfaces to ribbon graphs as follows. 
\begin{itemize}
\item[(i)]
Place a  a disc $D_v$ on top of  each vertex $v$ and a band $B_e$ on top of each edge $e$,  then glue the discs to the band using the cyclic ordering: we thus get a surface $S(\Gamma)$ named the \textit{ fattening of }$\Gamma$. 
\item[(ii)]
 The \textit{closed punctured surface} $\mathbf{\Sigma}(\Gamma)=(\Sigma(\Gamma), \mathcal{P})$ \textit{associated to }$\Gamma$ is the closed punctured surface obtained from $S(\Gamma)$ by gluing a disc to each boundary component and placing a puncture inside each added disc. So $S(\Sigma)$ retracts by deformations to $\Sigma_{\mathcal{P}}(\Gamma)$.
 \item[(iii)]
The \textit{open punctured surface} $\mathbf{\Sigma}^0(\Gamma, c)=(\Sigma^0(\Gamma,c), \mathcal{P}^0)$ \textit{associated to }$(\Gamma,c)$ is obtained from $S(\Gamma)$ by first pushing each vertex $v$ to the boundary of $S(\Gamma)$ in the direction of the associated cilium. Said differently, if the ordered half-edges adjacent to $v$ are $e_1<e_2<\ldots<e_n$, we push $v$ in the boundary of $D_v$ such that it lies between the band $B_{e_n}$ and the band $B_{e_1}$. Next place a puncture $p_v$ next to $v$ (in the counterclockwise direction) on the same boundary component than $v$. Eventually, to each boundary component of $S(\Gamma)$ which does not contain any puncture $p_v$, glue a disc and place a puncture inside the disc. In the so-obtained punctured surface  $\mathbf{\Sigma}^0(\Gamma, c)$, each boundary arc contains exactly one vertex $v$ of $\Gamma$, so we denote by $a_v$ the boundary arc containing $v$. Suppose that $\Gamma$ is oriented. Then the oriented edges of $\Gamma$ form a set $\mathbb{G}$ of generators of $\Pi_1(\Sigma^0_{\mathcal{P}}, \mathbb{V})$ such that  $\mathbb{P}(\Gamma, c):=(\mathbb{G}, \emptyset)$ is a finite presentation without relations. 
\end{itemize}
\item For $v_1, v_2$ two distinct vertices of $(\Gamma,c)$, the ciliated graph $(\Gamma_{v_1 \#v_2}, c_{v_1\#v_2})$ is obtained by gluing the vertices $v_1$ and $v_2$ together to a vertex $v$ in such a way that if $e_1<\ldots<e_n$ and $f_1<\ldots <f_m$ are the ordered half-edges adjacent to $v_1$ and $v_2$ respectively, then the linear order of the half-edges adjacent to $v$ is $e_1<\ldots<e_n<f_1<\ldots<f_m$. Note that $c_{v_1\#v_2} \neq c_{v_2\#v_1}$.
\end{enumerate}

Figure \ref{fig_ciliated_graphs} illustrates two examples having the same ribbon graph but different ciliated structures: the punctured surface $\mathbf{\Sigma}^0(\Gamma,c)$ is a disc with two inner punctures and two boundary punctures whereas $\mathbf{\Sigma}^0(\Gamma,c')$ is an annulus with one puncture per boundary component and one inner puncture.

\end{definition}

\begin{figure}[!h] 
\centerline{\includegraphics[width=12cm]{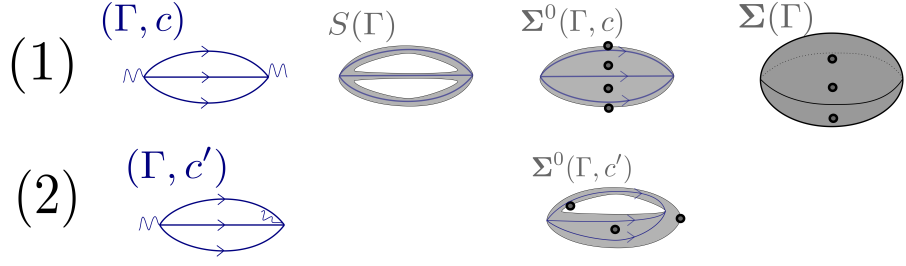} }
\caption{(1) From left to right: a ciliated graph $(\Gamma,c)$, its fattening $S(\Gamma)$, its open punctured surface $\mathbf{\Sigma}^0(\Gamma, c)$ and its closed punctured surface $\mathbf{\Sigma}(\Gamma)$. (2) The same ribbon graph with a different ciliated structure $c'$ and the associated open punctured surface $\mathbf{\Sigma}^0(\Gamma, c')$.}
\label{fig_ciliated_graphs} 
\end{figure}

\begin{remark}
Costantino and L\^e made in \cite{CostantinoLe19} the following important remark: the punctured surface $\mathbf{\Sigma}^0(\Gamma_{v_1\#v_2}, c_{v_1\#v_2})$ is obtained from $\mathbf{\Sigma}^0(\Gamma, c) \bigsqcup \mathbb{T}$ by gluing the boundary arcs $a_{v_1}$ and $a_{v_2}$  to two faces of the triangle $\mathbb{T}$. In particular, when $\Gamma = \Gamma_1 \bigsqcup \Gamma_2$ with $v_1\in \Gamma_1$ and $v_2\in \Gamma_2$, this property, together with Theorem \ref{theorem_exactsequence} permitted the authors of \cite{CostantinoLe19} to prove that $\mathcal{S}_{\omega}(\mathbf{\Sigma}^0(\Gamma_{v_1\#v_2}, c_{v_1\#v_2}))$ is the cobraided tensor product of $\mathcal{S}_{\omega}(\mathbf{\Sigma}^0(\Gamma_1, c_1))$ with  $\mathcal{S}_{\omega}(\mathbf{\Sigma}^0(\Gamma_2, c_2))$. The same gluing property were discovered by Alekseev-Grosse-Schomerus in \cite{AlekseevGrosseSchomerus_LatticeCS1,AlekseevGrosseSchomerus_LatticeCS2} for the quantum moduli spaces.
\end{remark}

For an oriented ciliated graph $(\Gamma, c)$, we denote by $V(\Gamma)$ its set of vertices and $\mathcal{E}(\Gamma)$ its set of (oriented) edges. Like in the previous section, we see the elements of $\mathcal{E}(\Gamma)$ as oriented arcs. Denote by $\mathbb{D}_0$ the punctured surface made of a disc with a single puncture on its boundary.  The closed punctured surface $\mathbf{\Sigma}(\Gamma)$ is obtained from the open one $\mathbf{\Sigma}^0(\Gamma,c)$ by gluing a copy $\mathbb{D}_0$ along each boundary arc $a_v$. Therefore, writing $\widehat{\mathbb{D}}:= \bigsqcup_{v\in V(\Gamma)} \mathbb{D}_0$, by Theorem \ref{theorem_exactsequence} one has an exact sequence:

\begin{equation}\label{eq_exact_sequence_graph}
0 \rightarrow \mathcal{S}_{\omega}(\mathbf{\Sigma}(\Gamma)) \xrightarrow{i} \mathcal{S}_{\omega}(\mathbf{\Sigma}^0(\Gamma,c) \bigsqcup \widehat{\mathbb{D}}) \xrightarrow{\Delta^R -\sigma \circ \Delta^L}   \mathcal{S}_{\omega}(\mathbf{\Sigma}^0(\Gamma,c)\bigsqcup \widehat{\mathbb{D}}) \otimes \mathcal{O}_q[\SL_2]^{\otimes V(\Gamma)}, 
\end{equation}
where $i$ represents the gluing map.

Using the isomorphism $\mathcal{S}_{\omega}(\mathbb{D}_0) \cong \mathds{k}$ sending the class of the empty stated tangle to the neutral element $1\in \mathds{k}$, we define an isomorphism 

 $$\kappa: \mathcal{S}_{\omega}\left(\mathbf{\Sigma}^0(\Gamma,c) \bigsqcup \widehat{\mathbb{D}}\right) \cong \mathcal{S}_{\omega}(\mathbf{\Sigma}^0(\Gamma,c))\otimes \otimes_{v\in V(\Gamma)} \mathcal{S}_{\omega}(\mathbb{D}_0) \cong \mathcal{S}_{\omega}(\mathbf{\Sigma}^0(\Gamma,c)). $$

Denote by $\iota :  \mathcal{S}_{\omega}(\mathbf{\Sigma}(\Gamma)) \hookrightarrow  \mathcal{S}_{\omega}(\mathbf{\Sigma}^0(\Gamma,c))$ the injective morphism $\iota := \kappa \circ i$. Also denote by $\Delta^{\mathcal{G}} : \mathcal{S}_{\omega}(\mathbf{\Sigma}^0(\Gamma,c)) \rightarrow \mathcal{S}_{\omega}(\mathbf{\Sigma}^0(\Gamma,c)) \otimes \mathcal{O}_q[\SL_2]^{\otimes V(\Gamma)}$ the (unique) morphism making the following diagram commuting: 
$$
\begin{tikzcd}
\mathcal{S}_{\omega}\left(\mathbf{\Sigma}^0(\Gamma,c) \bigsqcup \widehat{\mathbb{D}}\right) 
\arrow[r, "\Delta^R"] \arrow[d, "\kappa", "\cong"'] &
\mathcal{S}_{\omega}\left(\mathbf{\Sigma}^0(\Gamma,c) \bigsqcup \widehat{\mathbb{D}}\right)  \otimes  \mathcal{O}_q[\SL_2]^{\otimes V(\Gamma)} 
\arrow[d, "\kappa\otimes \id", "\cong"'] \\
\mathcal{S}_{\omega}(\mathbf{\Sigma}^0(\Gamma,c)) 
\arrow[r, "\Delta^{\mathcal{G}}"] & 
\mathcal{S}_{\omega}(\mathbf{\Sigma}^0(\Gamma,c)) \otimes  \mathcal{O}_q[\SL_2]^{\otimes V(\Gamma)} 
\end{tikzcd}
$$

\begin{definition}\label{def_gauge_coaction}
The \textit{quantum gauge group} is the Hopf algebra $\mathcal{O}_q[\mathcal{G}]:= \mathcal{O}_q[\SL_2]^{\otimes V(\Gamma)}$. The (right) Hopf-comodule map $\Delta^{\mathcal{G}} :  \mathcal{S}_{\omega}(\mathbf{\Sigma}^0(\Gamma,c)) \rightarrow  \mathcal{S}_{\omega}(\mathbf{\Sigma}^0(\Gamma,c)) \otimes \mathcal{O}_q[\mathcal{G}] $ is called the \textit{quantum gauge group coaction}.
\end{definition}

Note that, by definition,  the following diagram commutes:
$$
\begin{tikzcd}
\mathcal{S}_{\omega}\left(\mathbf{\Sigma}^0(\Gamma,c) \bigsqcup \widehat{\mathbb{D}}\right) 
\arrow[r, "\sigma \circ \Delta^L"] \arrow[d, "\kappa", "\cong"'] &
\mathcal{S}_{\omega}\left(\mathbf{\Sigma}^0(\Gamma,c) \bigsqcup \widehat{\mathbb{D}}\right)  \otimes  \mathcal{O}_q[\mathcal{G}] 
\arrow[d, "\kappa\otimes \id", "\cong"'] \\
\mathcal{S}_{\omega}(\mathbf{\Sigma}^0(\Gamma,c)) 
\arrow[r, "\id \otimes \epsilon"] & 
\mathcal{S}_{\omega}(\mathbf{\Sigma}^0(\Gamma,c)) \otimes  \mathcal{O}_q[\mathcal{G}] 
\end{tikzcd}
$$
Therefore 
the exactness of the sequence \eqref{eq_exact_sequence_graph}  implies that we have the following exact sequence:
\begin{equation}\label{eq_exact_sequence_graph}
0 \rightarrow \mathcal{S}_{\omega}(\mathbf{\Sigma}(\Gamma)) \xrightarrow{\iota} \mathcal{S}_{\omega}(\mathbf{\Sigma}^0(\Gamma,c)) \xrightarrow{\Delta^{\mathcal{G}} - \id \otimes  \epsilon}  \mathcal{S}_{\omega}(\mathbf{\Sigma}^0(\Gamma,c)) \otimes \mathcal{O}_q[\mathcal{G}].
\end{equation}

Said differently, $\iota(\mathcal{S}_{\omega}(\mathbf{\Sigma}(\Gamma)))$ is the subalgebra of $ \mathcal{S}_{\omega}(\mathbf{\Sigma}^0(\Gamma,c))$ of coinvariant vectors for the quantum gauge group coaction.

\begin{notations}
For $x\in \mathcal{O}_q[\SL_2]$ and $v_0\in \mathring{V}$, we denote by $x^{(v_0)} \in \mathcal{O}_q[\mathcal{G}]=\mathcal{O}_q[\SL_2]^{\otimes V(\Gamma)}$ the element of the form $\otimes_v y_v$, where $y_v=1$ for $v\neq v_0$ and $y_{v_0}=x$. 
\end{notations}

Let $\alpha$ be an arc of type either $a$ or $d$ and write $v_1$ and $v_2$ the elements of $\mathbb{V}$ corresponding to the boundary arcs containing $s(\alpha)$ and $t(\alpha)$ respectively. The quantum gauge group coaction is characterised by the following formula illustrated in Figure \ref{fig_qgaugecoaction}: 

\begin{equation}\label{eq_qcoaction}
\Delta^{\mathcal{G}}( \alpha_{ij} ) = \sum_{a,b=\pm} \alpha_{ab} \otimes x_{jb}^{(v_2)} x_{ia}^{(v_1)}.
\end{equation}

\begin{figure}[!h] 
\centerline{\includegraphics[width=8cm]{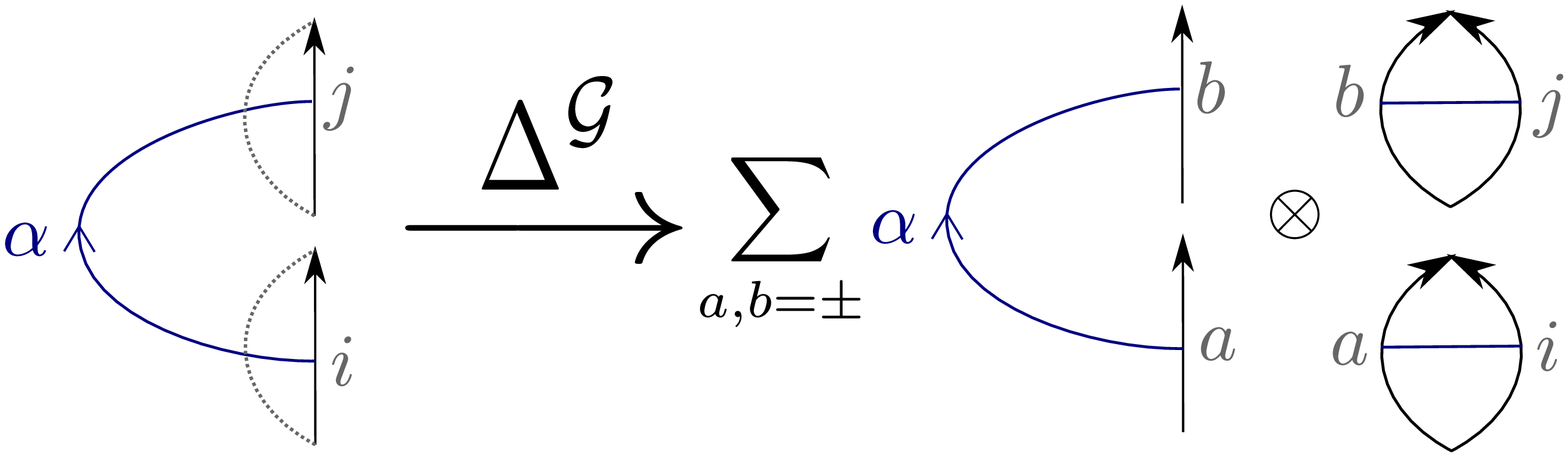} }
\caption{An illustration of Equation \eqref{eq_qcoaction}.}
\label{fig_qgaugecoaction} 
\end{figure} 

\par In order to prepare the comparison between stated skein algebras at $\omega=+1$ and relative character varieties in the next subsection, let us derive from \ref{theorem1} an alternative presentation of $\mathcal{S}_{\omega}(\mathbf{\Sigma})$. During all the rest of the section, we fix a finite presentation $\mathbb{P}=(\mathbb{G}, \mathbb{RL})$ of $\Pi_1(\Sigma_{\mathcal{P}}, \mathbb{V})$ such that every arc of $\mathbb{G}$ is either of type $a$ or $d$. 

\vspace{2mm}
\par 
When comparing skein algebras with character varieties, there is a well known sign issue which requires some attention. When $\mathbf{\Sigma}$ is closed, the skein algebra $\mathcal{S}_{+1}(\mathbf{\Sigma})$ is generated by the classes of closed curves $\gamma$ whereas the algebra $\mathbb{C}[\mathcal{X}_{\SL_2}(\mathbf{\Sigma})]$ of regular functions of the character variety is generated by curve functions $\tau_{\gamma}$, sending a class $[\rho]$ of representation $\rho : \pi_1(\Sigma_{\mathcal{P}})\rightarrow  \SL_2(\mathbb{C})$ to $\tau_{\gamma}([\rho]):= \tr( \rho(\gamma))$. However there is no isomorphism $\mathcal{S}_{+1}(\mathbf{\Sigma}) \cong \mathbb{C}[\mathcal{X}_{\SL_2}(\mathbf{\Sigma})]$ sending $\gamma$ to $\tau_{\gamma}$. Instead, we fix a spin structure on $\Sigma_{\mathcal{P}}$ with associated Johnson quadratic form  $\omega : \mathrm{H}_1(\Sigma_{\mathcal{P}}; \mathbb{Z}/2\mathbb{Z})\rightarrow \mathbb{Z}/2\mathbb{Z}$ and define $w(\gamma):= 1+\omega([\gamma])$. Then it follows from \cite{Bullock, PS00, Barett} that we have an isomorphism $\mathcal{S}_{+1}(\mathbf{\Sigma}) \cong \mathbb{C}[\mathcal{X}_{\SL_2}(\mathbf{\Sigma})]$ sending $\gamma$ to $(-1)^{w(\gamma)}\tau_{\gamma}$. A similar sign issue appears when dealing with stated skein algebras and relative character varieties; this was studied in \cite{KojuQuesneyClassicalShadows} to which we refer for further details (see also \cite{Thurston_PositiveBasis, CostantinoLe19} for an elegant interpretation of this sign issue in term of \textit{twisted character variety}). 

In short, the authors defined in \cite{KojuQuesneyClassicalShadows} the notion of \textit{relative spin structure} to which one can associate a map $w : \mathbb{G} \rightarrow  \mathbb{Z}/2\mathbb{Z}$ having the property that for any simple relation $R=\beta_k \star \ldots \star \beta_1$, one has $\sum_{i=1}^k w(\beta_i) = 1$. We will call \textit{spin function} a map $w : \mathbb{G} \rightarrow  \mathbb{Z}/2\mathbb{Z}$ satisfying this property.

\begin{notations}
Let $w$ be a spin function.
For $\alpha \in \mathbb{G}$, we denote by $U(\alpha)$ the $2\times 2$ matrix with coefficients in $\mathcal{S}_{\omega}(\mathbf{\Sigma})$ defined by 
\begin{equation}\label{eq_defU}
 U(\alpha) := \left\{ 
\begin{array}{ll}
(-1)^{w(\alpha)} \omega C^{-1}M(\alpha) & \mbox{, if }\alpha \mbox{ is of type a;}\\
(-1)^{w(\alpha)}  C^{-1}M(\alpha) =(-1)^{w(\alpha)}N(\alpha) & \mbox{, if }\alpha \mbox{ is of type d.}
\end{array}
\right. 
\end{equation}
\end{notations}

\begin{proposition}\label{prop_presU}
\begin{enumerate}
\item
The stated skein algebra $\mathcal{S}_{\omega}(\mathbf{\Sigma})$ admits the alternative presentation with  generators the elements $U(\alpha)_i^j$, with $\alpha \in \mathbb{G}$ and $i,j = \pm$, together with the following relations: 
\begin{itemize}
\item the $q$-determinant relations $\det_q(U(\alpha))=1$, when $\alpha$ is of type $a$, and $\det_{q^2}(U(\alpha))=1$, when $\alpha$ is of type $d$; 
\item for $R=\beta_k \star \ldots \star \beta_1 \in \mathbb{RL}$ a relation, where $l$ generators $\beta_i$ are of type $a$, the trivial loop relation:
\begin{equation}\label{trivial_loop_relU}
 U(\beta_k)\ldots U(\beta_1) = A^3 \omega^l.
 \end{equation}
\item for each pair $(\alpha, \beta)$ of generators in $\mathbb{G}$, the arc exchange relations obtained from the relations in Lemma \ref{lemma_arcsrelations} by replacing $N(\alpha)$ and $N(\beta)$ by $U(\alpha)$ and $U(\beta)$ respectively.
\end{itemize}
\item The quantum gauge group coaction is characterized by the formula: 

\begin{equation}\label{eq_qcoaction2}
\Delta^{\mathcal{G}} \left( U(\alpha)_i^j \right) = \sum_{a,b =\pm} U(\alpha)_{a}^b \otimes S(x_{bj})^{(v_2)} x_{ia}^{(v_1)}, 
\end{equation}

where we used the same notations than in Equation \eqref{eq_qcoaction}.
\end{enumerate}
\end{proposition}

\begin{proof}
It is clear from Equation \eqref{eq_defU} that the matrix elements $U(\alpha)_i^j$ generate the same algebra than the elements $M(\alpha)_i^j=\alpha_{ij}$, so they generate $\mathcal{S}_{\omega}(\mathbf{\Sigma})$. We need to check that the $q$-determinant, trivial loop and arcs exchange relations for the elements $\alpha_{ij}$ are equivalent to the relations of the proposition for the elements $U(\alpha)_i^j$. When $\alpha\in \mathbb{G}$ is of type $d$, clearly the relation $\det_{q^2}(N(\alpha))=1$ is equivalent to the relation $\det_{q^2}(U(\alpha))=1$. When $\alpha\in \mathbb{G}$ is of type $a$, the equivalence $\det_q(M(\alpha))=1 \Leftrightarrow \det_q(U(\alpha)) =1$ follows from a straightforward computation (and is the reason for the $\omega$ in the expression $U(\alpha)= (-1)^{w(\alpha)} \omega C^{-1}M(\alpha)$). The equivalence between Equations \eqref{eq_trivial_loops_rel} and \eqref{trivial_loop_relU} is straightforward (and is responsible for the introduction of the spin function and for  the $(-1)^{w(\alpha)}$ factor in the definition of $U(\alpha)$). The fact that the arcs exchange relations are equivalent to the same relations with $N(\alpha), N(\beta)$ replaced by $U(\alpha), U(\beta)$ follows from the fact that $C^{-1}\odot C^{-1}$ commutes with $\tau$, $\mathscr{R}$ and $\mathscr{R}^{-1}$. Indeed, that  $C^{-1}\odot C^{-1}$ commutes with $\tau$ is obvious. Recall that $\mathscr{R}$ is the matrix of $ \tau\circ q^{H\otimes H /2} \circ (\mathds{1} + (q-q^{-1})\rho(E)\otimes \rho(F))$ and $C^{-1}\odot C^{-1}$ is the matrix of $C^{-1}\otimes C^{-1}$. That $C^{-1}\otimes C^{-1}$ commutes with $q^{H\otimes H /2} $ and $(\mathds{1} + (q-q^{-1})\rho(E)\otimes \rho(F))$ follows from the fact that $C^{-1}$ is $U_q\mathfrak{sl}_2$ equivariant, so $C^{-1}\odot C^{-1}$ commutes with $\mathscr{R}$. The arguments for $\mathscr{R}^{-1}$ is similar. 

\vspace{2mm}
\par It remains to derive the formula \eqref{eq_qcoaction2} from \eqref{eq_qcoaction}. This is done by direct computation, left to the reader, using the following fact: for the two $2\times 2$ matrices $X= \begin{pmatrix}x_{++} & x_{+-} \\ x_{-+} & x_{--} \end{pmatrix}$ and $S(X) = \begin{pmatrix} S(x_{++}) & S(x_{+-}) \\ S(x_{-+}) & S(x_{--}) \end{pmatrix}$ with coefficients in $\mathcal{O}_q[\SL_2]$, one has $S(X) = C^{-1}{}^t X C$. Figure \ref{fig_qcoactionU} illustrates Equation \eqref{eq_qcoaction2}. In Figure \ref{fig_qcoactionU}, we used a special convention: we drawn stated diagrams that go "outside" of $\Sigma_{\mathcal{P}}$ in some small  bigon neighbourhoods of the boundary arcs; it must be understood that we need to apply a boundary skein relation in those neighborhood. This convention permits to draw pictorially the matrix coefficients $(C^{-1}M(\alpha))_i^j$. Note also that in Figure  \ref{fig_qcoactionU}, we dropped the scalar factor $(-1)^{w(\alpha)}$.

\begin{figure}[!h] 
\centerline{\includegraphics[width=12cm]{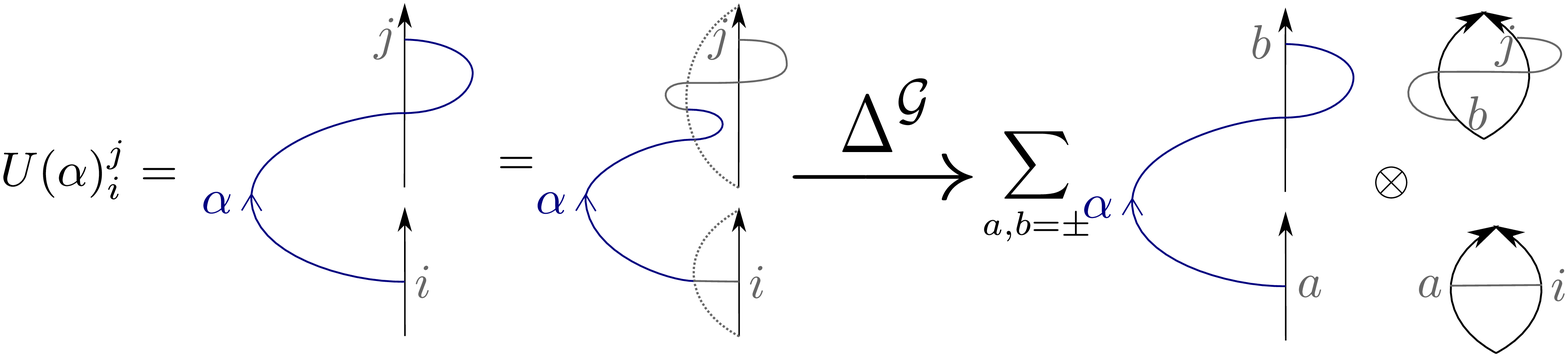} }
\caption{An illustration of Equation \eqref{eq_qcoaction2}.}
\label{fig_qcoactionU} 
\end{figure}

\end{proof}

\subsection{Relative character varieties}\label{sec_charvar}

Since the quantum moduli algebras are deformation quantizations of the (relative) character varieties studied by Fock and Rosly in \cite{FockRosly}, we briefly recall their construction and refer to \cite{Audin_Survey_FR} for a detailed survey. 
\vspace{2mm}
\par 
  Let us first consider a closed connected punctured surface $\mathbf{\Sigma}$ and denote by $\mathcal{M}_{\SL_2}(\mathbf{\Sigma})$ the set of isomorphism classes of $\SL_2(\mathbb{C})$ flat structures on $\Sigma_{\mathcal{P}}$, that is the set of isomorphism classes of  pairs $(P, \nabla)$ where $P$ is a principal $\SL_2(\mathbb{C})$-bundle and $\nabla$ a flat connection. Since any principal $\SL_2(\mathbb{C})$-bundle on a surface is trivializable, one can restrict to the classes of pairs $(P,\nabla)$ where $P=\Sigma_{\mathcal{P}} \times \SL_2(\mathbb{C})$ is trivial and $\nabla = d+A$, where $A \in \Omega^1(\Sigma_{\mathcal{P}}, \mathfrak{sl}_2)$. The flatness of $\nabla$ then translates to the equation $dA+[A\wedge A]=0$ and we denote by $\mathcal{A}_F \subset \Omega^1(\Sigma_{\mathcal{P}}, \mathfrak{sl}_2)$ the subspace of such $1$-forms. The gauge group $\mathcal{G}:=\mathrm{Aut}(P) \cong C^{\infty}(\Sigma_{\mathcal{P}}, \SL_2(\mathbb{C}))$ is the group of automorphisms and the moduli space can be identified with the quotient $\mathcal{M}_{\SL_2}(\mathbf{\Sigma}) \cong \quotient{ \mathcal{A}_F}{\mathcal{G}}$. By imposing some Sobolev regularity, as done by Atiyah-Bott in \cite{AB}, one can give to $\mathcal{A}_F$ a structure of Banach space however the gauge group action is far from been free; as a result the quotient $\quotient{ \mathcal{A}_F}{\mathcal{G}}$ does not inherit a geometric structure (it is not even Haussdorf as a topological space). One possibility to remedy this problem is to restrict to the (open dense) subspace $\mathcal{A}^0_F \subset \mathcal{A}_F$ of principal orbits and consider only the subset $\mathcal{M}^0_{\SL_2}(\mathbf{\Sigma}):= \quotient{\mathcal{A}^0_F}{\mathcal{G}} \subset \mathcal{M}_{\SL_2}(\mathbf{\Sigma})$ which can be endowed with the structure of smooth manifold. When $\mathcal{P}=\emptyset$, Atiyah-Bott defined in \cite{AB} a symplectic structure on $\mathcal{M}^0_{\SL_2}(\mathbf{\Sigma})$. A second possibility to get a geometric object out of $\mathcal{M}_{\SL_2}(\mathbf{\Sigma})$ is to consider character varieties. Fix a base point $v\in \Sigma_{\mathcal{P}}$. The Riemann-Hilbert correspondence asserts that the holonomy map induces a bijection 
  $$ \mathrm{Hol} : \mathcal{M}_{\SL_2}(\mathbf{\Sigma}) \xrightarrow{\cong} \quotient{ \Hom(\pi_1(\Sigma_{\mathcal{P}}, v), \SL_2(\mathbb{C}))}{\SL_2(\mathbb{C})}, $$
  where $\SL_2(\mathbb{C})$ acts on the set $\mathcal{R}_{\SL_2(\mathbb{C})}(\mathbf{\Sigma}) := \Hom(\pi_1(\Sigma_{\mathcal{P}}, v), \SL_2(\mathbb{C}))$ of representations by conjugacy. Now $\mathcal{R}_{\SL_2(\mathbb{C})}(\mathbf{\Sigma})$ is an affine variety and the action of (the reducible algebraic group) $\SL_2(\mathbb{C})$ is algebraic, therefore we might consider the algebraic quotient (familiar in Geometric Invariant Theory): 
  $$ \mathcal{X}_{\SL_2}(\mathbf{\Sigma}) := \mathcal{R}_{\SL_2}(\mathbf{\Sigma}) \sslash \SL_2(\mathbb{C}) $$
  named the \textit{character variety}, which is an affine (irreducible) variety. One has a surjective (so-called Reynolds) set-theoretical map $p: \mathcal{M}_{\SL_2}(\mathbf{\Sigma}) \rightarrow \mathcal{X}_{\SL_2}(\mathbf{\Sigma})$ which restricts to a  bijection $p : \mathcal{M}^0_{\SL_2}(\mathbf{\Sigma}) \xrightarrow{\cong} \mathcal{X}^0_{\SL_2}(\mathbf{\Sigma})$ on the smooth locus $\mathcal{X}^0_{\SL_2}(\mathbf{\Sigma})$. Therefore $\mathcal{X}_{\SL_2}(\mathbf{\Sigma})$ might be regarded as a good approximation of the moduli space $\mathcal{M}_{\SL_2}(\mathbf{\Sigma})$ which has a geometric structure of affine variety. Goldman showed in \cite{Goldman86} that the Atiyah-Bott symplectic structure on $\mathcal{M}^0_{\SL_2}(\mathbf{\Sigma})$ induces a Poisson structure on the algebra of regular functions of $ \mathcal{X}_{\SL_2}(\mathbf{\Sigma}) $ by giving an explicit formula for the Poisson bracket of two curve functions.
  
\vspace{2mm}
\par Now, let us consider an open punctured surface $\mathbf{\Sigma}$. We would like to change the definition of $\mathcal{M}_{\SL_2}(\mathbf{\Sigma})$ in such a way that the moduli space would have a nice behaviour for the operation of gluing two boundary arcs together, more precisely, we would like to have a surjective gluing map $\pi_{a\#b}: \mathcal{M}_{\SL_2}(\mathbf{\Sigma}) \rightarrow \mathcal{M}_{\SL_2}(\mathbf{\Sigma}_{a\#b})$. The trick is to consider the subset $\mathcal{A}_F^{\partial}  \subset \mathcal{A}_F$ of $1$-forms whose holonomy along any subarc of a boundary arc is null and to restricts to the smaller group $\mathcal{G}^{\partial} \subset \mathcal{G}$ corresponding to those maps in $C^{\infty}(\Sigma_{\mathcal{P}}, \SL_2(\mathbb{C}))$ whose restriction to $\partial \Sigma_{\mathcal{P}}$ is the constant map with value the neutral element of $\SL_2(\mathbb{C})$. We then define $\mathcal{M}_{\SL_2}(\mathbf{\Sigma}):= \quotient{\mathcal{A}_F^{\partial}}{\mathcal{G}^{\partial}}$ (which coincides with the previous definition when the punctured surface $\mathbf{\Sigma}$ is closed). If $c$ denotes the common image of $a$ and $b$ in $\mathbf{\Sigma}_{|a\#b}$, any flat connection is gauge equivalent to a connection whose restriction to subarcs of $c$ has trivial holonomy, therefore we have  a surjective gluing map $\pi_{a\#b}: \mathcal{M}_{\SL_2}(\mathbf{\Sigma}) \rightarrow \mathcal{M}_{\SL_2}(\mathbf{\Sigma}_{a\#b})$ as desired. To turn $\mathcal{M}_{\SL_2}(\mathbf{\Sigma})$ into a geometric (algebraic) object, one consider as previously the holonomy map, though we need more than one base point now. Let $\mathbb{V} \subset \Sigma_{\mathcal{P}}$ be a finite subset which intersects each boundary arc exactly once, denote by $\mathring{V}:= \mathbb{V} \cap \mathring{\Sigma}_{\mathcal{P}}$ its (possibly empty) subset of inner points and consider the discrete gauge group $\mathcal{G}_{\mathbb{V}}:= \SL_2(\mathbb{C})^{\mathring{V}}$. The holonomy map induces a bijection
$$ \mathrm{Hol} : \mathcal{M}_{\SL_2}(\mathbf{\Sigma}) \xrightarrow{\cong} \quotient{ \mathcal{R}_{\SL_2}(\mathbf{\Sigma}, \mathbb{V})}{\mathcal{G}_{\mathbb{V}}}, $$
where $\mathcal{R}_{\SL_2}(\mathbf{\Sigma}, \mathbb{V})$ is the set of functors $\rho : \Pi_1(\Sigma_{\mathcal{P}}, \mathbb{V}) \rightarrow \SL_2(\mathbb{C})$ and the discrete gauge group acts on the right by:
$$ (\rho\bullet g) (\alpha) := g(t(\alpha))^{-1} \rho(\alpha) g(s(\alpha)), \quad \mbox{ for all }\rho \in \mathcal{R}_{\SL_2}(\mathbf{\Sigma}, \mathbb{V}), g\in \mathcal{G}_{\mathbb{V}}, \alpha \in \Pi_1(\Sigma_{\mathcal{P}}, \mathbb{V}).$$

We claim that $\mathcal{R}_{\SL_2}(\mathbf{\Sigma}, \mathbb{V})$ can be given a structure of affine variety in such a way that the action of the reducible algebraic group $\mathcal{G}_{\mathbb{V}}$ is algebraic, so we can define the GIT quotient
$$  \mathcal{X}_{\SL_2}(\mathbf{\Sigma}) := \mathcal{R}_{\SL_2}(\mathbf{\Sigma}, \mathbb{V}) \sslash \mathcal{G}_{\mathbb{V}},  $$
which we call the \textit{relative character variety}. To prove the claim, consider a finite presentation $\mathbb{P}=(\mathbb{G}, \mathbb{RL})$ of $\Pi_1(\Sigma_{\mathcal{P}}, \mathbb{V}) $ and write $\mathbb{G}=(\alpha_1, \ldots, \alpha_n)$ and $\mathbb{RL}=(R_1, \ldots, R_m)$. Consider the regular map $\mathcal{R} : \SL_2(\mathbb{C})^{\mathbb{G}} \rightarrow \SL_2(\mathbb{C})^{\mathbb{RL}}$ written $\mathcal{R}=(\mathcal{R}_1, \ldots, \mathcal{R}_m)$, where the coordinate $\mathcal{R}_i$ associated to a relation $R_i = \alpha_{i_1}^{\varepsilon_1} \star \ldots \star \alpha_{i_k}^{\varepsilon_k}$ is the polynomial function 
$$ \mathcal{R}_i (g_1, \ldots, g_n) = g_{i_1}^{\varepsilon_1} \ldots g_{i_k}^{\varepsilon_k}.$$
Clearly, one has $\mathcal{R}_{\SL_2}(\mathbf{\Sigma}, \mathbb{V}) = \mathcal{R}^{-1}(\mathds{1}_2, \ldots, \mathds{1}_2)$ (where $\mathds{1}_2$ is the identity matrix), so $\mathcal{R}_{\SL_2}(\mathbf{\Sigma}, \mathbb{V}) $ is a subvariety of $\SL_2(\mathbb{C})^{\mathbb{G}}$. 

Note that  the algebra $\mathbb{C}[\mathcal{R}_{\SL_2}(\mathbf{\Sigma}, \mathbb{V})]$ of regular functions lies in the following exact sequence

\begin{equation}\label{cc}
 \mathbb{C}[\SL_2(\mathbb{C})]^{\otimes \mathbb{RL}} \xrightarrow{\mathcal{R}^* - \eta^{\otimes \mathbb{G}} \circ \epsilon^{\otimes \mathbb{RL}} }
  \mathbb{C}[\SL_2(\mathbb{C})]^{\otimes \mathbb{G}} \rightarrow  \mathbb{C}[\mathcal{R}_{\SL_2}(\mathbf{\Sigma}, \mathbb{V})]\rightarrow 0.
\end{equation}

So we have turned $\mathcal{R}_{\SL_2}(\mathbf{\Sigma}, \mathbb{V})$ into an affine variety. Now the discrete gauge group action is induced by the Hopf comodule map 
$\Delta^{\mathcal{G}} :  \mathbb{C}[\mathcal{R}_{\SL_2}(\mathbf{\Sigma}, \mathbb{V})] \rightarrow  \mathbb{C}[\mathcal{R}_{\SL_2}(\mathbf{\Sigma}, \mathbb{V})] \otimes \mathbb{C}[\mathcal{G}_{\mathbb{V}}]$,
 which is the restriction of the right co-module map $\widetilde{\Delta}_{\mathcal{G}} : \mathbb{C}[\SL_2(\mathbb{C})]^{\otimes \mathbb{G} } \rightarrow \mathbb{C}[\SL_2(\mathbb{C})]^{\otimes \mathbb{G}} \otimes \mathbb{C}[\SL_2(\mathbb{C})]^{\otimes \mathring{V}}$ defined by 
 
$$
\widetilde{\Delta}^{\mathcal{G}} \left( x^{(\alpha)}\right) = \sum {x''}^{(\alpha)} \otimes S(x''')^{(v_2)} {x'}^{(v_1)}, 
$$

for $x\in \mathcal{O}_q[\SL_2]$, $\alpha : v_1 \rightarrow v_2$ in $\mathbb{G}$ and we used Sweedler's notation $\Delta^{(2)}(x) = \sum x'\otimes x'' \otimes x'''$. In particular, when $x=x_{ij}$ with $i,j \in \{-, +\}$, the formula gives:

\begin{equation}\label{eq_qcoaction3}
\Delta^{\mathcal{G}} \left( x_{ij}^{(\alpha)}\right) = \sum_{a, b =\pm} {x}_{ab}^{(\alpha)} \otimes S(x_{bj})^{(v_2)} {x}_{ia}^{(v_1)}.
\end{equation}
Note the analogy between Equations \eqref{eq_qcoaction3} and \eqref{eq_qcoaction2}.

\vspace{2mm}
\par
Eventually, the algebra of regular functions of the relative character variety is defined as the set of coinvariant vectors for this coaction, that is by the exact sequence

\begin{equation}\label{eq_gauge_exactseq}
0 \rightarrow \mathbb{C}[\mathcal{X}_{\SL_2}(\mathbf{\Sigma})] \rightarrow \mathbb{C}[\mathcal{R}_{\SL_2}(\mathbf{\Sigma})] \xrightarrow{\Delta^{\mathcal{G}} - \id \otimes \epsilon} 
\mathbb{C}[\mathcal{R}_{\SL_2}(\mathbf{\Sigma})] \otimes \mathbb{C}[\mathcal{G}_{\mathbb{V}}].
\end{equation}

The relative character variety $\mathcal{X}_{\SL_2}(\mathbf{\Sigma})$ does not depend (up to unique isomorphism) on the choice of the triple $(\mathbb{V}, \mathbb{G}, \mathbb{RL})$ used to define it but only on $\mathbf{\Sigma}$; we refer to \cite{KojuTriangularCharVar} for a proof. Note that in the particular case where $\mathbb{V}\subset \partial \Sigma_{\mathcal{P}}$, the gauge group is trivial so $\mathcal{X}_{\SL_2}(\mathbf{\Sigma})=\mathcal{R}_{\SL_2}(\mathbf{\Sigma})$. Moreover, if the presentation $\mathbb{P}$ does not have any relation, then $\mathcal{R}_{\SL_2}(\mathbf{\Sigma})=\SL_2(\mathbb{C})^{\mathbb{G}}$. As we saw in Example \ref{exemple_pres}, such a presentation $\mathbb{P}$ always exists when $\mathbf{\Sigma}$ is a connected punctured surface with non trivial boundary, therefore in that case one has 
$$ \mathcal{X}_{\SL_2}(\mathbf{\Sigma}) = \SL_2(\mathbb{C})^{\mathbb{G}}.$$
Now consider an oriented ciliated graph $(\Gamma,c)$ and consider the associated finite presentation $(\mathbb{V}, \mathbb{G}, \mathbb{RL})$ of the groupoid $\Pi_1(\Sigma^0_{\mathcal{P}}(\Gamma,c), \mathbb{V})$ associated to the open punctured surface defined in the previous subsection. The same triple $(\mathbb{V}, \mathbb{G}, \mathbb{RL})$ gives also a finite presentation of $\Pi_1(\Sigma_{\mathcal{P}}(\Gamma), \mathbb{V})$ associated to the \textit{closed} punctured surface, where this time, all elements of $\mathbb{V}$ are inner vertices of $\Sigma_{\mathcal{P}}(\Gamma)$. Therefore one has 
$$\mathcal{X}_{\SL_2}(\mathbf{\Sigma}^0(\Gamma,c))= \mathcal{R}_{\SL_2}(\mathbf{\Sigma}(\Gamma)) = \SL_2(\mathbb{C})^{\mathcal{E}(\Gamma)}.$$
So  the exact sequence \eqref{eq_gauge_exactseq} can be rewritten as 

\begin{equation}\label{eq_gauge_exactseq2}
0 \rightarrow \mathbb{C}[\mathcal{X}_{\SL_2}(\mathbf{\Sigma}(\Gamma))] \rightarrow \mathbb{C}[\mathcal{X}_{\SL_2}(\mathbf{\Sigma}^0(\Gamma,c))] \xrightarrow{\Delta^{\mathcal{G}} - \id \otimes \epsilon} 
\mathbb{C}[\mathcal{X}_{\SL_2}(\mathbf{\Sigma}^0(\Gamma,c))] \otimes \mathbb{C}[\mathcal{G}_{\mathbb{V}}].
\end{equation}

Note the analogy with the exact sequence \eqref{eq_exact_sequence_graph}. The main achievement in the work of Fock Rosly in \cite{FockRosly} is the construction of  Poisson structures on $\mathbb{C}[\mathcal{X}_{\SL_2}(\mathbf{\Sigma}^0(\Gamma,c))] = \mathbb{C}[\SL_2]^{\otimes \mathcal{E}(\Gamma)}$ and $\mathbb{C}[\mathcal{G}_{\mathbb{V}}] = \mathbb{C}[\SL_2]^{\otimes \mathring{V}(\Gamma)}$ such that the coaction $\Delta^{\mathcal{G}}$ is a Poisson morphism. Therefore, using the exact sequence \eqref{eq_gauge_exactseq2}, the affine variety $\mathcal{X}_{\SL_2}(\mathbf{\Sigma}(\Gamma))$ received a (quotient) Poisson structure. A great deal then is to show that this Poisson structure only depends on the surface $\Sigma_{\mathcal{P}}(\Gamma)$ and not on $(\Gamma,c)$. This strategy permitted the authors of \cite{FockRosly} to extend the Atiyah-Bott-Goldman Poisson structure from unpunctured closed surfaces to closed general punctured surfaces (see also \cite{KojuTriangularCharVar} for a general treatment in the language of punctured surfaces rather than ciliated graphs and using groupoid cohomology). 

\vspace{2mm}
\par Let us conclude this subsection by the following observation. It is well known that the (stated) skein algebra $\mathcal{S}_{+1}(\mathbf{\Sigma})$ is isomorphic (though non canonically) to the algebra $\mathbb{C}[\mathcal{X}_{\SL_2}(\mathbf{\Sigma})]$ of regular functions of the (relative) character variety. For closed punctured surfaces, this was shown by Bullock \cite{Bullock} under the assumption that $\mathcal{S}_{+1}(\mathbf{\Sigma})$ is reduced; this assumption was proved in \cite{PS00} (see also \cite{ChaMa} for an alternative proof). For open punctured surface, this was proved independently in \cite[Theorem 1.3]{KojuQuesneyClassicalShadows} and \cite[Theorem 8.12]{CostantinoLe19} using triangulations of surfaces. Let us note that Theorem \ref{theorem1} gives a straightforward alternative proof of this result. 

\begin{theorem}[\cite{Bullock, PS00, KojuQuesneyClassicalShadows, CostantinoLe19}]
The algebras  $\mathcal{S}_{+1}(\mathbf{\Sigma})$ (where $\mathds{k}=\mathbb{C}$) and $\mathbb{C}[\mathcal{X}_{\SL_2}(\mathbf{\Sigma})]$ are isomorphic.
\end{theorem}

\begin{proof}[Alternative proof using Theorem \ref{theorem1} with the additional assumption that $\mathcal{P}\neq \emptyset$]
First suppose that $\mathbf{\Sigma}$ is an open connected punctured surface, let $\mathbb{V}$ be such that each of its vertices are on the boundary (so the representation and relative character varieties are the same) and let $\mathbb{P}=(\mathbb{G}, \mathbb{RL})$ be a finite presentation of $\Pi_1(\Sigma_{\mathcal{P}}, \mathbb{V})$ whose generators are either of type $a$ or $d$ and fix a spin function $w$. By Equation \eqref{cc}, the algebra $\mathbb{C}[\mathcal{X}_{\SL_2}(\mathbf{\Sigma})]$ is presented by the generators $x_{ij}^{(\alpha)}$ for $\alpha \in \mathbb{G}$ and $i,j \in \{-, +\}$ with the following relations, where we set $X(\alpha):= \begin{pmatrix}x_{++}^{(\alpha)} & x_{+-}^{(\alpha)} \\ x_{-+}^{(\alpha)} & x_{--}^{(\alpha)} \end{pmatrix}$:
\begin{itemize}
\item[(i)] the exchange relations $x_{ij}^{(\alpha)} x_{kl}^{(\beta)} = x_{kl}^{(\beta)} x_{ij}^{(\alpha)}$ for all $\alpha, \beta\in \mathbb{G}$, $i,j \in \{-,+\}$; 
\item[(ii)] the determinant relations $\mathrm{det}(X(\alpha))=1$, for all $\alpha \in \mathbb{G}$; 
\item[(iii)] the trivial loops relations $X(\beta_k)\ldots X(\beta_1)=\mathds{1}_2$, for $R=\beta_k \star \ldots \star \beta_1 \in \mathbb{RL}$.
\end{itemize}

By comparing this presentation of $\mathbb{C}[\mathcal{X}_{\SL_2}(\mathbf{\Sigma})]$  with the presentation of $\mathcal{S}_{\omega}(\mathbf{\Sigma})$ obtained in Proposition \ref{prop_presU} by setting $\omega=+1$, we see that one has an isomorphism of algebras $\Theta : \mathcal{S}_{+1}(\mathbf{\Sigma}) \xrightarrow{\cong} \mathbb{C}[\mathcal{X}_{\SL_2}(\mathbf{\Sigma})]$ sending $U(\alpha)$ to $X(\alpha)$ (note that when $\omega=+1$, $\mathscr{R}=\tau$ so all arc exchange relations become $U(\alpha)\odot U(\beta) = \tau U(\alpha)\odot U(\beta)  \tau$ giving relations $\alpha_{ij} \beta_{kl}= \beta_{kl}\alpha_{ij}$).
Moreover, by comparing Equations \eqref{eq_qcoaction3} and \eqref{eq_qcoaction2}, we see that $\Theta$ is equivariant for the gauge group coactions. 

\vspace{2mm}
\par Now suppose that $\mathbf{\Sigma}$ is closed and connected with $\mathcal{P}\neq \emptyset$ and let $(\Gamma, c)$ be a ciliated fat graph such that $\mathbf{\Sigma}(\Gamma)=\mathbf{\Sigma}$. By the preceding case, one has an equivariant isomorphism $\Theta :  \mathcal{S}_{+1}(\mathbf{\Sigma}^0(\Gamma,c)) \xrightarrow{\cong} \mathbb{C}[\mathcal{X}_{\SL_2}(\mathbf{\Sigma}^0(\Gamma,c))]$, so one has a commutative diagram 
$$
\begin{tikzcd}
0 \arrow[r] & \mathcal{S}_{+1}(\mathbf{\Sigma}(\Gamma))] \arrow[r] \arrow[d, dotted, "\exists!", "\cong"'] & 
\mathcal{S}_{+1}(\mathbf{\Sigma}^0(\Gamma,c))] 
\arrow[r, "\Delta^{\mathcal{G}} - \id \otimes \epsilon"] \arrow[d, "\Theta", "\cong"']&
\mathcal{S}_{+1}(\mathbf{\Sigma}^0(\Gamma,c))] \otimes \mathbb{C}[\mathcal{G}_{\mathbb{V}}] \arrow[d, "\Theta \otimes \id", "\cong"']
\\
0 \arrow[r] &  \mathbb{C}[\mathcal{X}_{\SL_2}(\mathbf{\Sigma}(\Gamma))] \arrow[r] &  \mathbb{C}[\mathcal{X}_{\SL_2}(\mathbf{\Sigma}^0(\Gamma,c))] 
\arrow[r, "\Delta^{\mathcal{G}} - \id \otimes \epsilon"] &
\mathbb{C}[\mathcal{X}_{\SL_2}(\mathbf{\Sigma}^0(\Gamma,c))] \otimes \mathbb{C}[\mathcal{G}_{\mathbb{V}}].
\end{tikzcd}
$$
The fact that both lines are exact implies the existence of an isomorphism $\mathcal{S}_{+1}(\mathbf{\Sigma}(\Gamma)) \xrightarrow{\cong} \mathbb{C}[\mathcal{X}_{\SL_2}(\mathbf{\Sigma}(\Gamma))]$ obtained by restriction of $\Theta$.

\end{proof}

\subsection{Combinatorial quantizations of (relative) character varieties}

The work of Fock and Rosly suggests a natural way of quantizing character varieties. The following problem was raised and solved independently by Alekseev-Grosse-Schomerus \cite{AlekseevGrosseSchomerus_LatticeCS1,AlekseevGrosseSchomerus_LatticeCS2} and Buffenoir-Roche \cite{BuffenoirRoche} (see also \cite{BullockFrohmanKania_Survey_LGFT} for a survey): 

\begin{problem}\label{problem_quantization}
Associate to each oriented ciliated graph $(\Gamma,c)$ an (associative unital) algebra $\mathcal{L}_{\omega}(\Gamma,c)$ over the ring $\mathds{k}:=\mathbb{C}[\omega^{\pm 1}]$ satisfying the following properties:
\begin{itemize}
\item[(A1)] As a $\mathds{k}$-module, $\mathcal{L}_{\omega}(\Gamma,c)$ is just the (free) module $\mathbb{C}[\mathcal{R}_{\SL_2}(\mathbf{\Sigma}^0(\Gamma,c))] \otimes_{\mathbb{C}} \mathds{k}\cong \mathbb{C}[\SL_2]^{\otimes \mathcal{E}(\Gamma)} \otimes_{\mathbb{C}} \mathds{k}$, 
\item[(A2)] As before, write $\mathcal{O}_q[\mathcal{G}]:=  \mathcal{O}_q[\SL_2]^{\otimes V(\Gamma)}$.
The linear map $\Delta^{\mathcal{G}} : \mathcal{L}_{\omega}(\Gamma,c) \rightarrow \mathcal{L}_{\omega}(\Gamma,c) \otimes \mathcal{O}_q[\mathcal{G}]$ defined by the formulas

\begin{equation*}
\Delta^{\mathcal{G}} \left( x_{ij}^{(\alpha)}\right) = \sum_{a, b =\pm} {x}_{ab}^{(\alpha)} \otimes S(x_{bj})^{(v_2)} {x}_{ia}^{(v_1)}.
\end{equation*}

is a Hopf-comodule map. In particular, it is a morphism of algebras.
\item[(Inv)] The subalgebra $\mathcal{L}^{inv}_{\omega}(\Gamma) \subset \mathcal{L}_{\omega}(\Gamma,c)$ defined by the exact sequence
$$ 0 \rightarrow  \mathcal{L}^{inv}_{\omega}(\Gamma) \rightarrow  \mathcal{L}_{\omega}(\Gamma,c) \xrightarrow{\Delta^{\mathcal{G}}-\id\otimes \epsilon} \mathcal{L}_{\omega}(\Gamma,c) \otimes \mathcal{O}_q[\mathcal{G}],$$
 only depends (up to canonical isomorphism) on the (homeomorphism class of) surface $S(\Gamma)$.
\item[(Q)] 
Let $\mathds{k}_{\hbar}:= \mathbb{C}[[\hbar]]$ and write $\omega_{\hbar}:= \exp(-\frac{i\pi}{2\hbar}) \in \mathds{k}_{\hbar}$ so that $\mu: \mathds{k}\rightarrow \mathds{k}_{\hbar}$ defined by $\mu(\omega):=\omega_{\hbar}$ is a ring morphism. Then the $\mathds{k}_{\hbar}$ algebra  $\mathcal{L}^{inv}_{\omega}(\Gamma)\otimes_{\mu} \mathds{k}_{\hbar}$  is a deformation quantization of the Poisson algebra  $\mathbb{C}[\mathcal{X}_{\SL_2}(\mathbf{\Sigma}(\Gamma))]$ equipped with its Fock-Rosly Poisson structure.
\end{itemize}
\end{problem}

\begin{theorem}[Alekseev-Grosse-Schomerus (\cite{AlekseevGrosseSchomerus_LatticeCS1,AlekseevGrosseSchomerus_LatticeCS2, AlekseevSchomerus_RepCS}), Buffenoir-Roche (\cite{BuffenoirRoche, BuffenoirRoche2})]\label{theorem_LGFT}
Problem \ref{problem_quantization} admits the solution $\mathcal{L}_{\omega}(\Gamma, c) := \mathcal{L}_{\omega}(\mathbf{\Sigma}^0(\Gamma,c))$, where the $\mathds{k}$-module isomorphism $\mathcal{L}_{\omega}(\Gamma, c)  \cong \mathbb{C}[\mathcal{R}_{\SL_2}(\mathbf{\Sigma}^0(\Gamma,c))] \otimes_{\mathbb{C}} \mathds{k}$ is given by sending $U(\alpha)$ to $X(\alpha)$. 
\end{theorem}
\par The algebras  $\mathcal{L}_{\omega}(\Gamma, c)$ are the so-called \textit{quantum moduli algebras} and Theorem \ref{theorem3} is an obvious consequence of Theorem \ref{theorem1}.

More precisely, the ciliated graphs considered in \cite{BuffenoirRoche, BuffenoirRoche2} are those whose underlying graph is the $1$-skeleton of some combinatorial triangulation of a Riemann surface. By combinatorial we mean that each edges has two distinct endpoints, so every arcs is of type a and the only arcs exchange relations among distinct arcs are in configurations $(i)$ or $(ii)$ (in the notations of Lemma \ref{lemma_arcsrelations}). In \cite{AlekseevGrosseSchomerus_LatticeCS1,AlekseevGrosseSchomerus_LatticeCS2, AlekseevSchomerus_RepCS} general ciliated graphs are considered, though in \cite{AlekseevGrosseSchomerus_LatticeCS2, AlekseevSchomerus_RepCS} a special attention is given to the quantum moduli algebras of the daisy graphs defined in Example \ref{exemple_pres} (they are called \textit{standard graphs} in \cite{AlekseevGrosseSchomerus_LatticeCS2, AlekseevSchomerus_RepCS}) and have been further studied and related to stated skein algebras in \cite{Faitg_LGFT_SSkein}. In those daisy graphs, the arcs are of type $d$ and the more complicated arcs exchange relations in configurations $(viii), (ix), (x)$ appeared under the name of braid relations (see \cite[Definition $12$]{AlekseevGrosseSchomerus_LatticeCS2}).
\vspace{2mm}
\par
Note that, except for the study of the Poisson structure (which could have been easily done), we did reproved Theorem \ref{theorem_LGFT} in this paper.  In \cite{MeusburgerWise_LGFT}, Meusburger and Wise proved that the solution of Problem \ref{problem_quantization} is unique, provided that we add some natural axioms for the operation of gluing graphs together. Actually the authors of \cite{MeusburgerWise_LGFT} consider quantum moduli algebras associated to finite dimensional ribbon algebras, whereas here we consider the infinite dimensional one $U_q\mathfrak{sl}_2$, but their proof extends word-by-word to our context.

\subsection{Comparison with previous works}\label{sec_comparaison}

Let $\mathbf{\Sigma}^0$ be a connected punctured surface with one boundary component, one puncture on its boundary and eventually some inner punctures. Let $(\Gamma,c)$ be its daisy graph and $\mathbb{P}=(\mathbb{G}, \emptyset)$ be the associated finite presentation as defined in Example \ref{exemple_pres} (so $\mathbf{\Sigma}^0=\mathbf{\Sigma}^0(\Gamma,c)$). In this case, since the presentation has no relation, one can consider the spin function $w$ sending every generator to $0 \in \mathbb{Z}/2\mathbb{Z}$. Since every generator $\alpha \in \mathbb{G}$ is of type $d$, the isomorphism $\Psi : \mathcal{S}_{\omega}(\mathbf{\Sigma}^0) \xrightarrow{\cong} \mathcal{L}_{\omega}(\Gamma,c)$ sends $U(\alpha)=C^{-1}M(\alpha)$ to $X(\alpha)$. By precomposing with the reflection anti-involution $\theta$, one obtains an isomorphism 
$$\Psi': \mathcal{S}_{\omega^{-1}}(\mathbf{\Sigma}^0)^{op} \xrightarrow{\cong} \mathcal{L}_{\omega}(\Gamma,c),$$ which corresponds to Faitg's isomorphism in \cite{Faitg_LGFT_SSkein}. Let us stress that our notations are quite different from the ones in \cite{Faitg_LGFT_SSkein}; in particular:
\begin{itemize}
\item the letter $q$ in \cite{Faitg_LGFT_SSkein} is what we denoted by $A$ (so our $q$ corresponds to $q^2$ in \cite{Faitg_LGFT_SSkein}), 
\item the letter $\mathcal{R}$ in  \cite{Faitg_LGFT_SSkein}  is related to our $\mathscr{R}$ by $\mathscr{R}= \tau \circ \mathcal{R}$, 
\item Faitg actually considered $\mathcal{S}_{\omega^{-1}}(\mathbf{\Sigma}^0)^{op} $, the opposite of the stated skein algebra.
\end{itemize}

\vspace{2mm}
\par As Faitg kindly explained to the author, the existence of an isomorphism $\Psi : \mathcal{S}_{\omega}(\mathbf{\Sigma}^0) \xrightarrow{\cong} \mathcal{L}_{\omega}(\Gamma,c)$ could have been derived from the works in \cite{BenzviBrochierJordan_FactAlg1, GunninghamJordanSafranov_FinitenessConjecture} as we now briefly explain using the notations in \cite{GunninghamJordanSafranov_FinitenessConjecture} to which we refer for further details. Set $\mathds{k}=\mathbf{C}[\omega^{\pm 1}]$ and fix a structure of Riemann surface $\Sigma$. To any $\mathds{k}$-ribbon category $\mathcal{A}$, one can associate a skein category $\mathrm{SkCat}_{\mathcal{A}}(\Sigma)$ whose objects are oriented embeddings of finitely many disjoint discs $\mathbb{D}\rightarrow \Sigma$ colored by objects in $\mathcal{A}$ and morphisms are framed $\mathcal{A}$-colored ribbon graphs in $\Sigma\times [0,1]$ considered up to skein relations (see \cite[Section $4.2$]{Cooke_FactorisationHomSkein} for a precise definition). We denote by $\mathbf{1} \in \mathrm{SkCat}_{\mathcal{A}}(\Sigma^0)$ the empty set. Let $\Sigma^0$ be obtained from a connected closed oriented surface $\Sigma$ by removing an open disc. Fixing an arbitrary disc embedding $\mathbb{D}\rightarrow \Sigma^0$, one get a functor $\mathcal{P} : \mathcal{A} \rightarrow  \mathrm{SkCat}_{\mathcal{A}}(\Sigma^0)$ in an obvious way. Let $\widehat{\mathcal{A}}:= \mathrm{Fun}(\mathcal{A}^{op}, \mathrm{Vect})$ be the free cocompletion of $\mathcal{A}$ (which inherits a monoidal structure from $\mathcal{A}$). The \textit{internal skein algebra} is defined as the coend: 
$$  \mathrm{SkCat}^{int}_{\mathcal{A}}(\Sigma^0) := \int^{x \in \mathcal{A}} \mathrm{Hom}_{\mathrm{SkCat}_{\mathcal{A}}(\Sigma^0)} (\mathcal{P}(x), \mathbf{1})\otimes x \in \widehat{\mathcal{A}}.$$
The functor $\mathrm{Hom}_{\mathrm{SkCat}_{\mathcal{A}}(\Sigma^0)} (\mathcal{P}(\bullet), \mathbf{1}) : \mathcal{A}^{op} \rightarrow \mathrm{Vect}$ has a natural Lax monoidal structure, given by stacking ribbon graphs on top of each others, which endows $\mathrm{SkCat}^{int}_{\mathcal{A}}(\Sigma^0)$ with the structure of an algebra object in $\widehat{\mathcal{A}}$. If $\mathcal{A}$ is Tannakian, that is if is equipped with a fully faithful monoidal functor $\mathrm{for} : \mathcal{A}\rightarrow \mathrm{Vect}$, then 
$$ \mathcal{S}_{\mathcal{A}}(\Sigma^0) := \mathrm{for}(\mathrm{SkCat}^{int}_{\mathcal{A}}(\Sigma^0)) = \int^{x \in \mathcal{A}} \mathrm{Hom}_{\mathrm{SkCat}_{\mathcal{A}}(\Sigma^0)} (\mathcal{P}(x), \mathbf{1})\otimes \mathrm{for}(x) \in \mathrm{Vect}$$
is a unital associative algebra, that we might call the \textit{stated skein algebra} associated to $\mathcal{A}$ and $\Sigma^0$. Let us consider two Tannakian ribbon categories: the (Cauchy closure of the) Temperley Lieb category $\mathrm{TL}$ and the category of finite dimensional $U_q\mathfrak{sl}_2$ left modules $\mathrm{Rep}^{fd}_q(\SL_2)$ (recall that $q$ is generic here).  The Tannakian structure $\mathrm{forget} : \mathrm{Rep}^{fd}_q(\SL_2) \rightarrow \mathrm{Vect}$ is just the forgetful functor. It is well known that one has a monoidal braided equivalence of categories $G : \mathrm{TL} \rightarrow \mathrm{Rep}^{fd}_q(\SL_2)$ sending the one strand ribbon $[1]\in \mathrm{TL}$ to the fundamental representations $V$ of Section \ref{sec_def_basic} with basis $\{v_+, v_-\}$, thus we get a Tannakian structure $\mathrm{forget} \circ G : \mathrm{TL} \rightarrow \mathrm{Vect}$. Let us underline that $G$ \textbf{does not preserve the duality} ! This fact is what will make the present argument ambiguous.

On the one hand,  there is a natural algebra morphism 
$$ \Psi_1 : \mathcal{S}_{\omega}(\mathbf{\Sigma}^0) \rightarrow \mathcal{S}_{TL}(\Sigma^0)$$ 
sending the class $[T,s]$ of a stated tangle, where $\partial T$ has $n$ elements,  to the class of 
\\ $T\otimes v_s \in \mathrm{Hom}_{\mathrm{SkCat}_{\mathrm{TL}}(\Sigma^0)}  (\mathcal{P}([1]^{\otimes n}), \mathbf{1})\otimes V^{\otimes n}$, where $v_s \in V^{\otimes n}$ is obtained from the state $s$ by identifying the signs $+$ and $-$ with the basis vectors $v_+$ and $v_-$ of $V$. As noted in Remark \cite[Remark $2.21$]{GunninghamJordanSafranov_FinitenessConjecture}, a detailed comparison of the definitions shows that $\Psi_1$ is an isomorphism.

On the other hand, thanks to Cooke's excision theorem in \cite{Cooke_FactorisationHomSkein}, and as proved in \cite[Proposition $2.19$]{GunninghamJordanSafranov_FinitenessConjecture}, the internal skein algebra $  \mathrm{SkCat}^{int}_{\mathcal{A}}(\Sigma^0)$ is isomorphic to the so-called moduli algebra $\mathcal{A}_{\Sigma^0}=\underline{\mathrm{End}}(\mathbf{1}) \in \widehat{A}$ introduced in \cite[Definition $5.3$]{BenzviBrochierJordan_FactAlg1}. In \cite[Theorem $5.14$]{BenzviBrochierJordan_FactAlg1}, the authors defined an explicit isomorphism $\left[\mathrm{Rep}^{fd}_q(\SL_2)\right]_{\Sigma^0} \cong \mathcal{L}_{\omega}(\Gamma)$, so by composing the two isomorphisms, one get an isomorphism
$$ \Psi_2 : \mathcal{S}_{\mathrm{Rep}^{fd}_q(\SL_2)}(\Sigma^0) \xrightarrow{\cong} \mathcal{L}_{\omega}(\Gamma).$$

 \par Now, even though the equivalence of categories $G : \mathrm{TL} \rightarrow \mathrm{Rep}^{fd}_q(\SL_2)$ does not preserve the duality, one should be able to get a (non-canonical) algebra isomorphism $\Psi_3: \mathcal{S}_{\mathrm{Rep}^{fd}_q(\SL_2)}(\Sigma^0) \xrightarrow{\cong} \mathcal{S}_{\mathrm{TL}}(\Sigma^0)$. So by composing the isomorphisms $\Psi_1, \Psi_2$ and $\Psi_3$, one would get an isomorphism $\mathcal{S}_{\omega}(\mathbf{\Sigma}^0) \xrightarrow{\cong} \mathcal{L}_{\omega}(\Gamma,c)$, which would recover Faitg result and is a particular case of our Theorem \ref{theorem3}. Note that the definition of $\Psi_3$ is not very clear and certainly not canonical: it is in this definition that the non-canonical choice of a spin function should appear. 
 
 \begin{remark}
 The above construction of the isomorphism $\mathcal{S}_{\omega}(\mathbf{\Sigma}^0) \xrightarrow{\cong} \mathcal{L}_{\omega}(\Gamma,c)$ is very indirect and not so enlightening. However, it generalises the notion of stated skein algebra $\mathcal{S}_{\mathcal{C}}(\Sigma^0)$ to an arbitrary Tannakian ribbon category $\mathcal{C}$ (how to replace $\mathbf{\Sigma}^0$ to an arbitrary punctured surface is obvious), and \cite[Theorem $5.14$]{BenzviBrochierJordan_FactAlg1} seems to permit to give explicit finite presentations for $\mathcal{S}_{\mathcal{C}}(\Sigma^0)$. The detailed study of these generalized stated skein algebras will appear in a separate publication \cite{KojuTannakianSSkein_ToAppear}.
 \end{remark}

\section{Concluding remarks}\label{sec_final}

We conclude the paper by making some remarks concerning the usefulness of relating stated skein algebras and quantum moduli spaces (Theorem \ref{theorem3}). We can see the stated skein algebras as defined by a huge set of generators (all stated tangles) and a huge set of relations (isotopy and skein relations) whereas the quantum moduli algebra is defined by a finite subset of generators and by a finite subset of relations. Both presentations have their own advantages.

\vspace{2mm}
\par $(1)$ The fact that the quantum moduli algebra $\mathcal{L}^{inv}_{\omega}(\Gamma)$ only depends, up to canonical isomorphism, on the thickened surface $S(\Gamma)$ (or equivalently $\mathbf{\Sigma}(\Gamma)$) is usually proved by defining elementary moves on graphs that preserve the thickened surface and showing that those elementary moves induce isomorphism on the algebras. This strategy was pioneered by Fock and Rosly in the classical case of relative character varieties \cite{FockRosly} and latter carried on in \cite{AlekseevGrosseSchomerus_LatticeCS2, BuffenoirRoche2}  for quantum moduli algebras (see also \cite{MeusburgerWise_LGFT} for very detailed study). Thanks to the isomorphism $\mathcal{L}^{inv}_{\omega}(\Gamma)\cong \mathcal{S}_{\omega}(\mathbf{\Sigma}(\Gamma))$ (and the fact that stated skein algebra depends on surfaces rather than graphs),  this fact is also an immediate consequence of Theorem \ref{theorem3}. Also the image of a closed curve $\gamma$ through the reverse isomorphism $\Psi^{-1} : \mathbf{\Sigma}(\Gamma) \rightarrow \mathcal{L}^{inv}_{\omega}(\Gamma)$ is usually called its \textit{holonomy} $\mathrm{Hol}(\gamma)$ or \textit{Wilson loop operators} and the expression of this holonomy in terms of generators and proof of some composition properties is the subject of long and technical computations in \cite{AlekseevGrosseSchomerus_LatticeCS1,AlekseevGrosseSchomerus_LatticeCS2,BuffenoirRoche, BuffenoirRoche2,  MeusburgerWise_LGFT, Faitg_LGFT_MCG}, whereas they become easy in the skein algebra setting.
\vspace{2mm}
\par $(2)$ Since the quantum moduli algebra $\mathcal{L}_{\omega}(\Gamma,c)$ is quadratic homogeneous, we might have tried to prove that it is Koszul (so prove that $\underline{\mathcal{B}}^{\mathbb{G}}$ is free) without the help of the stated skein algebra. The standard technique to prove that the family $\mathcal{B}$ of Equation \eqref{def_PBW_basis} is a PBW basis consists in examining the set of critical monomials of the form $v_iv_jv_k$ (we use the notations of Section \ref{sec_Koszul}) where both $v_iv_j$ and $v_jv_k$ are leading terms. To such a critical monomial, we associate a finite graph (which might have the shape of a Diamond) and the Diamond Lemma implies that if each of these graph is confluent (has a terminal object) then $\mathcal{B}$ is a basis so the quadratic algebra is Koszul (see \cite[Section 4]{LodayValletteOperads} for details). In our case, due to the huge amount of different kind of relations in our presentation, this strategy would require to verify the confluence of $6578$ different graphs ! This is way too much to be handled by hand. It is thanks to the fact that stated skein algebras have a lot of relations and generators that L\^e was able to successfully use the Diamond Lemma in \cite{LeStatedSkein} to prove that $\mathcal{B}$ is basis, and our proof of the fact that  $\underline{\mathcal{B}}^{\mathbb{G}}$ is a basis is directly derived from this fact. So proving the Koszulness of $\mathcal{L}_{\omega}(\Gamma,c)$ without the help of stated skein algebra could have been a very difficult problem.
\vspace{2mm}
\par $(3)$ Even if we could find PBW bases for the algebras  $\mathcal{L}_{\omega}(\Gamma,c)$ without the help of skein algebras, finding bases for $\mathcal{L}^{inv}_{\omega}(\Gamma)$ would be extremely difficult, since it is only defined as a kernel and no presentation was known. However, skein algebras $\mathcal{S}_{\omega}(\mathbf{\Sigma}(\Gamma)) \cong \mathcal{L}^{inv}_{\omega}(\Gamma)$ have well known bases (of multi-curves).
\vspace{2mm}
\par $(4)$ As we saw in Section \ref{sec_charvar}, the fact that $\mathcal{L}_{+1}(\mathbf{\Sigma}, \mathbb{P})$ is isomorphic to the algebra of regular functions of the (relative) character variety $\mathcal{X}_{\SL_2}(\mathbf{\Sigma})$ is very easy to prove, whereas relating the (stated) skein algebra $\mathcal{S}_{+1}(\mathbf{\Sigma})$ to $\mathbb{C}[\mathcal{X}_{\SL_2}(\mathbf{\Sigma})]$ is not so obvious (see \cite{Bullock, PS00} for closed surfaces and \cite{KojuQuesneyClassicalShadows, CostantinoLe19} for open ones). 
\vspace{2mm}
\par $(5)$ In \cite{BonahonWong1}, Bonahon and Wong proved that the Kauffman-bracket skein algebra $\mathcal{S}_{+1}(\mathbf{\Sigma})$, with deforming parameter $+1$, embeds into the center of the skein algebra $\mathcal{S}_{\zeta}(\mathbf{\Sigma})$ with deforming parameter $\zeta$ a root of unity of odd order (see also \cite{Le_QTraceMuller} for an alternative proof). This result was generalized in \cite{KojuQuesneyClassicalShadows} to  stated skein algebras as well (see also the forthcoming paper \cite{BloomquistLe} for generalizations). 
In  \cite{BaseilhacRoche_LGFT1}, Baseilhac and Roche showed that the construction of this so-called Chebyshev-Frobenius morphism is much easier in the context of quantum moduli algebras (that is using the finite presentations of Theorem \ref{theorem1}). Even though their study only concerns genus $0$ surfaces, their proofs seems to generalize easily to general surfaces, providing simpler proofs for  the results in \cite{BonahonWong1, KojuQuesneyClassicalShadows}. 
\vspace{2mm}
\par $(6)$ Bullock, Frohman and Kania-Bartoszynska already proved in \cite[Theorem 10]{BullockFrohmanKania_LGFT} that $\mathcal{L}_{\omega}^{inv}(\Gamma)$ and $\mathcal{S}_{\omega}(\mathbf{\Sigma}(\Gamma))$ are isomorphic in the particular case where $\mathds{k}=\mathbb{C}[[\hbar]]$ and $\omega= -\exp(-\hbar/4)$. Their proof consists in $(1)$ defining an algebra morphism 
$\Psi : \mathcal{L}_{\omega}^{inv}(\Gamma)\rightarrow \mathcal{S}_{\omega}(\mathbf{\Sigma}(\Gamma))$ (by techniques similar to what we did in Section \ref{sec_presentation}) and $(2)$ noting that under the $\pmod{\hbar}$ identifications $\quotient{\mathcal{L}_{\omega}^{inv}(\Gamma)}{(\hbar)}\cong \mathbb{C}[\mathcal{X}_{\SL_2}(\mathbf{\Sigma})]$ and $\quotient{\mathcal{S}_{\omega}(\mathbf{\Sigma}(\Gamma))}{(\hbar)} \cong \mathcal{S}_{-1}(\mathbf{\Sigma}(\Gamma))$, the morphism $\Psi$ reduces modulo $\hbar$ to Bullock's isomorphism $\mathbb{C}[\mathcal{X}_{\SL_2}(\mathbf{\Sigma})] \cong \mathcal{S}_{-1}(\mathbf{\Sigma}(\Gamma))$. So the fact that the reduction of $\Psi$ modulo $\hbar$ is an isomorphism implies that $\Psi$ is an isomorphism. This proof does not seem (at least to the author) to generalize straightforwardly to prove the identification $\mathcal{L}_{\omega}^{inv}(\Gamma)\cong \mathcal{S}_{\omega}(\mathbf{\Sigma}(\Gamma))$ for more general rings (such as $\mathds{k}=\mathbb{C}$ and $\omega$ a root of unity) whereas our Theorem \ref{theorem3} works in full generality.

\bibliographystyle{amsalpha}
\bibliography{biblio}

\end{document}